\documentclass[11pt,reqno]{amsart}

 \usepackage{mathrsfs}
\newcommand{\scrL }{\mathscr{L}}

\newcommand{\scrP }{\mathscr{P}}
\newcommand{\scrS }{\mathscr{S}}

\usepackage{color}
\usepackage{amsmath, amsthm, amssymb}
\usepackage{amsfonts}
\usepackage[ansinew]{inputenc}
\usepackage{graphicx}
\usepackage{caption}
\usepackage{subcaption}
\usepackage{marginnote}
\usepackage[english]{babel}
\usepackage{hyperref}

\usepackage{tikz}
\usepackage{rotating}

\usepackage{cite}
\usepackage{amscd}
\usepackage{color}
\usepackage{bm}
\usepackage{enumerate}

\usepackage{verbatim}
\usepackage{hyperref}
\usepackage{amstext}
\usepackage{latexsym}
\theoremstyle{plain}
\newtheorem{theorem}{Theorem}[section]

\newtheorem{corollary}[theorem]{Corollary}
\newtheorem{proposition}[theorem]{Proposition}
\newtheorem{lemma}[theorem]{Lemma}
\newtheorem{remark}[theorem]{Remark}

\numberwithin{theorem}{section}
\numberwithin{equation}{section}

\newcommand{\average}{{\mathchoice {\kern1ex\vcenter{\hrule height.4pt
width 6pt depth0pt} \kern-9.7pt} {\kern1ex\vcenter{\hrule
height.4pt width 4.3pt depth0pt} \kern-7pt} {} {} }}

\def\R{\mathbb{R}}
\def\loc{\text{loc}}

\def\div{\text{div}}

\renewcommand{\a }{\alpha }
\renewcommand{\b }{\beta }
\renewcommand{\d}{\delta }

\newcommand{\D }{\Delta }

\newcommand{\e }{\varepsilon }
\newcommand{\g }{\gamma}

\newcommand{\G }{\Gamma}
\renewcommand{\l }{\lambda }
\renewcommand{\L }{\Lambda }
\newcommand{\n }{\nabla }
\newcommand{\vp }{\varphi }
\renewcommand{\phi}{\varphi}

\newcommand{\s }{\sigma }

\renewcommand{\t }{\tau }
\newcommand{\z }{\zeta}
\renewcommand{\th }{\theta }

\renewcommand{\o }{\omega }

\newcommand{\ov}{\overline}

\newcommand{\be}{\begin{equation}}
\newcommand{\ee}{\end{equation}}

\newcommand{\de}{\partial}

\newcommand{\ti}{\widetilde}

\renewcommand{\k}{\kappa}

\newcommand{\N}{\mathbb{N}}

\newcommand{\Z}{\mathbb{Z}}

\newcommand{\cA}{{\mathcal A}}
\newcommand{\cB}{{\mathcal B}}
\newcommand{\cC}{{\mathcal C}}
\newcommand{\cD}{{\mathcal D}}
\newcommand{\cE}{{\mathcal E}}
\newcommand{\cF}{{\mathcal F}}

\newcommand{\cH}{{\mathcal H}}

\newcommand{\cJ}{{\mathcal J}}
\newcommand{\cK}{{\mathcal K}}
\newcommand{\cL}{{\mathcal L}}
\newcommand{\cM}{{\mathcal M}}
\newcommand{\cN}{{\mathcal N}}
\newcommand{\cO}{{\mathcal O}}

\newcommand{\cS}{{\mathcal S}}

\newcommand{\cU}{{\mathcal U}}
\newcommand{\cV}{{\mathcal V}}

\newcommand{\eps}{\varepsilon}

\DeclareMathOperator{\id}{id}

\renewcommand{\epsilon}{\varepsilon}

\begin{document}
\author{Xavier Cabr\'e}
\address{X. Cabr\'e\textsuperscript{1,2}:
\newline
\textsuperscript{1} Universitat Polit\`ecnica de Catalunya, Departament de Matem\`{a}tiques, Diagonal 647, 
08028 Barcelona, Spain
\newline
\textsuperscript{2} ICREA, Pg. Lluis Companys 23, 08010 Barcelona, Spain}
\email{xavier.cabre@upc.edu}

\author[Mouhamed M. Fall]
{Mouhamed Moustapha Fall}
\address{M.M. Fall: African Institute for Mathematical Sciences of Senegal, 
KM 2, Route de Joal, B.P. 14 18. Mbour, S\'en\'egal}
\email{mouhamed.m.fall@aims-senegal.org}

\author{Tobias Weth}
\address{T. Weth: Goethe-Universit\"{a}t Frankfurt, Institut f\"{u}r Mathematik.
Robert-Mayer-Str. 10 D-60054 Frankfurt, Germany}
\email{weth@math.uni-frankfurt.de}

\thanks{The first author is supported by MINECO grant MTM2014-52402-C3-1-P. 
He is part of the Catalan research group 2014 SGR 1083. The second author's work is supported by the Alexander 
von Humboldt foundation. Part of this work was done while he was visiting
the University of Frankfurt, during July and August, 2015 and 2016. The third author is supported by DAAD (Germany) within the program 57060778.}

\title[Near-Sphere lattices with constant nonlocal mean curvature]
{Near-Sphere lattices with constant nonlocal mean curvature}

\begin{abstract}
We are concerned with unbounded sets of $\mathbb{R}^N$ whose boundary has constant nonlocal (or fractional) 
mean curvature, which we call CNMC sets. This is the equation associated to critical points of the fractional 
perimeter functional under a volume constraint. We construct CNMC sets which are the countable union of a certain 
bounded domain and all its translations through a periodic integer lattice of dimension $M\leq N$.
Our CNMC sets form a $C^2$ branch emanating  from the unit ball alone and where the parameter in the branch is
essentially the distance to the closest lattice point. Thus, the new translated near-balls (or near-spheres)
appear from infinity. We find their exact asymptotic shape as the parameter tends to infinity. 
\end{abstract}

\maketitle

\section{Introduction}
\label{sec:introduction}
Let $\alpha \in (0,1)$. If $\cA$ is a smooth oriented hypersurface in $\R^N$ with unit normal vector field $\nu$, 
its nonlocal or fractional mean curvature (abbreviated NMC in the following) of order $\alpha$ at a point $x \in \cA$
is defined as  
\begin{equation}
  \label{eq:def-frac-curvature}
H(\cA; x)=\frac{2d_{N,\a}}{\a} \int_{\cA} \frac{(y-x)\cdot \nu (y)}{|y-x|^{N+\alpha}}\,dV(y)  .  
\end{equation}
Here and in the following, $dV$ stands for the volume element on $\cA$, and 
\be \label{eq:d_N-alpha}
d_{N,\a}= \frac{1-\a}{(N-1)|B^{N-1}|}= \frac{(1-\a)\Gamma(\frac{N+1}{2})}{(N-1) \pi^{(N-1)/2}},
\ee
where $B^{N-1}$ is the unit ball in $\R^{N-1}$. If $\cA$ is of class $C^{1,\beta}$ for some $\beta>\alpha$ 
and we assume $\int_{\cA} (1+|y|)^{1-N-\alpha}\,dV(y)< \infty$, then the integral in \eqref{eq:def-frac-curvature} 
is absolutely convergent in the Lebesgue sense. 

The choice of the constant $d_{N,\alpha}$ guarantees that, if $\cA$ is of class $C^{2}$, the nonlocal mean curvature 
 $H(\cA;\cdot)$  converges,  as $\a\to 1$, locally uniformly to the classical mean curvature, i.e., 
 the arithmetic mean of principle curvatures, see \cite[Lemma A.1]{Davila2014B} and \cite[Theorem 12]{Abatangelo}.

There is  an alternative expression for   $H(\cA;\cdot)$  in terms of  a solid integral.  Suppose  that $\cA=\de E$ for some open set $E\subset\R^N$ and $\nu$ is the normal  exterior to $E$. Then, for all $x\in \cA$,  we have 
\be  \label{eq:NMC-PV}
H(\cA;x)= d_{N,\alpha} PV\int_{\R^N}\frac{1_{E^c}(y)-1_{E}(y)}{|y-x|^{N+\a}}\, dy= d_{N,\alpha} \lim_{\e\to 0 }\int_{|x-y|>\e}\frac{1_{E^c}(y)-1_{E}(y)}{|y-x|^{N+\a}}\, dy,
\ee
where $ E^c=\R^N\setminus \ov{E}$ and  $1_D$ denotes the characteristic function of a set $D \subset \R^N$. This can be derived using the divergence theorem and the fact that  $\nabla_y\cdot (y-x){|y-x|^{-N-\alpha}} =  {\alpha}{|y-x|^{-N-\alpha}}.$
 
In recent independent works, a result of Alexandrov type has been proved for the nonlocal mean curvature by 
Ciraolo, Figalli, Maggi, and Novaga~\cite[Theorem 1.1]{Ciraolo2015} and Cabr\'e, Fall, Sol\`a-Morales, 
and Weth~\cite[Theorem 1.1]{Cabre-Alex-Del-2015A}. This result states that every bounded (and a priori not necessarily 
connected) hypersurface without boundary and with constant nonlocal mean curvature must be a sphere. 
This result naturally led to questions related to the existence and shape of unbounded hypersurfaces of constant NMC. 
Obvious examples within this class are hyperplanes (which have zero NMC) and straight cylinders. 
In \cite{Cabre-Alex-Del-2015A} and \cite{CFW-2016} we proved the existence of periodic 
and connected hypersurfaces  
in $\mathbb{R}^N$ with constant NMC which bifurcate from a straight cylinder. 
These hypersurfaces should be regarded as Delaunay type cylinders in the nonlocal setting. 
We point out that, unlike in the local case, straight cylinders have positive constant NMC in every space 
dimension $N\geq 2$. Thus, our result also gave  periodic bands
in $\mathbb{R}^2$ with constant NMC and which bifurcate from a straight band. As proved in \cite{cinti-davila-del-pino}, a further unbounded hypersurface of constant NMC is the helicoid, 
which -- besides having zero mean curvature -- also has zero NMC.

Having constant nonlocal mean curvature is the equation associated to critical points of the fractional 
perimeter functional under a volume constraint. Thus, one would expect that CNMC sets can be constructed variationally. In this 
direction, the paper \cite{Davila2015} by D\'avila, del Pino, Dipierro, and Valdinoci,
established variationally the existence of periodic and cylindrically symmetric 
hypersurfaces in $\R^N$ which minimize (under the volume constraint) a certain 
fractional perimeter functional adapted to periodic sets. More precisely, \cite{Davila2015} 
established the existence of a 1-periodic minimizer for every given volume 
within the slab $\{(s,\z)\in\R\times\R^{N-1}\, : \, -1/2<s<1/2\}$. We have realized recently 
that, in fact, their fractional perimeter functional adapted to periodic sets 
gives rise to CNMC hypersurfaces in a weak sense. They would be CNMC 
hypersurfaces in the classical sense defined above if one could prove
that they are of class $C^{1,\beta}$ for some $\beta>\alpha/2$
-- which is not done in \cite{Davila2015}. The article also proves that for small volume constraints, the minimizers
tend in measure (more precisely, in the so called Fraenkel asymmetry) to a
periodic array of balls.

Note that sets obtained by minimizing a fractional perimeter functional under 
a volume constraint are expected to have Morse index 1 -- within a proper 
functional analytic framework. This will be the case for the CNMC sets 
constructed in the present paper -- see Remark~\ref{sec:introduction-1-remark}(iv). As we will see, the linearized operator at 
them (acting on a space of even functions) will have only one negative 
eigenvalue -- all the rest being positive. Note that looking at the linearized 
operator in a space of even functions excludes the eigenfunction with zero 
eigenvalue produced by the invariance of the nonlinear problem under translations.

We recall that in the case of classical mean curvature, embedded Delaunay hypersurfaces 
vary from a cylinder to an infinite compound of tangent spheres. However, it is easy to see that 
an infinite compound of aligned round spheres, tangent or disconnected, does not have constant NMC.
Indeed, it is an open problem to establish the existence of global continuous 
branches of nonlocal Delaunay cylinders and to study their limiting configurations.

In the present paper, we study nonlocal analogues of the set given by an infinite compound of aligned round spheres, 
tangent or disconnected. In a more general setting, we construct CNMC sets which are the countable union of a certain 
bounded domain and all its translations through a periodic integer lattice of dimension $M\leq N$.
Our CNMC sets form a $C^2$ branch emanating from the unit ball alone, where the parameter in the branch is
essentially the distance to the closest lattice point. Thus, the new translated near-balls (or near-spheres)
appear from infinity. We point out that it is necessary
to consider infinite lattices in this problem --  a finite disjoint union of two or more bounded sets cannot have constant NMC by the Alexandrov type rigidity result in \cite{Cabre-Alex-Del-2015A,Ciraolo2015}. 

We expect (but we do not prove) that, when the distance from two consecutive near-spheres is large 
enough, our periodic CNMC set made of near-spheres is
a minimizer of the fractional perimeter under the volume and periodicity constraints. Note that, after rescaling,
large distance from two consecutive near-spheres turns into a fixed distance (or period) but now with  a very small 
volume constraint -- as in the result of \cite{Davila2015} mentioned above.

To be precise, we now assume $N\geq 2$ and let 
$$
S:= S^{N-1}\subset \R^N
$$  
denote the unit sphere of $\R^N$.  For $M\in\N$ with $1\leq M\leq N$ we regard $\R^M$ as a subspace of $\R^N$ by 
identifying $x'\in \R^M$ with $(x',0)\in \R^M \times \R^{N-M} =  \R^N$.  Let 
$\left\{ \textbf{a}_1;\dots;\textbf{a}_M\right\}$  be a basis of~$\R^M$.   
 By the above identification, we then consider the $M$-dimensional lattice 
\be \label{eq:def-scrL-lattice}
\scrL=\left\{\sum_{i=1}^M k_i \textbf{a}_i\,:\,  k=(k_1,\dots,k_M)\in \Z^M\right\}
\ee 
as a subset of $\R^N$. In the case where     $\left\{ \textbf{a}_1;\dots;\textbf{a}_M\right\}$   is an   orthogonal or an orthonormal basis, we say that $\scrL$ is a  {rectangular lattice} or a square lattice, respectively. 

 We define, for $r>0, $  
\be
  S+ r\scrL:= \bigcup_{p \in\scrL} \Bigl( S+  r p \Bigr)  \; \subset \; \R^N.
\ee
Then, for $r> 2 (\inf_{p \in \scrL \setminus \{0\}} |p|)^{-1}$, the set $S+ r\scrL$ is the  union of disjoint unit spheres  
centered at the  lattice points in $r \scrL$. Consequently,  $S+ r\scrL$  is a set of constant classical mean curvature 
(equal to one).  In contrast, as a consequence of our main result, we shall see that
the NMC $H(S+ r\scrL;\cdot )$ is in general \textit{not} constant on this periodic set.  It is therefore natural to ask if 
the sphere $S$ can be perturbed smoothly to a set $S_\vp$,  such that   $S_\vp+ r\scrL$, for $r>0$ large enough,  
has constant NMC.

To answer this question, we fix $\b\in (\a,1)$ and define  the set 
$$
\cO:=\{\vp \in  C^{1,\b}(S) \,:\, \|\vp\|_{L^\infty(S)}<1 \}. 
$$
We then consider  the deformed sphere 
\begin{equation}\label{eq:first-phi}
S_\phi:= \{ (1+ \phi(\s)) \s \::\: \s \in S\}, \qquad \vp \in \cO.
\end{equation}
Provided that $r>0$ is large enough, the deformed sphere lattice (or near-sphere lattice)
$$
S_\vp+ r\scrL := \bigcup_{p \in \scrL} \Bigl(S_\vp+  rp \Bigr)
$$
is a noncompact hypersurface of class $C^{1,\beta}$, which by construction is periodic with respect to $r \scrL$-translations.  
 
The main result of the present paper is the following.  

\begin{theorem}\label{th:main-th-Brav-lat}
Let $\a\in (0,1)$, $\b\in \left(\a,1\right)$,  $N\geq 2$, $1\leq M\leq N$ and $\scrL$ be an $M$-dimensional lattice as 
given in \eqref{eq:def-scrL-lattice}. Then,
there exist $r_0>0$,  and a  $C^2$-curve $(r_0,+\infty) \to C^{1,\b}(S)$, $r \mapsto \phi_r$, with the following properties: 
\begin{itemize}
\item[(i)] $\phi_r \to 0$ in $C^{1,\beta}(S)$ as $r \to +\infty$;
\item[(ii)] For every $r\in (r_0,+\infty)$, the function $\vp_r:S\to \R$ is even (with respect to reflection through 
the origin of $\R^N$); 
\item[(iii)] For every $r\in (r_0,+\infty)$,  the hypersurface $S_{\vp_r}+ r\scrL $ has constant nonlocal mean curvature given by $H(S_{\vp_r}+ r\scrL;\cdot  ) \equiv   H(S;\cdot)$.
\item[(iv)] Letting $\scrL_*:= \scrL \setminus \{0\}$,  the function $\phi_r$ expands as
\begin{align*}
\phi_r(\th) 
&=r^{-N-\alpha}\left(-\kappa_0  +   r^{-2}\left\{\kappa_1 \sum_{ p\in \scrL_*} \frac{(\th \cdot p)^2}{|p|^{N+\alpha+4}} -\kappa_2\right\}
 +  o\left(r^{-2}\right) \right) \quad \text{for $\theta \in S$ as $r \to +\infty,$}
\end{align*}
with positive constants $\kappa_0$, $  \k_1$ and $\k_2$ (see Remark~\ref{sec:introduction-1-remark} below for their 
explicit values) and with $r^2 o(r^{-2}) \to 0$ in $C^{1,\beta}(S)$ as $r \to +\infty$.
\item[(v)] If $1\leq M\leq N-1$,  then the functions $\phi_r$, $r > r_0$, are non-constant on $S$.
\end{itemize}

Moreover, if $r_1>r_0$ and $(r_1,+\infty) \to C^{1,\beta}(S)$, $r \mapsto \tilde \phi_r$ is another (not necessarily continuous) curve satisfying $(i)$, $(ii)$ and $(iii)$, then $\tilde \phi_r = \phi_r$ for $r$ sufficiently large.  
\end{theorem}

The curve $r\mapsto\varphi_r$ is not $C^3$, in general. It is not $C^3$ for instance when $N=2$. This is due to 
the presence of the factor $|\tau|^{N+\alpha}=r^{-N-\a}$ in our functional equation \eqref{eq:first-funct-eq}.

To establish the theorem it will be essential to analyze the linearized operator for the NMC $H$ at the unit sphere $S$. 
We will see that the linearization is given by the operator
\begin{equation}\label{eq:linNMC}
\phi \mapsto 2d_{N,\alpha}(L_\alpha \phi-\lambda_1\phi),
\end{equation}
where
\begin{equation}
  \label{eq:def-L-alpha}
L_\a : C^{1,\beta}(S) \to C^{\beta-\alpha}(S), \qquad  L_\a\vp(\th )= 
PV \int_S\frac{\vp(\th)-\vp(\s)}{|\th -\s|^{N+\a}}\, dV(\s),
\end{equation}
$\lambda_1$ is defined next in \eqref{eq:l_k-0}, and $\varphi$ is a deformation of $S$ in the direction of 
its normal -- as in \eqref{eq:first-phi}.
The operator $L_\a$ can be seen as a spherical fractional  Laplacian, and the above integral is understood in the principle value sense, i.e., 
 $$
PV \int_S\frac{\vp(\th)-\vp(\s)}{|\th -\s|^{N+\a}}\, dV(\s):= \lim_{\eps \to 0} \int_{S \setminus B_{\eps}(\th)} \frac{\vp(\th)-\vp(\s)}{|\th -\s|^{N+\a}}\, dV(\s) \qquad \text{for $\eps \to 0$.}
$$
The operator  $L_\a$ has the spherical harmonics   as eigenfunctions corresponding to the increasing sequence $\l_0=0<\l_1<\l_2<...$ of eigenvalues   given by
\be \label{eq:l_k-0}
\l_k= \frac{\pi^{(N-1)/2}    \Gamma((1-\a)/2)    }{  (1+\a)2^\a       \Gamma((N+\a)/2) } \left(  \frac{ \Gamma\left(  \frac{2k+N+\a}{2} \right)  }{ \Gamma\left(  \frac{2k+N-\a-2}{2} \right) } -   \frac{\Gamma\left(  \frac{N+\a}{2} \right)  }{ \Gamma\left(  \frac{N-\a-2}{2} \right) }   \right),
\ee 
see \cite[Lemma 6.26]{Samko2002} and Section \ref{sec:line-nmc-oper} below. Here, as before,  $\Gamma$ is the Gamma 
function.  We shall also see, as a consequence of \eqref{eq:eq-to-solve} and \eqref{eq:def-rem-l1-0},
that the NMC of the unit sphere 
$S= S^{N-1} \subset \R^N$ is given by  
\begin{equation}\label{eq:NMC-sphere}
H(S;\cdot) \equiv \frac{2d_{N,\alpha}}{\a} \l_1 \qquad \text{on $S$.}
\end{equation}

Now that $\l_1$ and $\l_2$ have been introduced, we can give the value of the constants in  
Theorem~\ref{th:main-th-Brav-lat}(iv). In the following remark, we also comment on the size of the near-spheres
depending on the parameter $r$, as well as on their smoothness.

\begin{remark}{\rm 
\label{sec:introduction-1-remark}
(i) The constants in  Theorem~\ref{th:main-th-Brav-lat}(iv) are given by 
\begin{align*}
\kappa_0 &=  \frac{ |S|}{N \l_1 }\sum_{p\in\scrL_*}\frac{1}{|p| ^{N+\a}}, \qquad  
\kappa_1= \frac{|S|(N+\alpha ) (N+\alpha+2 )}{6 N (\l_2-\l_1)}  \qquad \text{and}\\
\kappa_2&=\frac{ |S|}{6} \left\{\frac{(N+\alpha)(N+\alpha+2) }{N^2(\lambda_2-\lambda_1)}  +  
\frac{2(N+\a)(N+1)(\alpha +2)}{N^2(N+2) \l_1} 
\right\} \sum_{p \in \scrL_*}\frac{1}{|p |^{N+\alpha+2}},
\end{align*}
where $\lambda_1,\lambda_2$ are given in \eqref{eq:l_k-0}.  
 
(ii) Since $\kappa_0>0$, the expansion in Theorem \ref{th:main-th-Brav-lat}(iv) shows
that, for large $r$, the perturbed spheres $S_{\vp_{r}}$ become smaller than $S$ 
as the perturbation parameter $r$ decreases.  
With regard to the order $r^{-N-\alpha}$, the shrinking process is uniform on $S$, whereas non-uniform deformations of the spheres may appear at the order $r^{-N-\alpha-2}$. In particular, we shall detect these non-uniform deformations in the case $M \le N-1$, and 
from this we will then deduce part (v) of Theorem~\ref{th:main-th-Brav-lat}. In the case $M=N$, it remains an open problem to characterize the lattices which give rise to non-uniform deformations. We conjecture that $H(S+r \scrL;\cdot)$ is non-constant for any $N$-dimensional lattice $\scrL$ and large $r$.
 
(iii) The smoothness (i.e., the $C^\infty$-character) of our $C^{1,\beta}$ hypersurfaces $S_{\vp_r}+ r\scrL$ , and in general of 
$C^{1,\beta}$ hypersurfaces in $\R^N$ with constant NMC which are, locally, Lipschitz graphs, follows (since $\beta >\alpha$)
from the methods and results of Barrios, Figalli, and Valdinoci~\cite{Barrios} on nonlocal minimal graphs. 
This holds for all $N\geq 2$.
More generally, to deduce the $C^\infty$ regularity, \cite{Barrios} needs to assume that the hypersurface is 
$C^{1,\beta}$ for some $\beta>\alpha/2$  and that has constant 
nonlocal mean curvature in the viscosity sense; this fact can be found in Section~3.3
of \cite{Barrios}. Here, the notion of viscosity solution is needed since the expression 
\eqref{eq:def-frac-curvature} for the NMC is only well defined
for $C^{1,\beta}$ sets when $\beta > \alpha$.

(iv) As already remarked above, the CNMC sets constructed in Theorem~\ref{th:main-th-Brav-lat} have Morse index 1 
within our functional analytic framework. More precisely, for $r> r_0$ sufficiently large, the linearization of the 
nonlocal mean curvature operator
$$
\cO \to C^{\beta-\alpha}(S),\qquad  \varphi \:\mapsto\: \Bigl[ \sigma \:\mapsto\: H(S_\varphi + r \scrL; 
(1+ \varphi(\sigma))\sigma)\Bigr]
$$
at $\varphi_r$ has exactly one negative eigenvalue when restricted to even functions in $C^{1,\beta}(S)$, 
whereas all other eigenvalues are positive. This property follows from the fact 
that the linearization at $\varphi_r$ converges to the operator $2d_{N,\alpha}(L_\alpha -\lambda_1)$,
given by (\ref{eq:linNMC}-(\ref{eq:def-L-alpha}), 
as $r \to \infty$. This convergence is a mere consequence of the $C^2$-smoothness of the operator 
$\tilde H$ defined in (\ref{def-tilde-H-intro}) below, and the fact that $2d_{N,\alpha}(L_\alpha -\lambda_1)$
is the linearization at the unit sphere $S=\lim_{r\to\infty}S_{\varphi_r}$ by Lemma~\ref{lem:lin-NMC-op}.
Finally, one uses the spectral decomposition of $L_\alpha -\lambda_1$, already mentioned previously, and sees that,
among even functions, its eigenvalues are given by $-\lambda_1 < \lambda_2-\lambda_1 < \lambda_4-\lambda_1 < \cdots$.
The first one is negative and all others are positive.
} 
\end{remark}

As a corollary of Theorem~\ref{th:main-th-Brav-lat}, we obtain the following more explicit form of $\vp_r$ in
the case of  rectangular lattices. 

\begin{corollary} \label{cor:main-th}
Assume that  $\scrL$ is a rectangular lattice of dimension $ M\in \left\{1,\dots, N\right\}$.  Then the function $\vp_r$ in  Theorem \ref{th:main-th-Brav-lat}  expands as 
\be \label{eq:expan-orthog-lattice}
\phi_r(\th )= r^{-N-\alpha}\left(-\kappa_0  +   r^{-2}\left\{\kappa_1  \sum_{j=1}^M  \mu_j \th_j^2 -\kappa_2 \right\}
 +  o(r^{-2}) \right) \quad \text{for $\th \in S$ as $r \to +\infty$},
\ee
where $\mu_j=\sum \limits_{{p \in \scrL_*}} \frac{p_j^2}{|p|^{N+\alpha+4}}$. If, in particular,  $\scrL$ is a square  lattice then
\be \label{eq:expan-square-lattice}
\phi_r(\th )= r^{-N-\alpha}\left(-\kappa_0  +   r^{-2}\left\{ \tilde \kappa_1  \sum_{j=1}^M    \th_j^2 -\kappa_2 \right\}
 +  o(r^{-2}) \right) \quad \text{for $\th \in S$ as $r \to +\infty$},
\ee
where $\tilde \kappa_1=\frac{\k_1}{M} \sum \limits_{p \in \scrL_*}\frac{1}{|p |^{N+\alpha+2}}$.
\end{corollary}

As observed in Theorem \ref{th:main-th-Brav-lat}, for $M\leq N-1 $ the perturbation $\vp_r$ is nonconstant on $S$, i.e., the NMC of $H(S+r\scrL;\cdot)$ is nonconstant for $r$ large enough. On the other hand, if $\scrL$ is a square lattice of  dimension $N$, then by \eqref{eq:expan-square-lattice} we have 
$$
 \phi_r(\th )= r^{-N-\alpha}\Bigl(-\kappa_0  +   r^{-2}\Bigl(\tilde  \kappa_1  -\kappa_2 \Bigr)
 +  o(r^{-2}) \Bigr) \qquad \text{as $r\to \infty$,}
$$
hence the deformation of the lattice $S_{\vp_r}+r\scrL$ is uniform up to the order $r^{N-\alpha-2}$.  

In order to explain the idea of the proof of Theorem~\ref{th:main-th-Brav-lat}, it is convenient first to 
pay some attention to the linearized operator at $S\subset\R^N$ for the classical mean curvature ($\a=1$).
Since $S$ is a CMC surface, it is well known (see for instance Section~6 of \cite{FFMMM}) that the linearization of 
the mean curvature operator (recall that we take as mean curvature the arithmetic mean of the principal
curvatures) agrees with $(N-1)^{-1}$ times the second variation of perimeter, and thus is given by the 
Jacobi operator 
\begin{equation}
J \phi := (N-1)^{-1} \{-\Delta_S\phi - c^2\phi\}= (N-1)^{-1}\{-\Delta_S \phi- (N-1)\phi\} \qquad\text{ on } S,
\end{equation}
where $\Delta_S$ is the Laplace-Beltrami operator on $S$ and
$c^2=N-1$ is the sum of the squares of the principal curvatures of $S$. Here $\phi$ is a normal deformation as in 
\eqref{eq:first-phi}. Recall that $\Delta_S$ has the spherical 
harmonics as eigenfunctions, corresponding to the increasing sequence $k(k+N-2)$ of eigenvalues, with $k\geq 0$.  
Thus, $J$ has the same eigenfunctions but with eigenvalues 
\begin{equation}\label{eq:eigen-1}
\mu_k-\mu_1:=(N-1)^{-1}\{k(k+N-2)-(N-1)\}.
\end{equation}
Thus,
the first eigenvalue is negative and corresponds to constant functions on $S$ (that is, to the perturbation corresponding
to changing the radius of the sphere $S$). The third and next eigenvalues are all positive. But the second one ($k=1$)
is zero and has $\theta_i=x_i/|x|$ for $i=1,\ldots,N$ (the spherical harmonics of degree one) as eigenfunctions.
It is simple to see that this zero eigenvalue corresponds to translations of $S$ in $\R^N$, which do not change the mean curvature and thus
provide a zero eigenvalue.

As mentioned above, the linearized operator at $S\subset\R^N$ for the NMC $H$ is given by 
\eqref{eq:linNMC}-\eqref{eq:def-L-alpha}. It coincides, thus, with the second variation at $S$ of fractional perimeter. This nice formula is not immediate at all. We will derive it in
Section~5 in the Fr\'echet sense of linearization, after proving the smoothness of the NMC operator in Section~4. In a restricted sense related to the existence of directional derivatives, this formula for
the linearization also follows from results of D\'avila, del Pino, and Wei \cite[Appendix B]{Davila2014B} 
and of Figalli, Fusco, Maggi, Millot, and Morini \cite[Section 6]{FFMMM}.
These interesting papers found -  at any hypersurface $\cA$ - a simple
expression for the linearization of NMC with respect to any given normal
boundary variation. Note here that the NMC as defined 
in these two papers agrees with our $H/d_{N,\a}$; see \eqref{eq:NMC-PV}.

Note that the linearization of NMC at $S$, \eqref{eq:linNMC}-\eqref{eq:def-L-alpha}, has also the 
spherical harmonics as eigenfunctions -- as mentioned above. In addition, its second eigenvalue 
$2d_{N,\a}(\l_k-\l_1)$ (which corresponds to $k=1$) 
vanishes -- as in the local case. We will see below that, to apply the implicit function theorem, we
must get rid of this zero eigenvalue. For this, we will work only with perturbations of the sphere $S$ which
are even with respect to the origin of $\R^N$.

Just for consistency, we can now check that the eigenvalues of our nonlocal linearized
operator \eqref{eq:linNMC}  satisfy
$$
2d_{N,\a}(\lambda_k -\lambda_1) \to \mu_k-\mu_1=(N-1)^{-1}\{k(k+N-2) -(N-1)\} \qquad\text{ as } \a\to 1.
$$
Indeed, using  the fact that $ \G(z)=(z-1)\G(z-1)$, we get 
$$
 \frac{  \Gamma\left(  \frac{2k+N+1}{2} \right)  }{ \Gamma\left(  \frac{2k+N-3}{2} \right) } -   
 \frac{\Gamma\left(  \frac{N+1}{2} \right)  }{ \Gamma\left(  \frac{N-3}{2} \right) } =   k(k+N-2)
$$
and $((1-\a)/2)\Gamma((1-\a)/2)=\G((3-\a)/2)$.  
From these identities and recalling \eqref{eq:d_N-alpha} and \eqref{eq:l_k-0}, we deduce that
\begin{align*}
   \lim_{\a\to 1}2d_{N,\a}\l_k &=\lim_{\a\to 1} \frac{2(1-\alpha)\Gamma((N+1)/2)}{(N-1)\pi^{(N-1)/2}} 
   \frac{ \pi^{(N-1)/2}\, 2}{(1-\a)(1+\a)2^\a \G((N+1)/2)}  k(k+N-2)\\
   &= (N-1)^{-1} k(k+N-2)
 \end{align*}
  for $k\in \N$.

We can now outline the idea of the proof of Theorem~\ref{th:main-th-Brav-lat}, which is based on the implicit function 
theorem.
Let $c>0$ be sufficiently small such that the translates $S_\vp+r p $, $p \in \scrL$ do not intersect 
each other for $r>1/c$ 
and $\vp\in \cO=\{\vp \in  C^{1,\b}(S) \,:\, \|\vp\|_{L^\infty(S)}<1 \}$. 
We then rewrite the problem in the variable $\t =1/r$ and show that the nonlinear operator
$$
\tilde{H}: (-c ,c ) \times \cO   \to C^{\b-\a}(S)
$$
given by 
\begin{equation}
\label{def-tilde-H-intro}
\tilde{H}(\tau,\phi)(\th):= 
\begin{cases}
 H\left(S_\vp+ \frac{1}{\tau}\scrL ; (1+\vp(\th))\th \right)& \quad \textrm{ for $\t\in (-c ,c )\setminus\{0\}$, 
 $\vp\in\cO$} \vspace{3mm}\\
 H(S_\vp; (1+\vp(\th))\th) &   \quad \textrm{ for $\t=0  $, $\vp\in\cO$}
\end {cases}
\end{equation}
is of class $C^2$ in a neighborhood of $(\tau,\phi)=(0,0)$, and that its linearization at this point is given by   
$D_{\phi} \tilde{H}(0,0)=2d_{N,\a}\{L_\a-  \l_1\}: C^{1,\b}(S)  \to C^{\b-\a}(S)$.  As mentioned earlier, 
$\l_1$ is the first nontrivial eigenvalue of  the operator $L_\a$ with corresponding eigenspace spanned by the 
coordinate functions $\th_1,\dots,\th_N$. This yields  an $N$-dimensional kernel for the linearized NMC operator 
$D_{\phi} \tilde{H}(0,0)$. As mentiones above, this kernel comes from the invariance of the 
NMC operator under translations in $\R^N$.

Thus, in order to apply the implicit function theorem, we need to introduce function subspaces contained in a 
complement of this kernel. We consider
\begin{equation}
  \label{eq:def-space-X}
X=\{\vp\in  C^{1,\beta} (S)\,:\, \vp(-\th)=\vp(\th)\, \textrm{ for all $\th\in S$} \} 
\end{equation}
and
\begin{equation}
  \label{eq:def-space-Y}
Y=\{\vp\in C^{\beta-\a} (S)\,:\, \vp(-\th)=\vp(\th)\, \textrm{ for all $\th\in S$} \} ,
\end{equation}
the spaces of normal deformations which are even with respect to the origin of $\R^N$. In terms of 
the orthogonal basis given by the spherical harmonics, $X$ and $Y$ are generated by the 
spherical harmonics of even degree. 
We then consider the  restriction of  $\tilde{H}:(-c,c)\times (\cO\cap X)\to Y$, which takes values in $Y$ -- and thus 
is well defined -- thanks to the invariance of the lattice $\scrL$ under reflection through the origin. 
Moreover, $D_{\phi} \tilde{H}(0,0)=2d_{N,\a}\{L_\a-  \l_1\}: X \to Y$ will be an isomorphism.

Establishing the  regularity  of the operator $\tilde{H}$ turns out to be the most difficult step in the proof of 
Theorem~\ref{th:main-th-Brav-lat}. This will be done in Section 4.

The computation of the expansion in part (iv) of Theorem \ref{th:main-th-Brav-lat} is not straightforward and 
requires some care. In particular, we note that this is an expansion of order $o(|\tau|^{N+\alpha+2})$, 
whereas we shall see from~(\ref{eq:-1st-def-H}) below that $\tilde{H}$ fails to have more than $C^N$-regularity in the 
$\tau$-variable. 

The paper is organized as follows. In Section 2 we set up the functional analytic formulation of the problem in
order to apply the implicit function theorem. We also state Theorem~\ref{sec:nmc-lattice-prop-h-0} 
(to be proved in Section 4) on the 
smoothness of the NMC operator acting on perturbed spheres. In Section 3 we complete the proof of our main 
result, Theorem~\ref{th:main-th-Brav-lat}, after having stated in Theorem~\ref{sec:nmc-lattice-prop-h}
the main properties of the linearized NMC operator at the unit sphere. This theorem is proved in Section 5,
while the one on the nonlinear NMC (Theorem~\ref{sec:nmc-lattice-prop-h-0}) is established in Section 4.

\section{Preliminaries and functional analytic formulation of the problem}
\label{ss:Brav-lat}
Throughout the remainder of the paper, we let $N \ge 2$, and we let $S \subset \R^N$ 
and $B\subset\R^N$ denote the unit 
sphere and unit ball, respectively.
Let $M\in\N$ with $1\leq M\leq N$, and let $\scrL \subset \R^N$ be an $M$-dimensional lattice as 
defined in \eqref{eq:def-scrL-lattice}. Throughout this paper, we put  
$$
\scrL_*:= \scrL \setminus \{0\}\quad\text{and}\quad\Z^M_*:= \Z^M \setminus \{0\}.
$$
We note that 
\be\label{eq:inf-cJ}
\inf_{p\in\scrL _*  } |p|=: c_0>0
\ee 
and 
 \be
\sum_{p\in \scrL_*  }\frac{1}{|p|^{N+\a}} <\infty.
\ee

As in the introduction, we fix $\b\in (\a,1)$ and define 
\begin{equation}\label{eq:calOset}
\cO:= \{\phi \in C^{1,\b}(S):\: \|\phi\|_\infty < 1\}.
\end{equation}
Moreover, for $\phi \in \cO$, we consider the perturbed sphere 
$$
S_\phi:= \{ (1+ \phi(\s)) \s \::\: \s \in S\}
$$ 
and its parameterization  over the standard sphere defined by
\be \label{eq:def-F_vp}
F_\vp: S \to S_{\vp},\qquad  F_\vp(\s)=(1+\vp(\s)) \s.
\ee
For $\tau \in (-c_0/4,{c_0}/{4})\setminus \{0\}$, we then define 
$$
\cS_\phi^{\tau} :=\bigcup_{p \in \scrL} \Bigl( S_\phi  + \frac{p}{\tau}\Bigr)=S_{\vp}+\frac{1}{\t} \scrL.
$$
By \eqref{eq:inf-cJ} and since $S_\vp \subset B_{2}(0)$,  the set  $\cS_\phi^{\tau}$ is a noncompact hypersurface of class $C^{1,\beta}$, consisting of disjoint connected perturbed spheres and periodic with  respect to the lattice $\frac{1}{\tau} \scrL$. Due to the translation invariance properties of the lattice $\scrL$, the NMC of $ \cS_\phi^{\tau} $ is completely determined by its values on $S_\vp$.    More precisely, we have  
\be
\label{eq:restNMC-other-spheres}
 H(\cS_\phi^{\tau}\,;\, x +     \frac{\ov{p}}{\t})= H(\cS_\phi^{\tau}; x )\qquad \textrm{ for every $\ov{p}\in \scrL\,$ and $x\in S_\vp$}.
\ee
Thus, our aim is to solve the equation 
\be\label{eq:eq-to-solve} 
H(\cS_\phi^{\tau}\,;\,F_\vp(\th))=H(S;\th )=\frac{2d_{N,\alpha}}{\a} \int_{{S}}\frac{1- \s\cdot\th}{|\s-\th|^{N+\a}}\, dV(\s)    
\qquad \text{ for every  $\th\in S.$  } 
\ee
Note that $H(S;\th)$ is constant in $\th$.

In the following, for $\vp \in \cO$, we also let $B_\vp$ denote the unique open bounded set such that $\de B_\vp= S_\vp$, i.e.,
 $$
 B_\vp:= \bigl\{r F_\vp(\s)=r\left(1+ \vp\left(\s\right) \right) \s \,:\, 0\leq r<1,\, \s\in S  \bigr\}.
 $$
Moreover, we let $\nu_{S_\vp}$ denote the unit outer normal vector field on $S_{\vp}= \partial B_\vp$, and we let $dV_{S_\vp}$ denote the volume element on $S_\vp$.
For $\t \in(-c_0/4, {c_0}/{4})  \setminus \{0\}$, and $x \in S_\phi \subset \cS_\phi^{\tau}$, we then have 
\begin{align*}
&\frac{\a}{2d_{N,\a}}H(\cS_\phi^{\tau}\,;x) = \int_{\cS_\phi^\t} \frac{(z-x )\cdot \nu_{\cS_\vp^\t}(z)}{|z-x  |^{N+\alpha}}dV_{\cS_\vp^\t}(z) =  \sum_{p \in \scrL} \int_{S_\phi} \frac{(y-x+ \frac{p}{\tau} )\cdot \nu_{S_\vp}(y)}{|y-x+ \frac{p}{\tau} |^{N+\alpha}}dV_{S_\vp}(y)\\
&= \int_{S_\phi} \frac{(y-x)\cdot \nu_{S_\vp}(y)}{|y-x|^{N+\alpha}}\,dV_{S_\vp}(y)+ \frac{|\t|^{N+\a}}{\t} \sum_{p \in \scrL_*} \int_{S_\phi} \frac{(\t (y-x)+ {p} )\cdot \nu_{S_\vp}(y)}{|\t(y-x)+ {p}|^{N+\alpha}}\,dV_{S_\vp}(y).
\end{align*}

It will be convenient to use an alternative expression of the integrals appearing in the sum which does not involve boundary integration and which immediately shows that the sum is well defined. For this we note that, for fixed $\t \in(-c_0/4, {c_0}/{4})  \setminus \{0\}$, $p\in \scrL_*$ and $x \in S_\phi$, the function $y\mapsto |\tau(y- x )+{p}|^{-N-\alpha+2}$ is smooth in $\ov{B_\vp}$, and for all $ y\in \ov{B_\vp}$ we have 
$$
\n_y  |\tau(y-x)+{p}|^{-N-\alpha}=(-N-\alpha) \t ( \tau(y-x)+ p  ) |\tau(y-x)+{p}|^{-N-\alpha-2}.
$$
Since $S_\phi =\de B_\phi$, the divergence theorem  leads to
\begin{align*}
\int_{S_\phi} \frac{\bigl( \tau(y-x)+ p \bigr) \cdot \nu_{S_\vp }(y)}{|\tau(y-x)+ p |^{N+\alpha}}dV_{S_\vp}(y)&= \int_{B_\vp}\div_y \frac{\bigl( \tau(y- x)+{p}\bigr)  }{|\tau(y - x)+{p}|^{N+\alpha}}\,dy\\
&= - \a \t \int_{B_\vp }\frac{1}{ |\tau(y- x)+{p}|^{N+\alpha}} dy.
\end{align*}
Consequently, writing $x = F_\phi(\th)$ with $\th \in S$, we have 
\begin{equation}\label{eq:first-funct-eq}
\frac{\a}{2d_{N,\a}}H(\cS_\phi^{\tau}\,;\,F_\phi(\th ))= h(\phi) (\th) + |\t|^{N+\a} \sum_{p \in \scrL_*}
G_p(\tau,\phi)(\th )
\end{equation}
for $\th \in S$ and $ \tau \in (- {c_0}/{4}, {c_0}/{4}) \setminus \{0\}$, where 
\begin{equation}
   \label{eq:def-h}
  h(\phi)(\th ):=  \int_{S_\phi} \frac{(y-F_\phi(\th ))\cdot \nu_{S_\vp}(y)}{|y-F_\phi(\th )|^{N+\alpha}}\,dV_{S_\vp}(y)
\end{equation}
and 
\begin{equation}
  \label{eq:1-def-Gk}
G_p(\tau,\phi)(\th ):= -\alpha \int_{B_\vp }\frac{1}{ |\tau(y- F_\phi(\th ))+{p}|^{N+\alpha}} dy  \qquad \textrm{for $p \in \scrL_*$.} %
\end{equation}
Note that $h(\phi)(\th )$ is precisely the NMC of $S_\phi$ at $F_\phi(\th)$.  

In the following,  we will need that both $h$ and $G:= \sum_{p \in \scrL_*} G_p$ define smooth nonlinear operators 
between open subsets of suitable function spaces.
The following is the key result of the present paper in this regard. 

\begin{theorem}
\label{sec:nmc-lattice-prop-h-0}
With $\cO$ defined by \eqref{eq:calOset}, expression \eqref{eq:def-h} gives rise to a well defined map 
$$
h: \cO \to C^{\beta-\alpha}(S)
$$ 
which is of class $C^\infty$.
\end{theorem}

In the following, and with some abuse due to multiplicative constants, 
we will also call $h$ the nonlocal mean curvature operator over the sphere $S$. 
The proof of Theorem~\ref{sec:nmc-lattice-prop-h-0} is long and technically involved due to the singularity in the 
integrand in \eqref{eq:def-h}. Nevertheless, the result is a key step in our approach, 
and we believe that it might be of independent interest. We postpone the proof of Theorem~\ref{sec:nmc-lattice-prop-h-0} 
to Section~\ref{s:reg-NMC}; see Theorem~\ref{theorem-smoothness-h-1} below. 

With regard to $G_p$, we have a similar result.

\begin{proposition}
\label{sec:nmc-lattice-prop-Gk}
For $p \in \scrL_*$ and $\cO$ defined by \eqref{eq:calOset}, expression \eqref{eq:1-def-Gk} gives rise
to a well defined map 
$$
{G}_p: (-\frac{c_0}{4},\frac{c_0}{4}) \times \cO \to C^{\beta-\alpha}(S)
$$ 
which is of class $C^\infty$. Moreover, the map 
\be \label{eq:defG}
{G}:= \sum \limits_{p \in \scrL_*} {G}_p:  (-\frac{c_0}{4},\frac{c_0}{4}) \times \cO \to C^{\beta-\alpha}(S)
\ee
is well defined and of class $C^\infty$. 
\end{proposition}

The proof of Proposition~\ref{sec:nmc-lattice-prop-Gk} is also lenghty if all details are carried out, but it is much 
easier than the proof of Theorem~\ref{sec:nmc-lattice-prop-h-0} since the integrand   in 
(\ref{eq:1-def-Gk}) is not singular. We will outline the proof of Proposition~\ref{sec:nmc-lattice-prop-Gk} 
at the end of Section~\ref{s:reg-NMC} below.

We conclude this section by introducing the nonlinear operator 
$$
\cH: (-c_0/4, {c_0}/{4}) \times \cO \to C^{\beta-\alpha}(S) 
$$
given  by 
\begin{align}\label{eq:-1st-def-H} 
\cH(\tau,\phi)(\th):=  h(\phi)(\th )+ |\tau|^{N+\alpha} G (\tau,\phi)(\th ) 
\end{align}
for $\tau \in (-c_0/4, {c_0}/{4})$, $\vp \in \cO$ and $\th \in S$. By construction, we then have 
\be \label{eq:link-H--NMC}
\cH(\tau,\phi)(\th)=\frac{\a}{2d_{N,\alpha}}  H(\cS^\t_\vp\,;\,F_\vp(\th)) \qquad \text{for $\tau \in (-c_0/4, {c_0}/{4}) \setminus \{0\}$,}
\ee
i.e., the value $\cH(\tau,\phi)(\th)$ equals the NMC of $\cS^\t_\vp$ at the point $F_\vp(\th)$ up to a 
multiplicative constant.  We may thus formulate the parameter-dependent equation \eqref{eq:eq-to-solve} 
as an  operator equation in Hölder spaces. More precisely, we need to study the set of parameters $\tau \in (-c_0/4, {c_0}/{4})$ and functions $\phi \in \cO$ satisfying 
\begin{equation}
\label{eq:op-eq}
\cH(\tau,\phi) = h(0) \qquad \text{in $C^{\beta-\alpha}(S).$}
\end{equation}

Note that we will have
\begin{equation}
  \label{eq:C^2-cH-unrestricted}
\cH \in C^2\bigl( (-{c_0}/{4},{c_0}/{4}) \times \cO, C^{\beta-\alpha}(S)\bigr)
\end{equation}
as a consequence of Theorem~\ref{sec:nmc-lattice-prop-h-0}, Proposition~\ref{sec:nmc-lattice-prop-Gk} and the fact that the map $\t\mapsto |\t|^{N+\a}$ is of class $C^2$ since $N \ge 2$. Moreover, we have 
\begin{equation}
  \label{eq:C^2-cH-unrestricted-lin}
D_\vp \cH(0,0)= Dh(0) \in \cL \bigl (C^{1,\beta}(S), C^{\beta-\alpha}(S)\bigr).
\end{equation}

In the next section we restrict our attention to even functions and use the implicit function theorem to find a locally unique solution curve $\tau \mapsto (\tau,\phi(\tau))$ 
solving equation (\ref{eq:op-eq}) and with $\phi(0)=0$. 

%
%

\section{Completion of the proof of Theorem 1.1}
\label{sec:compl-proof-theor-1}
For fixed $\a\in (0,1)$ and $\b\in (\a, 1)$, we  consider, as in the introduction, the spaces  
$$
X=\{\vp\in  C^{1,\beta} (S)\,:\, \vp(-\th)=\vp(\th)\, \textrm{ for all $\th\in S$} \} ,
$$  
$$
Y=\{\vp\in C^{\beta-\a} (S)\,:\, \vp(-\th)=\vp(\th)\, \textrm{ for all $\th\in S$} \}.
$$ 
We claim that the operator $\cH$ defined in (\ref{eq:-1st-def-H}) restricts to a map 
$$
(-c_0/4, {c_0}/{4}) \times (\cO \cap X)  \to Y,
$$
which we will also denote by $\cH$ in the following. Indeed, for $\vp\in \cO \cap X$ and $\th \in S$ we have 
$-S_\vp=S_\vp$, $F_\vp(-\th)=-F_\vp(\th)$ and $\nu_{S_\vp}(-y)=-\nu_{S_\vp}(y)$ for $y \in S$. 
Thus, by a change of variables in (\ref{eq:def-h}), we have that   
\begin{align*}
h(\phi)(-\th )=  \int_{S_\phi} \frac{(y+F_\phi(\th ))\cdot \nu_{S_\vp}(y)}{|y+F_\phi(\th )|^{N+\alpha}}\,dV_{S_\vp}(y)&=
\int_{S_\phi} \frac{(F_\phi(\th )-y)\cdot \nu_{S_\vp}(-y)}{|y-F_\phi(\th )|^{N+\alpha}}\,dV_{S_\vp}(y)\\
&= \int_{S_\phi} \frac{(y-F_\phi(\th ))\cdot \nu_{S_\vp}(y)}{|y-F_\phi(\th )|^{N+\alpha}}\,dV_{S_\vp}(y)=h(\phi)(\th ). 
\end{align*}
Similarly, from (\ref{eq:1-def-Gk}), \eqref{eq:defG} and the fact that $-\scrL_*=\scrL_*$, we derive that 
$$
G(\tau,\phi)(-\th)= G(\tau,\phi)(\th) \qquad \text{for $(\tau,\vp) \in (-c_0/4,c_0/4)\times ( \cO\cap X)$ and $\th \in S$.}
$$
Consequently, it follows from (\ref{eq:-1st-def-H}) that $\cH$ maps $(-c_0/4, {c_0}/{4}) \times (\cO \cap X)$ into $Y$, as claimed. 

Moreover, by (\ref{eq:C^2-cH-unrestricted}) we have 
\begin{equation}
  \label{eq:C^2-cH-restricted}
\cH \in C^2\bigl( (-{c_0}/{4},{c_0}/{4}) \times (\cO \cap X),Y\bigr).
\end{equation}
Using the implicit function theorem within the spaces $X$ and $Y$, we will derive a locally unique curve $\tau \mapsto \vp(\tau) \in X$ such that $\vp(0)=0$ and 
$$
\cH(\tau,\vp(\t))=  h(0)  \qquad \text{in $Y$} 
$$
for $|\tau|$ sufficiently small. For this we shall need the following invertibility property of $Dh(0) \in \cL(X,Y)$, which by (\ref{eq:C^2-cH-unrestricted-lin}) coincides with $D_\phi \cH(0,0) \in \cL(X,Y)$.

\begin{theorem}
\label{sec:nmc-lattice-prop-h}
The linearized operator $Dh(0) \in \cL(C^{1,\beta}(S),C^{\beta-\alpha}(S))$ is given by  
$$
\frac{1}{\a}Dh(0)\vp=L_\a \vp -  \l_1\vp  \qquad \text{for $\phi \in C^{1,\beta}(S)$,}
$$
where 
$$
L_\a\vp(\th )=PV \int_S\frac{\vp(\th)-\vp(\s)}{|\th -\s|^{N+\a}}\, dV(\sigma)\qquad \text{for $\th \in S$}
$$
and $\lambda_1$ is given in \eqref{eq:l_k-0} for $k=1$. Moreover, $Dh(0)$ is an isomorphism when considered as a linear operator 
from $X$ to $Y$.
\end{theorem}
The proof of this theorem relies, in particular, on the spectral decomposition of the operator $L_\alpha$ and 
regularity estimates between Hölder spaces. It will be given in Section~\ref{sec:line-nmc-oper};
see Lemma~\ref{lem:lin-NMC-op} and Theorem~\ref{prop:Dh0-invert} below.

The following proposition is the result of applying the implicit function theorem. In its proof, we will need that
\begin{equation}
  \label{eq:H-even-tau}  
\cH(-\tau,\phi)= \cH(\tau,\phi) \qquad \text{for $\t \in (-c_0/4,c_0/4)$, $\phi \in \cO$,}  
\end{equation}
which is as a consequence of \eqref{eq:-1st-def-H}  and the fact that $-\scrL_*=\scrL_*.$

\begin{proposition}
\label{sec:nmc-lattice-implicit-function}
There exists $\t_0>0$ and an open neighborhood $\cU \subset X$ of $0$
for which there exists a unique curve $(-\t_0, \t_0) \to \cU$, $\tau \mapsto \phi(\tau)$, with $\phi(0)=0$ and 
\begin{equation}
  \label{eq:implicit-equation}
\cH(\tau,\phi(\tau))= h(0) \qquad \text{in $Y,\,\,\qquad$ for $-\t_0 < \tau < \t_0$.}
\end{equation}
Moreover, $\phi$ is of class $C^2$,  satisfies $\vp(-\t)=\vp(\t)$ and  the expansion 
\begin{equation}
  \label{eq:expansion-phi-tau-0}
\phi(\tau)= - |\tau|^{N+\alpha}\Bigl((Dh(0))^{-1} \Phi_0  
 + \frac{\tau^2}{6} (Dh(0))^{-1} \Phi_2+ o(\tau^2) \Bigr),
\end{equation}
where $\Phi_j:= \partial_\tau^j G(0,0) \in Y$, $j =0,2$, and $\frac{o(\tau^2)}{\tau^2} \to 0$ in $C^{1,\beta}(S)$ as $\tau \to 0$. 
\end{proposition}

\begin{proof}
Applying the implicit function theorem to the $C^2$-map $\cH: (-c_0/4, {c_0}/{4}) \times (\cO \cap X)  \to Y$ 
at the point $(0,0) \in (-\frac{c_0}{4},\frac{c_0}{4}) \times X$ and using Theorem~\ref{sec:nmc-lattice-prop-h}, we find $\t_0 \in (0, c_0/4)$ and a unique $C^2$-regular curve $(-\t_0, \t_0) \to \cU$, $\tau \mapsto \phi(\tau)$ such that 
(\ref{eq:implicit-equation}) holds. By \eqref{eq:H-even-tau}, we also have that 
$$
\vp(-\t)=\vp(\t)\qquad \text{ for every $\t\in (-\t_0,\t_0)$}.
$$

It thus remains to prove the expansion (\ref{eq:expansion-phi-tau-0}). For this we consider the $C^2$-curve 
$$
g: (-\t_0,\t_0) \to Y, \qquad g(\tau):= G(\tau,\phi(\tau)).
$$
Then (\ref{eq:implicit-equation}) can be written as 
$$
0=h(\phi(\tau))- h(0) + |\tau|^{N+\alpha} g(\tau) = Dh(0)\phi(\tau) + O(\|\phi(\tau)\|_X^{2}) + |\tau|^{N+\alpha}g(\tau) .
$$
Consequently,  we have 
\be\label{eq:intermed}
\phi(\tau) = - |\tau|^{N+\alpha} (Dh(0))^{-1} g(\tau)+ O(\|\phi(\tau)\|_X^{2}),
\ee
and thus the curve 
\be \label{eq:def-curv-psi}
\tau \mapsto \psi(\tau):=|\tau|^{-N-\alpha} \phi(\tau)
\ee
satisfies the expansion $\psi(\tau)=   -(Dh(0))^{-1}g(\tau)  + O( |\tau|^{N+\alpha}\|\psi(\tau)\|_X^{2})$ and,
in particular,
\begin{equation}
\label{2nd-exp-psi}
\psi(\tau) =  -(Dh(0))^{-1}g(\tau)  + o(\tau^2).
\end{equation}

We also note that  
$$
g'(\tau)= \partial_\tau G(\tau,\phi(\tau))+ \partial_\phi G(\tau,\phi(\tau))\phi'(\tau) 
$$
and 
\begin{align*}
g''(\tau) &= \partial_\tau^2 G(\tau,\phi(\tau)) + 2 \partial_\phi \partial_\tau G(\tau,\phi(\tau)) \phi'(\tau)\\
&+ \partial_\phi^2 G(\tau,\phi(\tau))[\phi'(\tau), \phi'(\tau)] + \partial_\phi G(\tau,\phi(\tau))\phi''(\tau) \nonumber
 \end{align*}
for $\tau \in (-\t_0,\t_0)$. Moreover, by \eqref{eq:intermed} we have $\phi(\tau)= O(|\tau|^{N+\alpha})$, and hence 
$\phi(0)=\phi'(0)= \phi''(0)=0$. We deduce
$$
g(0)= G(0,0),\qquad g'(0)= \partial_\tau G(0,0) \qquad \text{and}\qquad g''(0)= \partial_\tau^2 G(0,0).
$$
We thus infer that $g(\tau) = \sum \limits_{j=0}^2 \frac{\tau^j}{j!} \partial_\tau^j G(0,0)  + o(\tau^2)$, and together with \eqref{2nd-exp-psi} this yields the expansion
$$
\psi(\tau)=   - \sum_{j=0}^2 \frac{\tau^j}{j!} (Dh(0))^{-1} \partial_\tau^j G(0,0)  + o(\tau^2) .
$$
Using this in \eqref{eq:def-curv-psi}, and recalling that $\phi$ is even in $\t$ (and thus so is $\psi$),  
we get the expansion (\ref{eq:expansion-phi-tau-0}), as claimed.
\end{proof}

In the next two lemmas we compute the precise asymptotic expansion in $\tau$ for the perturbation $\phi$.

\begin{lemma}
\label{sec:nmc-lattice-lemma-tau-derivative}
The functions $\Phi_j:= \partial_\tau^j G(0,0) \in Y$, $j = 0,2$ are given by 
$$
\Phi_0 \equiv -\frac{\a |S|}{N}  \sum_{p \in \scrL_*}|{p} |^{-N-\alpha}   \qquad \text{on $S$} 
$$
and 
$$
\Phi_2(\theta) = a_1 \sum_{p \in \scrL_*}|{p}|^{-N-\alpha-2} 
    - a_2 \sum_{p \in \scrL_*} (p\cdot \th )^2 |{p} |^{-N-\alpha-4} \qquad \text{for $\th \in S$},
$$
where  
\begin{equation}
  \label{eq:def-a-i}
a_1 =\a  \frac{(N+\alpha )(N-\alpha)}{N(N+2)} |S| \qquad \text{and}\qquad a_2:= \a\frac{  (N+\alpha ) (N+\alpha+2 )  }{N} {|S|} . 
\end{equation}
\end{lemma}

\begin{proof}
 Let $p \in \scrL_*$ and $\th\in S$ be fixed. We then have 
$$
G_p(\tau,0)(\theta) = -\a \g_p(\t) \qquad  \text{for $\t \in (-c_0/4,c_0/4)$ with}\quad \g_p(\t)= 
\int_B|\tau(y-\th)+ p |^{-N-\alpha} dy,
$$
where $B$ is the unit ball of $\R^N$.
We first note that 
 $$
 \g_p(0)=| {p} |^{-N-\alpha}|B|= | {p} |^{-N-\alpha} \frac{|S|}{N} .
 $$
Moreover, for $\t \in (-c_0/4,c_0/4)$ we have 
$$
\g_p'(\t)=-(N+\alpha ) \int_B (y-\th)\cdot (\tau(y-\th)+ {p})|\tau(y-\th)+ {p} |^{-N-\alpha-2} dy
$$
and 
\begin{align*}
\g_p''(\t)&=-(N+\alpha ) \int_B |y-\th|^2  |\tau(y-\th)+ {p} |^{-N-\alpha-2} dy\\
&+(N+\alpha ) (N+\alpha+2)\int_B \left\{(y-\th)\cdot (\tau(y-\th)+{p})\right\}^2 |\tau(y-\th)+ {p} |^{-N-\alpha-4} dy. 
\end{align*}

Consequently, using the fact that odd terms do not contribute to the integral over $B$,
and recalling that $\int_B y_i^2 dy= N^{-1}\int_B |y|^2 dy$ and that $ \int_B |y|^2 dy=|S|/(N+2)$, we find that  
\begin{align*}
\g_p''(0)&= - (N+\alpha )| {p} |^{-N-\alpha-2}\int_B(|y|^2+1) dy\\
& \hspace{1cm}+   (N+\alpha )(N+\alpha +2)|  {p} |^{-N-\alpha-4}\int_B(( {p} \cdot \th )^2+ ( {p} \cdot y)^2) dy\\
&=  \frac{  (N+\alpha ) (N+\alpha+2 )  }{N} {|S|}  |{p} |^{-N-\a-4}  ({p} \cdot \th)^2\\
&\hspace{1cm}-(N+\alpha )|S| \Bigl( \frac{1}{N+2} +\frac{1}{N} - \frac{(N+\alpha+2)}{N(N+2)}\Bigr) |{p} |^{-N-\a-2}\\
&=  \frac{  (N+\alpha ) (N+\alpha+2 )  }{N} {|S|}  |{p} |^{-N-\a-4}  ({p} \cdot \th)^2 - \frac{(N+\alpha )(N-\alpha)}{N(N+2)} |S| |{p} |^{-N-\a-2}\\
&= \frac{a_2}{\a} |{p} |^{-N-\a-4}  ({p} \cdot \th)^2 - \frac{a_1}{\a} |{p} |^{-N-\a-2} ,
\end{align*} 
with $a_1,a_2$ defined in (\ref{eq:def-a-i}). We thus conclude that 
$$
\Phi_0 (\theta)= -  \a \sum_{p \in \scrL_*} \g_p(0) = -\frac{\a |S|}{N}  \sum_{p \in \scrL_*}|{p} |^{-N-\alpha}
$$
and 
$$
\Phi_2(\theta)= a_1 \sum_{p \in \scrL_*}|{p} |^{-N-\alpha-2} 
    - a_2 \sum_{p \in \scrL_*} (p\cdot \th )^2 |{p} |^{-N-\alpha-4}
$$
for $\theta \in S$, as claimed. 
 \end{proof}

\begin{lemma}
\label{sec:nmc-lattice-lemma-tau-derivative-h-inverse}
The functions $\Psi_j:= (Dh(0))^{-1} \Phi_j \in X$, $j = 0,2$, are given by 
$$
\Psi_0 \equiv \frac{ |S|}{\lambda_1 N }  \sum_{p \in \scrL_*}\frac{1}{| p|^{N+\alpha}}  \qquad \text{on $S$} 
$$
and 
\begin{align*}
\Psi_2(\theta)&= |S| \Bigl\{\frac{(N+\alpha)(N+\alpha+2) }{N^2(\lambda_2-\lambda_1)}  + 
\frac{2}{\lambda_1} \frac{(N+\a)(N+1)(\alpha +2)}{N^2(N+2)} 
\Bigr\} \sum_{p \in \scrL_*}\frac{1}{|p|^{N+\alpha+2}}\\
& -\frac{|S|(N+\alpha)(N+\alpha+2) }{N(\lambda_2-\lambda_1)} \sum_{p \in \scrL_*} \frac{(p\cdot \th)^2}{|p|^{N+\alpha+4}} \qquad \text{for $\th\in S$.}
\end{align*}
\end{lemma}

\begin{proof}
We recall from Theorem~\ref{sec:nmc-lattice-prop-h} that 
$$
(Dh(0))^{-1}= \frac{1}{\a} (L_\a - \lambda_1)^{-1}: Y \to X,
$$
with $L_\a$ given by (\ref{eq:def-L-alpha}). Since $L_\a$ maps constant functions to zero, we find that  
 $$ 
\Psi_0 \equiv \frac{1}{\a} (L_\a - \lambda_1)^{-1}\Bigl(-\frac{\a |S|}{N}  \sum_{p \in \scrL_*}|{p}|^{-N-\alpha}\Bigr) = \frac{|S|}{\lambda_1 N}\sum_{p \in \scrL_*}|{p} |^{-N-\alpha} \qquad \text{on $S$.}
$$

To compute $\Psi_2$, we introduce the functions 
$$
q_e \in C^{1,\beta}(S), \qquad q_e(\th )= (e \cdot \theta )^2-\frac{1}{N}
$$
for $e \in S$. Since $q_e$ is a spherical harmonic of degree two for every $e \in S$, we have 
$$
(Dh(0))^{-1} q_e = \frac{1}{\a} (L_\a - \lambda_1)^{-1} q_e = \frac{1}{\a(\lambda_2-\lambda_1)}q_e \qquad \text{for $e \in S$.}
$$
Moreover, by Lemma~\ref{sec:nmc-lattice-lemma-tau-derivative}, we have 
$$
\Phi_2 =  \sum_{p \in \scrL_*}\Bigl(a_1 -\frac{a_2}{N} -a_2 \: q_{\text{\tiny $\frac{p}{|p|}$}}\Bigr)|p |^{-N-\alpha-2}\qquad \text{in $Y$} 
$$
and thus 
$$
\Psi_2 = \frac{1}{\a} (L_\a - \lambda_1)^{-1}\Phi_2 = -\sum_{p \in \scrL_*}
\Bigl\{ \frac{1}{\a \lambda_1} \Bigl( a_1 -\frac{a_2}{N}\Bigr)  +\frac{ a_2 }{\a(\lambda_2-\lambda_1)}   q_{\text{\tiny $\frac{p}{|p|}$}}\Bigr\} |p|^{-N-\alpha-2} \qquad \text{in $X$,}
$$
i.e., 
\begin{align*}
&\Psi_2(\theta)= \Bigl\{\frac{ a_2 }{\a N(\lambda_2-\lambda_1)}- \frac{1}{\a \lambda_1} \Bigl( a_1 -\frac{a_2}{N}\Bigr) 
\Bigr\} \sum_{p \in \scrL_*}
| p |^{-N-\alpha-2} -\frac{a_2 }{\a(\lambda_2-\lambda_1)} \sum_{p \in \scrL_*} \frac{( p \cdot \th)^2}{ |p|^{N+\alpha+4}}\\
  &= |S| \Bigl\{\frac{(N+\alpha)(N+\alpha+2) }{N^2(\lambda_2-\lambda_1)}  - \frac{1}{\lambda_1} \frac{(N+\a)\{ N(N-\a)-(N+2)(N+\alpha+2)\}}{N^2 (N+2)} 
\Bigr\} \sum_{p \in \scrL_*}\frac{1}{| p |^{N+\alpha+2}}\\
&\hspace{1cm}-\frac{|S|(N+\alpha)(N+\alpha+2) }{N(\lambda_2-\lambda_1)} \sum_{p \in \scrL_*} 
\frac{( p \cdot \th)^2}{| p|^{N+\alpha+4}}\\
&=|S| \Bigl\{\frac{(N+\alpha)(N+\alpha+2) }{N^2(\lambda_2-\lambda_1)}  + \frac{1}{\lambda_1} 
\frac{(N+\a)(2N(\alpha +2) +4+2\a)}{N^2(N+2)} 
\Bigr\} \sum_{p \in \scrL_*}\frac{1}{|p|^{N+\alpha+2}}\\
& \hspace{1cm}-\frac{|S|(N+\alpha)(N+\alpha+2) }{N(\lambda_2-\lambda_1)} \sum_{p \in \scrL_*} 
\frac{( p\cdot \th)^2}{| p|^{N+\alpha+4}} \qquad \text{for $\th \in S$},
\end{align*}
as claimed.
\end{proof}

We may now complete the 

\begin{proof}[Proof of Theorem~\ref{th:main-th-Brav-lat}]
The existence and uniqueness of the curve $r \mapsto \vp_r$ with the properties of Theorem~\ref{th:main-th-Brav-lat}(i)--(iii) follows immediately from Proposition~\ref{sec:nmc-lattice-implicit-function} by setting $\phi_r:= \phi(\frac{1}{r})$. To obtain Theorem~\ref{th:main-th-Brav-lat}(iv), we note that by (\ref{eq:expansion-phi-tau-0}) and Lemma~\ref{sec:nmc-lattice-lemma-tau-derivative-h-inverse} we have the expansion 
\begin{align*}
\phi_r(\th)&=  -r^{-N-\alpha}
\Bigl(\Psi_0(\theta)+  \frac{r^{-2}}{6}\Psi_2(\th) + o(r^{-2})\Bigr)\\
&=r^{-N-\alpha}\Bigl(-\kappa_0  +   r^{-2}\Bigl\{  \kappa_1 \sum_{p \in \scrL_*} \frac{(\th \cdot p)^2}{|p|^{N+\alpha+4}} -\kappa_2\Bigr\}
 +  o(r^{-2}) \Bigr) \quad \text{for $\theta \in S$ as $r \to +\infty$},
\end{align*}
where
\begin{align*}
\kappa_0 &\equiv \Psi_0 =  \frac{ |S|}{\lambda_1 N }  \sum_{p \in \scrL_*}\frac{1}{|{p}|^{N+\alpha}}, \qquad  
\kappa_1 =\frac{|S|(N+\alpha)(N+\alpha+2) }{6 N(\lambda_2-\lambda_1)}  
\end{align*}
and
\begin{align*}
\kappa_2&= \frac{|S|}{6} \Bigl\{\frac{(N+\alpha)(N+\alpha+2) }{N^2(\lambda_2-\lambda_1)}  + 
\frac{2}{\lambda_1} \frac{(N+\a)(N+1)(\alpha +2)}{N^2(N+2)} 
\Bigr\} \sum_{p \in \scrL_*}\frac{1}{| p |^{N+\alpha+2}}.
\end{align*}

To prove Theorem~\ref{th:main-th-Brav-lat}(v), it suffices to show, after making $r_0$ larger if necessary, that 
the map 
$$
\th\mapsto \ti{f}(\th):= \sum_{ p\in \scrL_*} \frac{(\th \cdot p)^2}{|p|^{N+\alpha+4}}
$$
is non-constant on $S$ if $1\leq M\leq N-1$.  We readily observe that  
$\ti{f}(e_1)>0$ and  $ \ti{f}(e_N)=0$. The proof of Theorem~\ref{th:main-th-Brav-lat} is thus finished.

The last statement of the theorem, on uniqueness, is a direct consequence of the implicit function theorem.
\end{proof}

We conclude this section with the 
\begin{proof}[Proof of Corollary~\ref{cor:main-th}]
By assumption and up to a rotation, we may assume that 
 the lattice basis satisfies
$$ 
  \textbf{a}_i=\rho_i e_i \,   \qquad \textrm{  for   $i=1,\dots,M$},
$$
for some $\rho_i\in\R \setminus \{0\}$. It is convenient to define the map  
\be \label{eq:defcJ}
\cJ: \Z^M\to \scrL,\qquad  \cJ(k):= \sum_{i=1}^M k_i \textbf{a}_i =  (\rho_1k_1,\dots, \rho_M k_M,0,\dots,0) \quad \in \R^N.
\ee
Then we get  
 \begin{align*}
\sum_{p \in \scrL_*} \frac{(\th \cdot p)^2}{|p|^{N+\alpha+4}}=\sum_{k \in \Z^M_*} \frac{(\th \cdot \cJ(k))^2}{|\cJ(k)|^{N+\alpha+4}}&=
\sum_{k \in \Z^M_*} \frac{(\th_1\rho_1 k_1+ \dots +\th_M \rho_M k_M)^2}{|\cJ(k)|^{N+\alpha+4}}\\
&=\sum_{i,j=1}^M  \sum_{k \in \Z^M_*}  \frac{\th_i \th_j \rho_i  \rho_j k_i k_j}{|\cJ(k)|^{N+\alpha+4}}, 
\end{align*}
whereas for $i \not = j$ we have 
$$
\sum_{k \in \Z^M_*}  \frac{\th_i \th_j \rho_i  \rho_j k_i k_j}{|\cJ(k)|^{N+\alpha+4}} = 0 
$$
by oddness with respect to the reflection of $k$ at the axis $\{k_i=0\}$.  Hence we conclude that 
$$
\sum_{p \in \scrL_*} \frac{(\th \cdot p)^2}{|p|^{N+\alpha+4}}= \sum_{k\in\Z^M_*}  \frac{\th_1^2\rho_1^2 k_1^2+ \dots + \th_M^2\rho_M^2 k_M^2}{|\cJ(k)|^{N+\alpha + 4}}=
\sum_{i=1}^M \mu_i  \th_i^2
$$
with 
$$
\mu_i =    \sum_{k\in\Z^M_*}  \frac{ \rho_i^2 k_i^2 }{|\cJ(k)|^{N+\alpha + 4}}=   \sum_{p\in\scrL_*}  \frac{p_i^2 }{|p|^{N+\alpha + 4}}.
$$ 
Together with Theorem~\ref{th:main-th-Brav-lat}, this gives \eqref{eq:expan-orthog-lattice}.  To see \eqref{eq:expan-square-lattice}, we note that in the case of the square lattice we have $\rho_1=\rho_2= \dots = \rho_M$ and thus 
 $$
 \mu_i= \frac{1}{M}\sum_{j=1}^M \mu_j = \frac{1}{M}  \sum_{p\in\Z^M_*}  \frac{1}{|p|^{N+\alpha + 2}} \qquad \textrm{ for  $i=1,\dots,M$.}
 $$
 This ends  the proof of Corollary~\ref{cor:main-th}.
\end{proof}

\section{Regularity of the NMC operator over the unit sphere}\label{s:reg-NMC}

In this section we prove the smoothness of the nonlocal mean
curvature $h$ as asserted in Theorem~\ref{sec:nmc-lattice-prop-h-0}.

\subsection{Geometric preliminaries}
\label{sec:geom-prel}
For $\phi \in \cO$, we recall  the  parameterization $F_\vp: S\to S_{\vp}$ of $S_\vp$, defined in \eqref{eq:def-F_vp} by $F_\vp(\s)=(1+\vp(\s)) \s$. We shall need the following observation.
\begin{proposition}\label{prop:geom-pert-sphere}
Let $ \phi \in C^{1}({S}  )$ be  such that  $\|\phi\|_{L^{\infty}(S)}<1$. Then the unit outer normal (to the set enclosed by $S_\vp$) of $ S_\vp$ at a point $F_\vp(\sigma)$, $\sigma \in S$ is given by
$$
\nu_{S_\vp}(F_\vp(\s)) = \frac{(1+\phi(\s))\s  -\n \vp(\s)}{\sqrt{(1+\phi(\s))^2 + |\n\vp(\s)|^2}}.
$$
Moreover, for every continuous function $f$ on $\R^N$, we have
\begin{equation}
  \label{eq:transformation-rule}
\int_{S_\vp} f(y) \,dV_{S_\vp}(y)=  \int_{S} f\circ  F_\vp(\s) J_\vp(\s)  \, dV(\s) \qquad \text{with}\quad  J_\vp=   (1+ \vp)^{N-2}\sqrt{ (1+\vp)^{2}+  |\n \vp|^2}.
\end{equation}
Here and in the following, $\nabla \vp$ denotes the gradient vector field of $\vp$ on $S$.
\end{proposition}

\begin{proof}
We fix a local parametrization $ z\mapsto \s(z)$ of $S$, which gives rise to the local parameterization 
$z \to \widehat{F}(z)= (1+\vp(\s(z))) \s(z)$ of $S_\vp$.
The tangent vectors of $S_\vp$ at the point $\widehat{F}(z)$ are given by 
\be\label{e:Zi} 
Z_i(z):=\de_{z_i} \widehat{F}(z)= (1+ \vp(\s(z)))\,  \de_{z_i}\s(z)+ \de_{z_i} (\vp \circ \sigma)(z)\, \s(z) 
\ee 
with $\de_{z_i} (\vp \circ \sigma)(z) =\nabla \vp(\s(z)) \cdot  \de_{z_i}\s(z)$ for $i=1,\dots,N$. 
Since $\s\cdot \n\vp(\s)=0$ and $\sigma\cdot\de_{z_i}\sigma=0$ (which follows from $|\sigma|^2=1$), 
we thus conclude that the unit outer normal of $ S_\vp$ at a point $F_\vp(\sigma)$ with $\sigma= \sigma(z) \in S$ is given by 
 $$
\nu_{S_\vp}(F_\vp(\s))= \frac{(1+\phi(\s))\s  -\n \vp(\s)}{\sqrt{(1+\phi(\s))^2 + |\n\vp(\s)|^2}}.
 $$
 
We now turn to the proof of (\ref{eq:transformation-rule}). 
By the previous relations,
the first fundamental form of $ S_\vp$ is given by 
\be  \label{eq:def-ti-gij-s-vp} 
g_{ij}=Z_i \cdot Z_j = (1+  \vp \circ \s )^2 \de_{z_i}\s\cdot \de_{z_j}\s+ \de_{z_i} (\vp \circ \sigma) \de_{z_j} (\vp \circ \sigma) .
\ee
We now compute $\sqrt{\det(g)} (z)$ at a given point $z$ under the assumption that $ \de_{z_i}\s(z)\cdot \de_{z_j}\s(z)=\d_{ij}$. We then have that 
$$
(1+ \vp(\sigma(z)))^{-2(N-1)}\det(g)(z) =  \det(\id+ C)
$$
with the matrix $C=(C_{ij})_{ij}$ given by 
$$
C_{ij}=(1+ \vp(\sigma(z)))^{-2} \de_{z_i} (\vp \circ \sigma)(z) \de_{z_j} (\vp \circ \sigma)(z).
$$
Note that $C$ has only one non-zero eigenvalue given by $(1+ \vp(\s(z)))^{-2}|\n \vp(\s(z)) |^2$ with corresponding 
eigenvector 
$(\de_{z_i}(\vp\circ\sigma)(z))_i$. We thus have 
$$
(1+ \vp(\sigma(z)))^{-2(N-1)}\det(g)(z) = \det(\id+ C)
=1+ (1+ \vp(\s(z)))^{-2}|\n \vp(\s(z))  |^2, 
$$
and hence 
\begin{align*}
\sqrt{\det(g)}(z)    = (1+ \vp(\s(z)))^{N-2}\sqrt{ (1+\vp(\s(z)))^{2}+  |\n \vp(\s(z))|^2    }.
\end{align*}
We have thus computed the local change of the volume form when passing from $S_\vp$ to $S$, and this gives rise to the transformation rule (\ref{eq:transformation-rule}).
 \end{proof}

\subsection{Preliminary differential calculus formulas}\label{ss:prem-not}

For a finite set $\cN$, we let $|\cN| $ denote the number of elements of $\cN$. Moreover, we denote 
$\cN_\ell:=  \{1,\dots, \ell\}$ for $\ell\in \N$.
Let $Z$ be a  Banach space and $U$ a nonempty open subset of $Z$. If $T \in \cC^{\ell}(U,\R)$ and $u \in U$, 
then $D^\ell T(u)$ is a continuous 
symmetric $\ell$-linear form on $Z$ whose norm is given by  
$$
   \|D^{\ell}T ({u}) \|= \sup_{{u}_{1}, \dots,  {u}_{\ell}\in Z }
     \frac{\left|D^{\ell} T ({u})[u_1,\dots,u_\ell] \right|}{   \prod_{j=1}^\ell \|   {u}_{j} \|_{ Z }}    .
$$
If  $T_1,\, T_2 \in \cC^\ell(U,\R)$,  then also $T_1 T_2 \in \cC^\ell(U,\R)$, and the $\ell$-th derivative of $T_1 T_2$ at $u$ is given by 
\be \label{eq:Dk-T1T2}
D^\ell(T_1 T_2 )({u})[u_1,\dots, u_\ell]= \sum_{\cN \in   \scrS_\ell} D^{|\cN|} T_1({u})[u_n]_{n\in \cN} \,  
D^{\ell-|\cN|} T_2({u}) [u_n]_{n\in \cN^c} ,
\ee
where  $\scrS_\ell $ is the set of subsets of $\{1,\dots, \ell\} $ and  $ \cN^c= \{1,\dots, \ell\}\setminus \cN $ for $\cN \in  \scrS_\ell$. 

 If, in particular, $L: Z\to \R$ is a linear map and $|\cN|\geq 1$, we have 
\be \label{eq:Dk-LT2}
D^{|\cN|}(L T_2 )({u})[u_i]_{i\in \cN}= L({u})    D^{|\cN|} T_2({u})[u_i]_{i\in\cN } +  \sum_{j\in \cN} L({u}_j)   D^{|\cN|-1} T_2({u})     [u_i]_{{i \in \cN\setminus\{j\}}}.
\ee
Furthermore,  let $B:Z\times Z\mapsto \R$ be a   bilinear map and let  $Q:Z \mapsto \R $ be its associated    quadratic form (namely $Q(\vp)=B(\vp,\vp)$). Then   
\begin{align} \label{eq:Dk-QT2}
D^{|\cN|}(Q T_2 )(u)[u_i ]_{i\in \cN}=& B(u,u)    D^{|\cN|} T_2(u)[u_i]_{i\in\cN} \nonumber\\
&+  \sum_{j\in \cN} (B(u, u_j)+B(u_j,u))   D^{|\cN|-1} T_2(u) [u_i]_{{i \in \cN\setminus\{j\}}}  \nonumber\\
&+\sum_{i,j\in \cN}  (B(u_i,u_j)+B(u_j,u_i))  D^{|\cN|-2} T_2(u) [u_r]_{{r \in \cN\setminus\{i,j\}}}  .
\end{align}
We close this section by the well known  \textit{Fa\'{a} de Bruno formula}, see e.g. \cite{FaadeBruno-JW}. We let $T$ be as above and $g:\textrm{Im}(T)\to \R$ be a  $k$-times differentiable map.  The Fa\'{a} de Bruno formula states that 
\be 
\label{eq:Faa-de-Bruno}
D^k( g\circ T)(u)[u_1,\dots,u_k]= \sum_{\Pi\in\scrP_k} g^{ (\left|\Pi\right|)}(T(u)) \prod_{P\in\Pi} D^{\left|P\right| }T(u)[u_j]_{j \in P} ,
 \ee
for $u, u_1,\dots,u_k  \in U$, where $\scrP_k$ denotes the set of all partitions of  $\left\{ 1,\dots ,k \right\}$.

\subsection{Regularity of the nonlocal mean curvature operator over the sphere}

For every  $a,b\in S$, $b \not=-a$ 
we consider the regular curve 
\be 
\label{eq:def-gamma}
\g_{a,b}:[0,1] \to S, \qquad \g_{a,b}(t)=\frac{t a+(1-t)b}{| t a+(1-t)b |} ,
\ee
which clearly satisfies $\g_{a,b}(0)=b$ and $\g_{a,b}(1)=a$.

\begin{lemma}\label{lem:curve-gamma}
Consider the compact subset 
$$
S_*:= \{(a,b) \in S \times S\::\: |a-b| \le 1\} \subset S \times S.
$$
Then, there exists a constant $C>0$ depending only on $N$ with the property that for $(a,b),$ $(a_1,b_1),$ $(a_2,b_2) \in S_*$ and $t \in [0,1]$ we have  
\begin{align}
\label{eq:gamma-dot-bound}  
| \dot{\g}_{a,b}(t)|&\leq C |a-b|, \\
|\gamma_{a_1,b_1}(t)-\gamma_{a_2,b_2}(t)| &\le C\Bigl(|a_1-a_2|+|b_1-b_2|\Bigr) \quad \text{and}\label{eq:gamma-dot-bound-1}\\
|\dot \gamma_{a_1,b_1}(t)- \dot \gamma_{a_2,b_2}(t)| &\le C\Bigl(|a_1-a_2|+|b_1-b_2|\Bigr).\label{eq:gamma-dot-bound-1-1}   
\end{align}
\end{lemma}

\begin{proof}
For $t\in[0,1]$,  consider the function
\begin{equation}
\label{eq:def-Upsilont}
\Upsilon:  \R \times \R^{N} \times \R^N \to \R, \qquad \Upsilon(t,a,b) = |b+ t(a-b)|
\end{equation}
Since $|a-b| \le 1$ on $S_*$ and $t(1-t) \in [0,\frac{1}{4}]$ for $t \in [0,1]$, we see that 
\be \label{eq:up-low-boun-Upsilon-t}
\frac{\sqrt{3}}{2} \:\le\: \Upsilon(t,a,b)= \sqrt{1- t(1-t)| a-b |^2}\: \le\: 1 \qquad \textrm{ for $(t,a,b) \in [0,1] \times S_*$.}
\ee
By direct computations, we also see that
\begin{equation}
\label{eq:formular-derv-gamma}
  \dot{\g}_{a,b}(t)= \frac{a-b}{\Upsilon(t,a,b)}+\frac{(1-2t)|a-b|^2}{2}\: \frac{b+t( a-b) }{\Upsilon(t,a,b)^3}
\qquad \text{for $(t,a,b) \in [0,1] \times S_*$.}
\end{equation}
Hence \eqref{eq:up-low-boun-Upsilon-t} yields (\ref{eq:gamma-dot-bound}) with a suitable constant $C>0$.

Next we consider the function  
$$
\cV: \R \times \R^N \times \R^N \to \R^N,\qquad \cV(t,a,b):=\frac{(1-2t)|a-b|^2}{2}\,  \bigl(b+t(a-b)\bigr).
$$
Then we may write
$$
\gamma_{a,b}(t)=\frac{ t a+(1-t)b }{\Upsilon(t,a,b) }\quad \textrm{and}\quad \dot \gamma_{a,b}(t)=   \frac{a-b}{\Upsilon(t,a,b)}+\frac{\cV(t,a,b) }{\Upsilon(t,a,b)^3}\qquad \text{for $(t,a,b) \in [0,1] \times S_*$.} 
$$
By \eqref{eq:up-low-boun-Upsilon-t}, we see that the right hand sides of these equalities define $C^1$-functions in an open neighborhood of the compact set $[0,1] \times S_*$ in $\R^{2N+1}= \R \times \R^N \times \R^N$. Therefore, a standard argument shows that these functions are Lipschitz continuous on 
$[0,1] \times S_*$ with respect to the Euclidean distance of $\R^{2N+1}$, and from this (\ref{eq:gamma-dot-bound-1}) and (\ref{eq:gamma-dot-bound-1-1}) follow.
\end{proof}

%
%
%
 
The following is an expression for $h$, as defined in (\ref{eq:def-h}), 
where we remove the dependence on $\vp$ in the domain of integration.
\begin{proposition}\label{prop:new-express-nmc}
Let $\vp\in \cO$. Then, we have 
  \begin{align*} 
 {h}(\vp)(\th)=&-(1+\vp(\th)) \int_{{S}}\frac{  \vp(\th)-\vp(\s)   -  (\th-\s) \cdot\n\vp(\s) }{ |\th-\s|^{N+\alpha}}  (1+\vp(\s))^{N-2}\,  \cK_{\a}(\vp,\s,\th)       \,  dV(\s) \nonumber\\
&+ \int_{{S}}\frac{(\vp(\th)-\vp(\s)) ^2  }{ |\th-\s|^{N+\alpha}} (1+ \vp(\s))^{N-2}\,  \cK_{\a}(\vp,\s,\th)       \,  dV(\s)\\ 
&+\frac{1+\vp(\th )}{2}  \int_{{S}}\frac{1 }{  |\th-\s|^{N+\alpha-2}}      (1+\vp(\s))^{N-1 }      \,   \cK_{\a}(\vp,\s,\th)   \,dV(\s),
\end{align*}
where  $\cK_\a: \cO\times S\times S$ is given by
\begin{align*}
\cK_\a(\vp,\s,\th):=\frac{ 1}{\left( \frac{(\vp(\th)-\vp(\s))^2}{|\th-\s|^2} + (1+\vp(\s)) (1+\vp(\th))\right)^{(N+\a)/2}}.
\end{align*}
Moreover, all integrals above converge absolutely.

\end{proposition}
\begin{proof}
Let  $\vp\in \cO$.
By Proposition \ref{prop:geom-pert-sphere}, 
   for every $\th\in S$, we have 
\be\label{eq:def-h-reg}
-h(\vp)(\th)=  \int_{S} \frac{(F_\vp(\th)-F_\vp(\s))\cdot \nu_{S_\vp}(F_\vp(\s)) }{|F_\vp(\th)-F_\vp(\s)|^{N+\alpha}}J_\vp(\s)\,dV(\s) 
\ee
and thus 
\begin{align*}
-& h(\vp)(\th)=\int_{{S}} \frac{(\vp(\th)-\vp(\s))\, \s\cdot\nu_{S_\vp}(F_\vp(\s))\, J_\vp(\s) }{\left( {(\vp(\th)-\vp(\s))^2} +(1+\vp(\s)) (1+\vp(\th)) |\th -\s|^2\right)^{(N+\alpha)/2}}\,dV(\s)\nonumber\\
 &+(1+ \vp(\th)) \int_{{S}} \frac{ (\th-\s )\cdot\nu_{S_\vp}(F_\vp(\s)) \, J_\vp(\s) }{\left( {(\vp(\th)-\vp(\s))^2} + (1+\vp(\s)) (1+\vp(\th))|\th -\s|^2\right)^{(N+\alpha)/2}}\,dV(\s),
\end{align*}
where we used that $2 (1-\th\cdot\s)={|\th -\s|^2} $.   It follows that 
\begin{align} 
-& {h}(\vp)(\th)= \int_{{S}}\frac{\vp(\th)-\vp(\s)}{ |\th-\s|^{N+\alpha}}  \cK_{\a}(\vp,\s,\th)      \s\cdot \nu_{S_\vp}(F_\vp(\s))\, J_\vp(\s)\,  dV(\s) \nonumber\\
 &+ (1+\vp(\th ) )  \int_{{S}}\frac{\th-\s}{  |\th-\s|^{N+\alpha}} \cdot\nu_{S_\vp}(F_\vp(\s))\, J_\vp(\s)\,   \cK_{\a}(\vp,\s,\th)   \,dV(\s) .
  \label{eq:expr-h-almost-ok}
\end{align}

Letting 
  $\psi =1+\vp$, we get 
$$
  J_\vp(\s)=   \psi^{N-2}(\s)\sqrt{\psi^{2}(\s)+|\n \psi(\s)|^2}\qquad \text{and} \qquad   \nu_{S_\vp}(F_\vp(\s))= \frac{\s  \psi(\s)-\n \psi(\s)}{ \sqrt{\psi^{2}(\s)+|\n \psi(\s)|^2}}
$$
from Proposition~\ref{prop:geom-pert-sphere}. Consequently, 
$$
   \s \cdot   \nu_{S_\vp}(F_\vp(\s))=   \frac{\psi(\s)}{ \sqrt{\psi^{2}(\s)+|\n \psi(\s)|^2}}
$$
  and thus
$$
  J_\vp(\s) \s \cdot   \nu_{S_\vp}(F_\vp(\s))= \psi^{N-1}(\s).
$$
  Furthermore
  $$
  (\th-\s)  \cdot \nu_{S_\vp}(F_\vp(\s)) \, J_\vp(\s)=  - (\th-\s)  \cdot\n\psi(\s) \psi^{N-2 }(\s)  +(\th-\s)  \cdot\s   \psi^{N-1 }(\s).
  $$
Using the latter two identities in \eqref{eq:expr-h-almost-ok}, we find that
\begin{align*} 
- {h}(\vp)(\th)=  & \int_{{S}}\frac{\psi(\th)-\psi(\s)}{ |\th-\s|^{N+\alpha}}  \psi^{N-1}(\s)\,  \cK_{\a}(\vp,\s,\th)       \,  dV(\s) \nonumber\\
 &- \psi(\th )  \int_{{S}}\frac{(\th-\s) \cdot\n\psi(\s)  }{  |\th-\s|^{N+\alpha}}      \psi^{N-2 }(\s)      \,   \cK_{\a}(\vp,\s,\th)   \,dV(\s) \\
 &+\psi(\th )  \int_{{S}}\frac{(\th-\s) \cdot\s }{  |\th-\s|^{N+\alpha}}      \psi^{N-1 }(\s)      \,   \cK_{\a}(\vp,\s,\th)   \,dV(\s).
\end{align*}  
 Therefore
\begin{align*} 
-{h}(\vp)(\th)=& \int_{{S}}\frac{(\psi(\th)-\psi(\s)) \psi(\s) -  \psi(\th )(\th-\s) \cdot\n\psi(\s) }{ |\th-\s|^{N+\alpha}}  \psi^{N-2}(\s)\,  \cK_{\a}(\vp,\s,\th)       \,  dV(\s) \nonumber\\
 %
 %
 &+\psi(\th )  \int_{{S}}\frac{(\th-\s) \cdot\s }{  |\th-\s|^{N+\alpha}}      \psi^{N-1 }(\s)      \,   \cK_{\a}(\vp,\s,\th)   \,dV(\s).
\end{align*}  
We add and subtract  $(\psi(\th)-\psi(\s)) \psi(\th)$ to get 
 \begin{align*} 
- {h}(\vp)(\th)= &- \int_{{S}}\frac{(\psi(\th)-\psi(\s)) ^2  }{ |\th-\s|^{N+\alpha}}  \psi^{N-2}(\s)\,  \cK_{\a}(\vp,\s,\th)       \,  dV(\s) \nonumber\\
 &+\psi(\th) \int_{{S}}\frac{  \psi(\th)-\psi(\s)   -  (\th-\s) \cdot\n\psi(\s) }{ |\th-\s|^{N+\alpha}}  \psi^{N-2}(\s)\,  \cK_{\a}(\vp,\s,\th)       \,  dV(\s) \nonumber\\
 &+\psi(\th )  \int_{{S}}\frac{(\th-\s) \cdot\s }{  |\th-\s|^{N+\alpha}}      \psi^{N-1 }(\s)      \,   \cK_{\a}(\vp,\s,\th)   \,dV(\s).
\end{align*} 
We then conclude that 
  \begin{align}\label{eq:formh} 
- {h}(\vp)(\th)=&- \int_{{S}}\frac{(\psi(\th)-\psi(\s)) ^2  }{ |\th-\s|^{N+\alpha}}  \psi^{N-2}(\s)\,  
\cK_{\a}(\vp,\s,\th)       \,  dV(\s) \nonumber\\
 &+\psi(\th) \int_{{S}}\frac{  \psi(\th)-\psi(\s)   -  (\th-\s) \cdot\n\psi(\s) }{ |\th-\s|^{N+\alpha}}  
 \psi^{N-2}(\s)\,  \cK_{\a}(\vp,\s,\th)       \,  dV(\s) \\
 &-\frac{\psi(\th )}{2}  \int_{{S}}\frac{1 }{  |\th-\s|^{N+\alpha-2}}      \psi^{N-1 }(\s)     
 \,   \cK_{\a}(\vp,\s,\th)   \,dV(\s)  \nonumber.
\end{align}  

Let us now check that all integrals above converge absolutely.  Indeed, it is clear that   
\begin{align*}
\int_{{S}}\frac{1 }{  |\th-\s|^{N+\alpha-2}}  &    \psi^{N-1 }(\s)      \,   \cK_{\a}(\vp,\s,\th)   \,dV(\s)\\
&\leq (1-\|\vp\|_{\infty})^{-N-\a} \int_{{S}}\frac{1 }{  |\th-\s|^{N+\alpha-2}}      \psi^{N-1 }(\s)      \,dV(\s)<\infty
\end{align*}
and, since $(\psi(\th)-\psi(\s)) ^2\leq  \|\psi\|_{C^1(S)} ^2 |\th-\s|^2$, we also get 
$$
\int_{{S}}\frac{(\psi(\th)-\psi(\s)) ^2  }{ |\th-\s|^{N+\alpha}}  \psi^{N-2}(\s)\,  \cK_{\a}(\vp,\s,\th)       \,  dV(\s)<\infty.
$$
Next, if $|\th-\s|<1$, we can write 
 $$
  \psi(\th)-\psi(\s)   -  (\th-\s) \cdot\n\psi(\s)=\int_0^1 \left\{\n \psi (\g_{\th,\s}(t))-\n \psi (\g_{\th,\s}(0) ) 
  \right\}\cdot \dot{\g}_{\th,\s}(t) dt
  $$
with $\g_{\th,\s}$ defined in \eqref{eq:def-gamma}. By \eqref{eq:gamma-dot-bound} we thus have 
$$
 |  \psi(\th)-\psi(\s)   -  (\th-\s) \cdot\n\psi(\s)|\leq C  \|\psi\|_{C^{1,\b}(S)} |\th-\s|^{1+\b}, 
$$ 
and this obviously also holds, by enlargening $C>0$ if necessary, for $\th, \s \in S$ with $|\th-\s| \ge 1$.  From this and the fact that $\b\in(\a, 1)$, we obtain
\begin{align}
& \int_{{S}}\frac{|  \psi(\th)-\psi(\s)   -  (\th-\s) \cdot\n\psi(\s)| }{ |\th-\s|^{N+\alpha}}  \psi^{N-2}(\s)\,  \cK_{\a}(\vp,\s,\th)       \,  dV(\s)\\
 &\qquad \leq C (1-\|\vp\|_{\infty})^{-N-\a} \|\psi\|_{C^{1,\b}(S)}^{N-1 } \int_{{S}}\frac{1 }{  |\th-\s|^{N+\alpha-1-\b}}           \,dV(\s)<\infty. 
\end{align}
We then have that the integrals in the expression of $h$ converge absolutely.
\end{proof}

 For $0<r<2$, we now put 
\be 
\cB(r)=\left\{ \psi\in C^{1,\b}(S)\,:\, r< \psi<2 \ \text{ in }S\right\}. 
\ee
 We  consider the map $\ov{\cK}_{\a}:\cB(r)\times S \times S\to \R$  defined by 
\begin{align}
\ov{\cK}_{\a}(\psi,\s,\th):=\frac{ 1}{\left( \frac{(\psi (\th)-\psi(\s))^2}{|\th-\s|^2} +  \psi(\s)  \psi(\th)\right)^{(N+\a)/2}}. 
\end{align}
We   also define  $\ti{h}: \cB(r) \to L^\infty(S)$ by   $\ti{h}(\psi) := {h}(\psi-1) $.    Then, by Proposition \ref{prop:new-express-nmc}, we have
\begin{align} 
 \ti{h}(\psi)(\th )= &\  h(\psi -1)(\th)= \int_{{S}}\frac{(\psi(\th)-\psi(\s)) ^2  }{ |\th-\s|^{N+\alpha}}  
 \psi^{N-2}(\s)\,  \ov{\cK}_{\a}(\psi,\s,\th)       \,  dV(\s) \nonumber\\
 &-\psi(\th) \int_{{S}}\frac{  \psi(\th)-\psi(\s)   - (\th-\s) \cdot\n\psi(\s) }{ |\th-\s|^{N+\alpha}}  \psi^{N-2}(\s)\, \ov{\cK}_{\a}(\psi,\s,\th)       \,  dV(\s) \label{eq:def-ti-h-00}\\
 &+\frac{\psi(\th )}{2}  \int_{{S}}\frac{1 }{  |\th-\s|^{N+\alpha-2}}      \psi^{N-1 }(\s)      \,   \ov{\cK}_{\a}(\psi,\s,\th)   \,dV(\s).\nonumber
\end{align}
%
  
%
The proof of Theorem \ref{sec:nmc-lattice-prop-h-0} will be completed once  we prove that  
$ \ti{h} :  \cB(r) \to C^{\b-\a}(S )$ is smooth for every $r>0$.

We define $\L_1 : C^{1,\b}(S)\times S \times S\to \R$  by
$$
\L_1(\psi,\s,\th)=   \psi(\th)-\psi(\s)   -  (\th-\s) \cdot\n\psi(\s) =\int_0^1 \left\{\n \psi (\g_{\th,\s}(t))-
\n \psi (\g_{\th,\s}(0) ) \right\}\cdot \dot{\g}_{\th,\s}(t) dt
$$
and $  \L_2 : C^{1,\b}(S)\times C^{1,\b}(S)\times  S \times S\to \R$ by
\begin{align*}
\L_2(\psi_1,\psi_2 ,\s,\th)&=  (\psi_1(\th)-\psi_1(\s)) (\psi_2(\th)-\psi_2(\s)). \nonumber 
\end{align*}
With this notation, we have 
\begin{align} 
 \ti{h}(\psi)(\th )= {h}(\psi-1)(\th)= & \int_{{S}}\frac{\Lambda_2(\psi,\psi,\s,\th) }{ |\th-\s|^{N+\alpha}}  \psi^{N-2}(\s)\,  \ov{\cK}_{\a}(\psi,\s,\th)       \,  dV(\s) \nonumber\\
 &-\psi(\th) \int_{{S}}\frac{\Lambda_1(\psi,\s,\th)}{ |\th-\s|^{N+\alpha}}  \psi^{N-2}(\s)\, \ov{\cK}_{\a}(\psi,\s,\th)       \,  dV(\s) \label{eq:def-ti-h-0}\\
 &+\frac{\psi(\th )}{2}  \int_{{S}}\frac{1 }{  |\th-\s|^{N+\alpha-2}}      \psi^{N-1 }(\s)      \,   \ov{\cK}_{\a}(\psi,\s,\th)   \,dV(\s).\nonumber
\end{align}

\begin{remark}\label{example-rotations}{\rm
It will be convenient to modify this representation further such that the singularity of the integrand does not depend on $\th$. For this we fix $e \in S$ and a Lipschitz continuous map of rotations $S \mapsto SO(N)$, $\th \mapsto R_\theta$ with the property that 
\begin{equation}
  \label{eq:def-Se}
S_e:=  \{\th \in S\::\: \th \cdot e \ge 0\}  \subset  \{\th \in S\::\: R_{\th}e = \th \}.
\end{equation}

The following is a possible way to construct $R$.
For fixed $e \in S$, consider the map $\th \mapsto R_\theta$ defined as follows. For $\theta \in S$ with $\th \cdot  e \ge 0$, we let $R_\theta$ be the rotation of the angle $\arccos \th \cdot e$ which maps $e$ to $\th$ and keeps all vectors perpendicular to $\th$ and $e$ fixed. We then extend the map $\th \mapsto R_\theta$ to all of $S$ as an even map with respect to reflection at the hyperplane $\{\th \in \R^N \::\: \th \cdot e = 0\}$.
}
\end{remark}

By construction, it is clear that 
\be\label{eq:Rota-isom}
|R_{\th}\s-\th |=|\s-e| \qquad \textrm{ for all $\th\in S_e$ and $\s\in S$}. 
\ee
Moreover, the Lipschitz property of the map $\th \mapsto R_\th$ implies that there is a constant $C>0$ with 
\be \label{eq:R-Lipschitz}
\|R_{\th_1} - R_{\th_2}\| \le C |\th_1-\th_2| \qquad \text{for all $\th_1,\th_2 \in S$,} 
\ee 
where, here and in the following, $\|\cdot\|$ denotes the usual operator norm with respect to the 
Euclidean norm on $\R^N$.

Thanks to \eqref{eq:Rota-isom},   a change of variable gives
\begin{align} 
 \ti{h}(\psi)(\th )=& \int_{{S}}\frac{\Lambda_2(\psi,\psi,R_\th \s,\th) }{ |e-\s|^{N+\alpha}}  \psi^{N-2}(R_\th \s)\,  \ov{\cK}_{\a}(\psi,R_\th \s,\th)       \,  dV(\s) \nonumber\\
 &-\psi(\th) \int_{{S}}\frac{\Lambda_1(\psi,R_\th \s,\th)}{ |e-\s|^{N+\alpha}}  \psi^{N-2}(R_\th \s)\, \ov{\cK}_{\a}(\psi,R_\th \s,\th)       \,  dV(\s) \label{eq:def-ti-h}\\
 &+\frac{\psi(\th )}{2}  \int_{{S}}\frac{ \psi^{N-1 }(R_\th \s)      \, }{  |e-\s|^{N+\alpha-2}}        \ov{\cK}_{\a}(\psi, R_\th \s,\th)   \,dV(\s) \qquad \text{ for $\th \in S_e$.} \nonumber
\end{align}

In the following, for a function $f: S \to \R$, we use the notation
$$
[f; \theta_1,\theta_2]:= f(\theta_1)-f(\theta_2)\qquad \text{ for $\theta_1,\theta_2 \in S$,} 
$$
and we note the obvious equality 
\be \label{eq:prod-bracket}
[fg; \theta_1,\theta_2] = [f ;\theta_1,\theta_2]g(\theta_1) +f(\theta_2)[g;\theta_1,\theta_2] \quad \text{for $f,g: S \to \R$, $\theta_1,\theta_2 \in S$.}
\ee

In the next results we collect helpful estimates for the functionals $\Lambda_1$ and $\Lambda_2$.

\begin{lemma}
\label{lemma-L_1-est-sec:regul-nonl-mean-1}
There exists a  constant $C>0$ depending only on $N$ and $\beta$
such that for all $\s,\s_1,\s_2,\th,\th_1,\th_2\in S$ and $\psi \in C^{1,\b}(S)$ we have 
\be \label{eq:estL1-th-sig-0}
 | \L_1(\psi,\s,\th)|\leq C  \|\psi\|_{C^{1,\b}(S)} |\s-\th|^{1+\b}
 \ee 
 and 
 \begin{align}
| \L_1(\psi,\s_1,\th_1)- \L_1(\psi,\s_2,\th_2)|\leq  & \ C\|\psi\|_{C^{1,\b}(S)} \,|\th_1-\s_1|  ( |\th_1-\th_2|^\b+|\s_1-\s_2|^\b )\nonumber\\
&+C  \|\psi\|_{C^{1,\b}(S)}\,|\th_2-\s_2|^\b( |\th_1-\th_2|+|\s_1-\s_2| ) .
\label{eq:estL1-th1-sig1-0}
\end{align}
\end{lemma}
\begin{proof}
To derive the   estimates in the lemma, we may assume that 
\be \label{eq:max-th-sig--th12-sig12}
\max \{|\s-\th|, |\th_1-\s_1|, |\th_2-\s_2| \}
< 1
\ee
(otherwise the estimates are easy to prove). Having \eqref{eq:max-th-sig--th12-sig12} 
is essential for applying Lemma~\ref{lem:curve-gamma} in the sequel. We have 
$$
\L_1(\psi,\s,\th)=   \psi(\th)-\psi(\s)   -  (\th-\s) \cdot\n\psi(\s) =\int_0^1 \left\{\n \psi (\g_{\th,\s}(t))-
\n \psi (\g_{\th,\s}(0) ) \right\}\cdot \dot{\g}_{\th,\s}(t) dt,
$$
where $\g_{\th,\s}$ is defined in \eqref{eq:def-gamma}.
Therefore  \eqref{eq:estL1-th-sig-0} follows  from \eqref{eq:gamma-dot-bound}.

We now prove \eqref{eq:estL1-th1-sig1-0}. We   have 
\begin{align*}
& \hspace{-1cm} \L_1(\psi,\s_1,\th_1)- \L_1(\psi,\s_2,\th_2)\\
= & \ \int_0^1 \left\{\n \psi (\g_{\th_1,\s_1}(t))-\n \psi (\g_{\th_1,\s_1}(0) ) \right\}\cdot \dot{\g}_{\th_1,\s_1}(t) dt\\
&- \int_0^1 \left\{\n \psi (\g_{\th_2,\s_2}(t))-\n \psi (\g_{\th_2,\s_2}(0) ) \right\}\cdot \dot{\g}_{\th_2,\s_2}(t) dt\\
= & \ \int_0^1 \left\{\n \psi (\g_{\th_1,\s_1}(t)) -\n \psi (\g_{\th_2,\s_2}(t))+\n\psi(\s_2)-\n \psi (\s_1 ) \right\}\cdot \dot{\g}_{\th_1,\s_1}(t) dt\\
&+  \int_0^1 \left\{\n \psi (\g_{\th_2,\s_2}(t))-\n \psi (\g_{\th_2,\s_2}(0) ) \right\}\cdot( \dot{\g}_{\th_1,\s_1}(t)- \dot{\g}_{\th_2,\s_2}(t)) dt.
\end{align*} 
This implies that
\begin{align*}
&| \L_1(\psi,\s_1,\th_1)- \L_1(\psi,\s_2,\th_2)|\\
&\leq \|\psi\|_{C^{1,\b}(S)}  \int_0^1 |\g_{\th_1,\s_1}(t) -\g_{\th_2,\s_2}(t) |^\b   | \dot{\g}_{\th_1,\s_1}(t)| dt 
+ \|\psi\|_{C^{1,\b}(S)} |\s_1 -\s_2 |^\b  \int_0^1     | \dot{\g}_{\th_1,\s_1}(t)| dt\\
&\qquad +\|\psi\|_{C^{1,\b}(S)}   \int_0^1 \left| \g_{\th_2,\s_2}(t)- \g_{\th_2,\s_2}(0)  \right|^\b |\dot{\g}_{\th_1,\s_1}(t)- \dot{\g}_{\th_2,\s_2}(t)| dt\\
&\leq \|\psi\|_{C^{1,\b}(S)}  \int_0^1 |\g_{\th_1,\s_1}(t) -\g_{\th_2,\s_2}(t) |^\b   | \dot{\g}_{\th_1,\s_1}(t)| dt + \|\psi\|_{C^{1,\b}(S)}  |\s_1 -\s_2 |^\b \int_0^1   | \dot{\g}_{\th_1,\s_1}(t)| dt\\
&\qquad +\|\psi\|_{C^{1,\b}(S)}   \int_0^1 \left| \int_0^1t \dot{\g}_{\th_2,\s_2}(rt) dr  \right|^\b |\dot{\g}_{\th_1,\s_1}(t)- \dot{\g}_{\th_2,\s_2}(t)| dt.
\end{align*} 
Thanks to \eqref{eq:gamma-dot-bound}, we get
\begin{align*}
&| \L_1(\psi,\s_1,\th_1)- \L_1(\psi,\s_2,\th_2)| \\
& \quad \leq C \|\psi\|_{C^{1,\b}(S)} |\s_1-\th_1|\left( \int_0^1 |\g_{\th_1,\s_1}(t) -\g_{\th_2,\s_2}(t) |^\b    dt+  |\s_1-\s_2|^\b\right) \\
&\qquad +\|\psi\|_{C^{1,\b}(S)}   \int_0^1 \left| \int_0^1t \dot{\g}_{\th_2,\s_2}(rt) dr \right|^\b |\dot{\g}_{\th_1,\s_1}(t)- \dot{\g}_{\th_2,\s_2}(t)| dt.
\end{align*} 

From Lemma \ref{lem:curve-gamma}, we have   that,
 for every $\th_1,\s_1,\th_2,\s_2\in S$ and satisfying \eqref{eq:max-th-sig--th12-sig12},
$$
| \g_{\th_1,\s_1}(t)- \g_{\th_2,\s_2}(t)|+|  \dot{\g}_{\th_1,\s_1}(t)-\dot{\g}_{\th_2,\s_2}(t)|\leq C( |\th_1-\th_2|+|\s_1-\s_2| ) \qquad \textrm{ for every $t\in[0,1]$}
$$
and 
$$
\left|\dot{\g}_{\th_2,\s_2}(rt)  \right| \leq C |\th_2-\s_2| \qquad \textrm{ for every $t,r\in[0,1]$}.
$$
Therefore 
\begin{align*}
| \L_1(\psi,\s_1,\th_1)- \L_1(\psi,\s_2,\th_2)|&\leq  C\|\psi\|_{C^{1,\b}(S)}  |\th_1-\s_1|
\left( |\th_1-\th_2|^\b+|\s_1-\s_2|^\b \right)\\
&\quad +C  \|\psi\|_{C^{1,\b}(S)}|\th_2-\s_2|^\b( |\th_1-\th_2|+|\s_1-\s_2| ) .
\end{align*}
This ends the proof of \eqref{eq:estL1-th1-sig1-0}.   
\end{proof}

\begin{corollary}
\label{cor-L_1-est-sec:regul-nonl-mean-1}
There exists a  constant $C>0$, depending only on $N$ and $\beta$,  such that for all 
$e\in S$, $\s \in S$, all $\th,\th_1,\th_2\in S_e$ and all $\psi \in C^{1,\b}(S)$ we have 
\be \label{eq:estL1-th-sig}
 | \L_1(\psi,R_\th \s,\th)|\leq C  \|\psi\|_{C^{1,\b}(S)} |e-\s|^{1+\b}
 \ee 
 and 
 \begin{equation}
| \L_1(\psi, R_{\th_1} \s,\th_1)- \L_1(\psi, R_{\th_2} \s,\th_2)|\leq  C\|\psi\|_{C^{1,\b}(S)} \, \bigl( |e-\s| 
|\th_1-\th_2|^\b + |e-\s|^\b|\th_1-\th_2|\bigr) .
\label{eq:estL1-th1-sig1}
\end{equation}
\end{corollary}

\begin{proof}
It suffices to apply \eqref{eq:estL1-th-sig-0} and (\ref{eq:estL1-th1-sig1-0}) with $\s$, $\s_1$ and $\s_2$
replaced by$R_{\th}\s$, $R_{\th_1}\s$ and $R_{\th_2}\s$ respectively, to use \eqref{eq:Rota-isom},
and the fact  that 
$$
 |R_{\th_1}\s - R_{\th_2}\s| \le \|R_{\th_1}-R_{\th_2}\| \le C |\th_1-\th_2|.
$$  
\end{proof}

Next, we derive estimates for $\Lambda_2$.
\begin{lemma}
\label{lemma-L-2-est-sec:regul-nonl-mean}
There exists a  constant $C>0$, depending only on $N$ and $\beta$, such that for all 
$e\in S$, $\s \in S$,  all $\th,\th_1,\th_2 \in S_e$  and all $\psi_1,\psi_2 \in C^{1,\b}(S)$ we have 
\be \label{eq:estL2-th-sig}
 | \L_2(\psi_1,\psi_2, R_\th \s, \th)|\leq C  \|\psi_1\|_{C^{1,\b}(S)}\|\psi_2\|_{C^{1,\b}(S)} |e-\s|^{2}
 \ee 
and 
\begin{align}
| \L_2(\psi_1,\psi_2,& R_{\theta_1} \s,\th_1)- \L_2(\psi_1,\psi_2, R_{\theta_2} \s, \th_2)| \label{eq:estL2-th1-sig1}\\
& \leq  C\|\psi_1\|_{C^{1,\b}(S)} \|\psi_2\|_{C^{1,\b}(S)}  |e-\s|^2 |\theta_1-\theta_2|^\b. \nonumber
\end{align} 
\end{lemma}

\begin{proof}
To prove the lemma, by \eqref{eq:Rota-isom} and \eqref{eq:R-Lipschitz} it is easy to see that     
we may assume $|e-\s| < 1$. This implies that 
$$
|R_\th \s-\th|= |R_{\th_1} \s-\th_1|= |R_{\th_2}\s- \th_2| = |\s-e| < 1,
$$
and therefore allows us to apply Lemma~\ref{lem:curve-gamma} in the following.
By \eqref{eq:gamma-dot-bound} we have 
\begin{align}
  |\psi(\th) &- \psi(R_\th \s)| = \Bigl| \int_0^1 \n \psi (\g_{\th,R_\th \s}(t)) \cdot \dot{\g}_{\th,R_\th \s}(t) dt\Bigr| \nonumber\\
&\le C  \|\psi\|_{C^{1,\b}(S)} |\th-R_\th \s| = C  \|\psi\|_{C^{1,\b}(S)} |e-\s| \qquad \text{for $\psi \in C^{1,\beta}(S)$} \label{eq:est-psi-1-S}
\end{align}
and 
\begin{align*}
| \L_2(\psi_1,\psi_2, R_\th \s,\th)| &= |(\psi_1(\th) - \psi_1(R_\th \s))| |(\psi_2(\th) - \psi_2(R_\th \s))|\\
&\le  C  \|\psi_1\|_{C^{1,\b}(S)}\|\psi_2\|_{C^{1,\b}(S)} |e-\s|^{2},  
\end{align*}
as claimed in \eqref{eq:estL2-th-sig}. 

Next, we note that, by (\ref{eq:formular-derv-gamma}) we have 
\begin{align*}
\dot{\g}_{\th, R_{\theta} \s}(t) &= \frac{\th-R_\th \s }{\Upsilon(t,\th,R_\th \s)}+\frac{(1-2t)|\th-R_\th \s|^2}{2}\:  \frac{R_\th \s+t( \th-R_\th \s)}{\Upsilon(t,\th,R_\th \s)^3}\\
&= R_\th \left\{ \frac{(e - \s )}{\Upsilon(t,e,\s)}+\frac{(1-2t)|e-\s|^2}{2}\:  \frac{ \s+t(e- \s)}{\Upsilon(t,e, \s)^3}\right\}
= R_\th \dot{\g}_{e,\s}(t),
\end{align*}
since $|\th-R_\th \s|= |R_\th (e-\s)|= |e-\s|$ and, by (\ref{eq:def-Upsilont}),  
$$
\Upsilon(t,\th,R_\th \s)= |R_\th \s +t(\th-R_\th \s)|=|R_\th( \s +t(e- \s))|=|\s + t(e-\s)|=\Upsilon(t,e,\s). 
$$
Consequently, 
\begin{align*}
\bigl|\psi(\theta_1)&-\psi(R_{\theta_1} \s ) - \bigl(\psi(\theta_2)-\psi(R_{\theta_2} \s )\bigr) \bigr|\\
=&   \Bigl| \int_0^1 \left\{ \n \psi (\g_{\th_1, R_{\theta_1} \s}(t)) \cdot \dot{\g}_{\th_1, R_{\theta_1} \s}(t)- 
\n \psi (\g_{\th_2, R_{\theta_2} \s}(t)) \cdot \dot{\g}_{\th_2, R_{\theta_2} \s}(t) \right\}  dt\Bigr|\\ 
=&   \Bigl| \int_0^1 \left\{ \n \psi (\g_{\th_1, R_{\theta_1} \s}(t)) \cdot R_{\th_1} \dot{\g}_{e,\s}(t)- 
\n \psi (\g_{\th_2, R_{\theta_2} \s}(t)) \cdot R_{\th_2} \dot{\g}_{e,\s}(t)  \right\}  dt\Bigr|\\ 
=&   \Bigl| \int_0^1 \Bigl( \n \psi (\g_{\th_1, R_{\theta_1} \s}(t)) -    \n \psi (\g_{\th_2, R_{\theta_2} \s}(t))\Bigr) \cdot R_{\th_1} \dot{\g}_{e,\s}(t) \\
 &+  \n \psi (\g_{\th_2, R_{\theta_2} \s}(t)) \cdot \Bigl( R_{\th_1} \dot{\g}_{e,\s}(t)-  R_{\th_2} \dot{\g}_{e,\s}(t)\Bigr)  dt\Bigr|\\ 
\le&  \|\psi\|_{C^{1,\b}(S)}  \int_0^1 \Bigl( |\g_{\th_1, R_{\theta_1} \s}(t)) - \g_{\th_2, R_{\theta_2} \s}(t))|^\b 
 |\dot{\g}_{e,\s}(t)|+ |(R_{\th_1}-R_{\th_2}) \dot{\g}_{e,\s}(t)|\Bigr)  dt. 
\end{align*}
Hence Lemma~\ref{lem:curve-gamma} and the Lipschitz continuity of the map $\th \mapsto R_\th $ give rise to the estimate
\begin{align}
\bigl|\psi(\theta_1)&-\psi(R_{\theta_1} \s ) - \bigl(\psi(\theta_2)-\psi(R_{\theta_2} \s )\bigr) \bigr|\nonumber\\
\le&  C \|\psi\|_{C^{1,\b}(S)} |e-\s| \Bigl( 
\int_0^1 |\g_{\th_1, R_{\theta_1} \s}(t)) - \g_{\th_2, R_{\theta_2} \s}(t))|^\b \,dt + \|R_{\th_1}- R_{\th_2}\|\Bigr)\nonumber\\
\le&  C \|\psi\|_{C^{1,\b}(S)} |e-\s| \Bigl( 
\int_0^1 \bigl(|\th_1-\th_2| + |R_{\theta_1} \s - R_{\theta_2} \s|\bigr)^\b \,dt + |\th_1- \th_2|\Bigr)\nonumber\\
\le& C \|\psi\|_{C^{1,\b}(S)} |e-\s| |\th_1-\th_2|^\b. \label{eq:est-psi-2-S}
\end{align}
Using (\ref{eq:est-psi-1-S}) and (\ref{eq:est-psi-2-S}), applied with $\psi$ replaced by $\psi_1$ and $\psi_2$, 
we then find that 
\begin{align*}
&| \L_2(\psi_1,\psi_2, R_{\theta_1} \s,\th_1)- \L_2(\psi_1,\psi_2, R_{\theta_2} \s, \th_2)| \\
&=|(\psi_1(\th_1)-\psi_1(R_{\th_1} \s)) (\psi_2(\th_1 )-\psi_2(R_{\th_1} \s))-
(\psi_1(\th_2 )-\psi_1(R_{\th_2} \s)) (\psi_2(\th_2 )-\psi_2(R_{\th_2} \s))|\\
&\leq\Bigl|  \Bigl( \psi_1(\th_1 )-\psi_1(R_{\th_1} \s) - \bigl(\psi_1(\th_2)-\psi_1(R_{\th_2} \s)\bigr)\Bigr)
(\psi_2(\th_1)-\psi_2(R_{\th_1}\s ))\Bigr|\\
&\quad + \Bigl| (\psi_1(\th_2)-\psi_1(R_{\th_2} \s)) \Bigl( \psi_2(\th_1)-\psi_2(R_{\th_1}\s ) - \bigl(\psi_2(\th_2 )-\psi_2(R_{\th_2}\s)\bigr)\Bigr)\Bigr|\\
&\le  \|\psi_1\|_{C^{1,\b}(S)}\|\psi_2\|_{C^{1,\b}(S)} |e-\s|^2 |\th_1-\th_2|^\b,
\end{align*}
as claimed in (\ref{eq:estL2-th1-sig1}).
\end{proof}

The following result provides some estimates related to the kernel $\ov{\cK}_{\a}$ and its derivatives.
\begin{lemma}\label{lem:est-cK}
Let $r>0$, $k \in \N\cup\{0\}$. Then there exists a constant $c=c(N,\a,\b,r,k)>1$ such that  
for all $e\in S$, $\s \in S$, all $\th, \th_1,\th_2 \in  S_e $ and $\psi \in \cB(r)$, we have 
   \be \label{eq:Dk-K-s}
\|  D_{\psi}^k \ov{\cK}_{\a}({\psi},R_\th \s,\theta   )   \|\leq   
{c\left(1+ \|{\psi}\|_{C^{1,\b}(S)} \right)^{c}   }   
 \ee 
 and 
   \be  \label{eq:Dk-K-s_1s_2}
\| D_{\psi}^k \ov{\cK}_\a({\psi}, R_{\th_1} \s,\theta_1) -D_{\psi}^k \ov{\cK}_\a(\psi, R_{\th_2}  \s,\th_2)   \|\leq   
{c\left(1+ \|\psi \|_{C^{1,\b}(S)} \right)^{c}   \, |\th_1-\th_2|^\b } . 
 \ee
\end{lemma}
\begin{proof}
Throughout this proof, the letter $c$ stands for different constants greater than one and depending only on 
$N,\a,\b,k$ and $r$. We assume $k\geq 1$ for notation coherence, but the case $k=0$ is simpler and can be proved
similarly. We define 
$$
Q:  C^{1,\b}(S)\times S\times  S  \to \R,\,\, Q(\psi,\s,\th)=\frac{|\psi(\th)-\psi(\s)|^2}{|\th-\s|^2 }  + \psi(\th) \psi(\s) = \frac{\Lambda_2(\psi,\psi,\sigma,\theta)}{|\th-\s|^2 }+ \psi(\th) \psi(\s)
$$
and, for $\a>0$, 
$$
g_\a \in C^\infty(\R_+, \R), \qquad 
g_\a(x)= x^{-(N+\a)/2},
$$
so that 
\begin{align} \label{eq:def-ti-cK-ell}
\ov{\cK}_\a(\psi,R_\th\s,\th)&=g_\a\left(  Q(\psi,R_\th\s,\th)  \right).
\end{align}
Note that for $\th \in S_e$ we have 
\begin{align*}
Q(\psi,R_\th \s,\th)&=\frac{\Lambda_2(\psi,\psi,R_\th \sigma,\theta)}{|e-\s|^2 } + \psi(R_\th\s) \psi(\th)\\
D_\psi Q(\psi, R_\th \s, \th)\psi_1 &= 2 \frac{\Lambda_2(\psi,\psi_1,R_\th \sigma,\theta)}{|e-\s|^2 }
 + \psi(\th) \psi_1(R_\th \s)+\psi_1(\th) \psi(R_\th \s)\\
D_\psi^2 Q(\psi,R_\th \s,\th)[\psi_1,\psi_2] &=  2 \frac{\Lambda_2(\psi_1,\psi_2,R_\th \sigma,\theta)}{|e-\s|^2 }
 + \psi_2(\th) \psi_1(R_\th \s) + \psi_1(\th) \psi_2(R_\th \s)
\end{align*}
for $\psi, \psi_1,\psi_2 \in C^{1,\beta}(S)$. For a subset $P \subset \{1,2\}$ and $\psi_1,\psi_2 \in C^{1,\beta}(S)$, we thus have, by \eqref{eq:prod-bracket} and~(\ref{eq:estL2-th1-sig1}),  
 \begin{equation}\label{eq:DP-Q-s_1s_2}
\bigl|[D^{|P|}_{\psi} Q (\psi,R_\th\s,\cdot)[\psi_j]_{j \in P};\theta_1,\th_2]\bigr|  \leq  c (1+ \|\psi\|_{C^{1,\b}(S)}^2)  |\theta_1-\th_2|^\b \prod_{j \in P} \|\psi_j\|_{C^{1, \b}(S)}.  
\end{equation}  
Moreover, by \eqref{eq:estL2-th-sig}, 
 \begin{equation}\label{eq:DP-Q-s}
\bigl|D^{|P|}_{\psi} Q (\psi,R_\th\s,\cdot)[\psi_j]_{j \in P}\bigr|  \leq  c (1+ \|\psi\|_{C^{1,\b}(S)}^2)  
\prod_{j \in P} \|\psi_j\|_{C^{1, \b}(S)} \qquad \text{on $S_e$.}  
\end{equation}  

By \eqref{eq:Faa-de-Bruno} and recalling that $Q$  is quadratic in $\psi$,  we have 
 $$
 D_{\psi}^k \ov{\cK}_\a(\psi,R_\th\s,\th)[\psi_1,\dots,\psi_k]= \sum_{\Pi\in\scrP_k^2}  g_\a^{(\left|\Pi\right| )  } ( Q(\psi,R_\th\s,\th)   ) \prod_{P\in\Pi} D^{\left|P\right| }_{\psi} Q(\psi,R_\th\s,\th)[\psi_j]_{j \in P},
$$
where $\scrP_k^2$ denotes the set of partitions $\Pi$ of $\{1,\dots,k\}$ such that $|P| \le 2$ for every $P \in \Pi$. By \eqref{eq:prod-bracket}, we now have  
 \begin{align}
 \label{eq:Dk-K-s_1s_2-1}
& \left[ D_{\psi}^k \ov{\cK}_\a(\psi, R_\th\s,\cdot )[\psi_1,\dots,\psi_k] ;\th_1,\th_2\right]\\
 &= \sum_{\Pi\in\scrP_k^2}  \left[  g^{(\left|\Pi\right| )  }_\a  ( Q(\psi, R_\th\s, \cdot) ); \th_1,\th_2 \right] \prod_{P\in\Pi, |P|\leq 2} D^{\left|P\right| }_{\psi} Q(\psi,R_\th\s,\theta_1 )[\psi_j]_{j \in P}\nonumber\\
 &\quad+\sum_{\Pi\in\scrP_k^2}    g^{(\left|\Pi\right| )  }_\a  ( Q(\psi,R_\th\s,\theta_2)   ) \Bigl[     \prod_{P\in\Pi, |P|\leq 2} D^{\left|P\right| }_{u} Q(\psi,  R_\th\s,\cdot )[\psi_j]_{j \in P}  ;\theta_1,\th_2 \Bigr],\nonumber
 \end{align}
whereas 
$$
  g^{(k)}_\a(t)=(-1)^{k}2^{-k} \prod_{i=0}^{k-1 } (N+\a +2 i) t^{-\frac{N+\a+2k}{2}} \qquad \text{for $k \in \N $ and $t>0$.}
$$
Consequently,  by \eqref{eq:DP-Q-s_1s_2}  and since $\psi \in \cB(r)$, we have the estimates
\begin{align}
&\left| \left[ g^{({\ell})}_\a  ( Q(\psi,R_\th\s, \cdot  ) ) ;\th_1,\th_2\right] \right| \nonumber\\
&\le \left|\left[ Q(\psi,R_\th\s, \cdot    ) ;\th_1,\th_2\right]  \int_0^1  g^{({\ell}+1)}_\a  (\tau Q(\psi,R_\th\s,\th_1 )+(1-\tau)Q(\psi,R_\th\s,\th_2 ))  d\tau \right|   \nonumber\\
&\leq  c\left(1 + \|\psi\|_{C^{1,\b}(S)}^2\right)\,    \, |\th_1-\th_2|^\b   \label{eq:gQ-s1-s2}
\end{align}   
and 
\be  
\label{eq:gQ-s}
|   g^{({\ell})  }_\a  ( Q(\psi ,R_\th\s, \cdot     )) |\leq   c\left(1+ \|\psi\|_{C^{1,\b}(S)}\right)^{c}
\ee
for ${\ell}=0,\dots,k$. 
Combining  \eqref{eq:DP-Q-s_1s_2}, \eqref{eq:DP-Q-s}, \eqref{eq:Dk-K-s_1s_2-1}, \eqref{eq:gQ-s1-s2} and \eqref{eq:gQ-s}, we obtain
 \begin{align*}
\Bigl | \Bigl[ D_{\psi}^k \ov{\cK}_\a(\psi, R_\th\s, \cdot )[\psi_1,\dots,\psi_k] ;&\th_1,\th_2\Bigr] \Bigr|\\
&\leq    
c  \Bigl(1+ \|\psi\|_{C^{1,\b}(S)} \Bigr)^{c}|\th_1-\th_2|^\b \! \sum_{\Pi\in\scrP_k^2}\!     
 \prod_{P \in \Pi}   \prod_{j \in P} \|\psi_j\|_{C^{1,\b}(S)}\\  
&\le    c\Bigl(1+ \|\psi\|_{C^{1,\b}(S)} \Bigr)^{c}   \, |\th_1-\th_2|^\b     \prod_{i=1}^k 
\|\psi_i\|_{C^{1, \b}(S)}.
\end{align*}
This yields \eqref{eq:Dk-K-s_1s_2}. Furthermore we easily deduce from \eqref{eq:DP-Q-s} and \eqref{eq:gQ-s}    that 
$$
\left|  D_{\psi}^k \ov{\cK}_\a(\psi, R_\th\s ,\th    )[\psi_1,\dots,\psi_k]   \right|\leq   
 {c\left(1+ \|\psi\|_{C^{1,\b}(S)} \right)^{c}   }   \prod_{i=1}^k \|\psi_i\|_{C^{1, \b}(S)},
$$
completing the proof.
\end{proof}

We now derive estimates for functions of a specific form which will appear in formulas for the derivatives of the transformed NMC operator $\tilde h$ in (\ref{eq:def-ti-h}).
\begin{lemma}\label{lem:est-cand-deriv}  
Let ${\psi } \in \cB(r)$, with $r>0$.  Let   $k\in \N $, $e\in S$, $\Psi\in C^{1,\b}(S)$ and 
$\o,\o_1, {\psi}_1,\dots, {\psi}_k \in C^{1,\b}(S)$.   Define the functions $\cF_1, \cF_2, \cF_3: S_e \to \R$ by 
\begin{align*}
\cF_1(\th )&=\int_{{S}}\frac{ \L_1(\o,R_\th \s, \th) }{ |e-\s|^{N+\alpha}}  
\Psi(R_\th \s)\, D_\psi^k\ov{\cK}_{\a}(\psi,R_\th \s,\th)[\psi_i]_{i=1,\ldots,k}      \,  dV(\s),
\end{align*}
\begin{align*}
\cF_2(\th )&=\int_{{S}}\frac{ \L_2(\o,\o_1,R_{\th} \s,\th) }{ |e-\s|^{N+\alpha}} \Psi(R_\th \s)\, 
D_\psi^k\ov{\cK}_{\a}(\psi,R_\th\s,\th) [\psi_i]_{i=1,\ldots,k}     \,  dV(\s)
\end{align*}
and 
\begin{align}
\cF_3(\th )&=\int_{{S}}\frac{\Psi(R_\th \s)}{ |e-\s|^{N+\alpha-2}}   D_\psi^k\ov{\cK}_{\a}(\psi,R_\th \s,\th)
[\psi_i]_{i=1,\ldots,k}      \,  dV(\s) .
\end{align}
 Then, there exists a constant  $c=c(N,\a,\b,k,r )>1$ such that 
\begin{align}\label{eq:est-cF1}
\|{\cF_1} \|_{C^{\b-\a}(S_e)} & \leq  c\left(1+ \|{\psi }\|_{C^{1,\b}(S)} \right)^{c}  \|\o\|_{C^{1,\b}(S)}   
\|\Psi\|_{C^{1,\b}(S)}   \prod_{i=1}^k \|{\psi}_i\|_{C^{1, \b}(S)} ,
\end{align}
\begin{align}\label{eq:est-cF2}
\|{\cF_2} \|_{C^{\b}(S_e)} & \leq  c\left(1+ \|{\psi }\|_{C^{1,\b}(S)}  \right)^{c}   \|\o\|_{C^{1,\b}(S)} \|\o_1\|_{C^{1,\b}(S)}
\|\Psi\|_{C^{1,\b}(S)}   \prod_{i=1}^k \|{\psi}_i\|_{C^{1, \b}(S)}
\end{align}
and
\begin{align}\label{eq:est-cF3}
\|{\cF_3} \|_{C^{\b}(S_e)} & \leq  c\left(1+ \|{\psi}\|_{C^{1,\b}(S)} \right)^{c}   \|\Psi\|_{C^{1,\b}(S)}   
\prod_{i=1}^k \|{\psi}_i\|_{C^{1, \b}(S)}   .
\end{align}
\end{lemma}
\begin{proof}
Let $\th_1,\th_2 \in S_e$. We first note that, for $\s \in S$,
\begin{equation}
  \label{easy-est-G}
|\Psi(R_{\th_1} \s)- \Psi(R_{\th_2}\s)| \le C \|\Psi\|_{C^{1,\b}(S)} |R_{\th_1} \s - R_{\th_2} \s| \le C
\|\Psi\|_{C^{1,\b}(S)}   |\th_1-\th_2|.
\end{equation}
To prove estimate (\ref{eq:est-cF1}) for $\cF_1$, we recall that by \eqref{eq:estL1-th1-sig1} we have 
$$
| {\L}_1(\o ,R_{\theta_1}  \s, \theta_1)- {\L}_1(\o,R_{\theta_2}  \s, \theta_2)|\le C  \|\o\|_{C^{1,\b}(S)} \mu(\s,\theta_1,\theta_2),
$$ 
where
$$
\mu(\s,\theta_1,\theta_2):= |e - \s|  |\th_1-\th_2|^\beta + |e - \s|^\beta  |\th_1-\th_2|.
$$ 
Combining this with the fact that 
$$
| {\L}_1(\o ,R_{\theta}  \s, \theta)| \le C\|\o \|_{C^{1,\b}(S)}  \,|R_{\theta} \s- \theta|^{1+\beta} = 
C\|\o \|_{C^{1,\b}(S)}  \,|e-\s|^{1+\beta}\qquad \text{for $\s\in S, \th \in S_e$}
$$
 by \eqref{eq:estL1-th-sig}, we find that 
   \begin{align}
| {\L}_1(\o ,R_{\theta_1}  \s, \theta_1)- {\L}_1(\o,R_{\theta_2}  \s, \theta_2)|& \leq 
C  \|\o\|_{C^{1,\b}(S)} \min(|e-\s|^{1+\beta},\mu(\s,\th_1,\th_2)  ) .
\label{eq:estL1-x1-x2}
\end{align}

Using inductively \eqref{eq:prod-bracket} together with Lemma \ref{lem:est-cK}, (\ref{easy-est-G}) and \eqref{eq:estL1-x1-x2},  we get the estimate
\begin{align}
&|[\cF_1;\th_1,\th_2 ]|  \leq c\left(1+ \|\psi \|_{C^{1,\b}(S)} \right)^c \|\o \|_{C^{1,\b}(S)}\|\Psi\|_{C^{1,\b}(S)}  
 \prod_{i=1}^k \|\psi_i\|_{C^{1,\b}(S)}   \:\times \label{intermed-est-F_1}\\
&\left(    |\th_1-\th_2|^\b\!\!\int_{S}|e-\s |^{1+\b-N-\a}dV(\sigma) + \! \int_{S}\!
\frac{\min(|e-\s|^{1+\b},\mu(\s,\theta_1,\theta_2)  )}{|e-\s|^{N+\a}} dV(\sigma)\right)\nonumber
\end{align}
for all $\th_1,\th_2\in S$. Since
\begin{align*}
& \int_{S}\frac{\min(|e-\s|^{1+\b},\mu(\s,\theta_1,\theta_2)  )}{|e-\s|^{N+\a}} dV(\sigma) \\
&\leq  \int_{ |e-\s| \leq |\th_1-\th_2|}    |e-\s|^{1+\b -N-\a} dV(\sigma)
+   \int_{  |\th_1-\th_2|\leq |e-\s|} \frac{\mu(\s,\theta_1,\theta_2)}{|e-\s|^{N+\a}} dV(\sigma)\\
&\leq  \int_{ |e-\s| \leq |\th_1-\th_2|}    |e-\s|^{1+\b -N-\a} dV(\sigma)\\
&\hspace{1cm} +   \int_{  |\th_1-\th_2|\leq |e-\s|} \{
|\th_1-\th_2|^\beta |e-\s|^{1-N-\a} + |\th_1-\th_2| |e-\s|^{\b-N-\a}\} dV(\sigma)\\
& \leq C   |\th_1-\th_2|^{\b-\a},
\end{align*}
we thus deduce from (\ref{intermed-est-F_1}) that  
\begin{align*}
|[\cF_1;\th_1,\th_2 ]| &  \leq c |\th_1-\th_2|^{\b-\a} \left(1+ \|\psi \|_{C^{1,\b}(S)} \right)^c
\|\o \|_{C^{1,\b}(S)}  
\|\Psi\|_{C^{1,\b}(S)}  \prod_{i=1}^k \|\psi_i\|_{C^{1,\b}(S)}.  
\end{align*}
Since, by a similar but easier argument, we have 
$$
|\cF_1| \leq c \left(1+ \|\psi \|_{C^{1,\b}(S)} \right)^c\|\o \|_{C^{1,\b}(S)}  
\|\Psi\|_{C^{1,\b}(S)}  \prod_{i=1}^k \|\psi_i\|_{C^{1,\b}(S)} \qquad \text{on $S_e$,}
$$  
\eqref{eq:est-cF1} follows. 

Next we consider $\cF_2$. For this  we recall that 
$$
|\L_2(\o,\o_1,R_\theta \s,\th) |\leq C \|\o\|_{C^{1,\b}(S)}\|\o_1\|_{C^{1,\b}(S)}\,  |e-\s|^{2}   
$$
by \eqref{eq:estL2-th-sig}, and that   
$$
|\L_2(\o,\o_1,R_{\theta_1} \s , \theta_1)-\L_2(\o,\o_1,R_{\theta_2} \s , \theta_2) | 
\leq  C \|\o_1\|_{C^{1,\b}(S)}\|\o\|_{C^{1,\b}(S)}|e-\s|^2 |\theta_1-\theta_2|^\b
$$
by (\ref{eq:estL2-th1-sig1}). Consequently, we find that 
\begin{align} \label{F-2-first-est}
|[\cF_2;&\th_1,\th_2 ]|  \leq c\left(1+ \|\psi \|_{C^{1,\b}(S)} \right)^c\|\o \|_{C^{1,\b}(S)}\|\o_1 \|_{C^{1,\b}(S)} 
\|\Psi\|_{C^{1,\b}(S)} \prod_{i=1}^k \|\psi_i\|_{C^{1,\b}(S)}
   \nonumber\\
& \qquad \times |\theta_1-\theta_2|^\b   \!\int_{ S }|e-\s |^{2-N-\a}dV(\sigma)   \\
& \leq c\left(1+ \|\psi \|_{C^{1,\b}(S)} \right)^c\|\o \|_{C^{1,\b}(S)}\|\o_1 \|_{C^{1,\b}(S)}  \|\Psi\|_{C^{1,\b}(S)} 
|\th_1-\th_2|^{\b}  
 \prod_{i=1}^k \|\psi_i\|_{C^{1,\b}(S)}  .\nonumber
\end{align}
Moreover, by a similar but easier argument,
\be
\label{F-2-second-est} 
|\cF_2| \leq c\left(1+ \|\psi \|_{C^{1,\b}(S)} \right)^c\|\o \|_{C^{1,\b}(S)}  \|\o_1 \|_{C^{1,\b}(S)}  
 \|\Psi\|_{C^{1,\b}(S)} \prod_{i=1}^k \|\psi_i\|_{C^{1,\b}(S)}   \qquad \text{on $S_e$.}
\ee
Combining (\ref{F-2-first-est}) and \eqref{F-2-second-est}, we get \eqref{eq:est-cF2}, as claimed. We skip the proof 
of \eqref{eq:est-cF3}, which is similar but easier.
\end{proof}

The following results contains all what is needed to prove the regularity of $\ti{h}$ and hence of~$h$.
 \begin{proposition}\label{prop:smooth-ovH-0}
Let $k \in \N \cup \{0\}$, $r>0$, $\psi_1,\dots,\psi_k \in C^{1,\beta}(S)$ and let $\cM_i: S \to \R$ with $i=1,2,3$, be defined by 
\begin{align*}
\cM_i(\th )&= \int_{S }   D^k_\psi {M_i}({\psi} , \s,\th  )[\psi_1,\dots,\psi_k]\,dV(\s),
\end{align*}
where $M_1,M_2,M_3 :   \cB(r) \times S\times S    \to \R$  are given by 
\begin{align*}
M_1({\psi} ,\s,\th)&=  \frac{\L_1 ({\psi}, \s,\th)}{|\th-\s|^{N+\a}}\ov{\cK}_\a({\psi}, \s,\th)   \, {\psi}^{N-2}(\s), \\
M_2({\psi} ,\s,\th)&=  \frac{\L_2 ({\psi}, \psi, \s,\th)}{|\th-\s|^{N+\a}} \ov{\cK}_\a({\psi}, \s,\th)   \, {\psi}^{N-2}(\s), \\
{M_3}({\psi} , \s,\th)&=  \frac{1 }{|\th-\s|^{N+\a-2}} \ov{\cK}_\a({\psi}, \s,\th)   \, {\psi}^{N-1}(\s).    
\end{align*}

Then, for $i=1,2,3$,  $\cM_i  \in C^{\beta-\a}(S)$, and there exists a constant $c=c(N,\a,\b,k,r)>1$ such that 
\begin{equation}
  \label{eq:est-M1}
\|\cM_i\|_{C^{\beta-\alpha}(S)}  \leq  c(1+ \|{\psi}\|_{C^{1,\b}(S)} )^{c}   \prod_{i=1}^k \|{\psi}_i\|_{C^{1, \b}(S)},
\end{equation}
understanding that the last product equals $1$ if $k=0$.
\end{proposition}

\begin{proof}
To prove \eqref{eq:est-M1}, it suffices to fix $e \in S$ and show that
\begin{equation}
  \label{eq:M-1-e-est}
|\cM_i(\th)|  \leq  c(1+ \|{\psi}\|_{C^{1,\b}(S)} )^{c}   \prod_{i=1}^k \|{\psi}_i\|_{C^{1, \b}(S)} 
 \end{equation}
and 
\begin{equation}
  \label{eq:M-i-e-est}
|\cM_i(\th_1)-\cM_i(\th_2)|  \leq  c|\th_1-\th_2|^{\b-\alpha} (1+ \|{\psi}\|_{C^{1,\b}(S)} )^{c}   \prod_{i=1}^k \|{\psi}_i\|_{C^{1, \b}(S)} 
 \end{equation}
for $\th,\th_1,\th_2 \in S_e$, where $S_e$ is defined in (\ref{eq:def-Se}) and $c>1$ does not depend on $e$. 
For this, we define a Lipschitz continuous map $\th \mapsto R_\th$ of rotations as in Remark~\ref{example-rotations} 
corresponding to $e$, so that the inclusion in (\ref{eq:def-Se}) holds. By a change of variable, we then have  
$$
\cM_i(\th )= \int_{S }   D^k_\psi {M_i}({\psi} , R_\th \s,\th  )[\psi_1,\dots,\psi_k]\,dV(\s) \qquad 
\text{for $\th \in S_e$.}
$$

We first consider the case $i=1,$ and we note that 
$$
{M_1}({\psi} , R_\th \s,\th  ) =\frac{T_1({\psi }, \s,\th )}{|e-\s|^{N+\a}} \, {\psi}^{N-2}(R_\th \s)  
 \qquad \text{for $\th \in S_e$,}
$$
where
$$
T_1: \cB(r) \times S \times S_e   \to \R \qquad \text{defined by}\qquad T_1({\psi } , \s,\th )=  \L _1({\psi},R_\th \s,\th)\ov{\cK}_\a({\psi}, R_\th \s,\th).
$$
By \eqref{eq:Dk-T1T2}, we thus have 
 $$
 D^k_{\psi} M_1({\psi } , R_\th \s,\th )[\psi_1,\dots,\psi_k]  = \frac{1}{|e-\s|^{N+\a}}
 \sum_{\cN\in   \scrS_k}\Psi^{\cN}(R_\th \s)  D^{|\cN|}_\psi T_1({\psi } , \s,\th )[\psi_i]_{i \in \cN}   ,  
 $$
 where $\Psi^{\cN}=\psi^{N-2}$ when $k=|\cN|$, and 
 $$
\Psi^{\cN}:= \prod_{\ell=0}^{ k-|\cN|-1} (N-2-\ell)\: {\psi}^{N-2-(k-|\cN |)} \prod_{j \in \cN^c}  {\psi}_{j} 
\qquad\text{ when } k>|\cN|
 $$
(noting that $|\cN^c|=k -|\cN|$). By \eqref{eq:Dk-LT2} we have, if $|\cN| \ge 1$, 
 \begin{align*}
 D^{|\cN|}_{\psi}  T_1({\psi } ,\s,\th )[\psi_i]_{i \in \cN} &=   \L_1({\psi},R_\th \s ,\th)    D^{|\cN| }_{\psi}\ov{\cK}_\a({\psi},R_\th \s,\th)[\psi_i]_{i \in \cN} \\
&+  \sum_{j\in\cN}   \L_1({\psi }_{j},R_\th \s,\th)   D^{|\cN|-1}_{\psi }\ov{ \cK}_\a({\psi },R_\th \s,\th ) 
 [\psi_i]_{ {i \in \cN \setminus\{j\}} }.
 \end{align*}
Consequently, 
\begin{equation}
  \label{eq:decomposition-M-1-der}
 D^k_{\psi} M_1({\psi} , R_\th \s,\th)[\psi_1,\dots,\psi_k] = \sum_{\cN\in       \scrS_k} M^{\cN}_1(\s,\th),
\end{equation}
where 
\begin{align*}
M^{\cN}_1(\s,\th)=&  \frac{\Psi^{\cN}(R_\th \s)}{|e-\s|^{N+\a}} \Bigl( \L_1({\psi},R_\th \s, \th )    D^{|\cN| }_{\psi}\ov{\cK}_\a({\psi}, R_\th \s,\th) [\psi_i]_{i \in \cN }\\
&+  \sum_{j\in\cN}  \L_1({\psi}_{j}, R_\th \s,\th)    D^{|\cN|-1}_{\psi} \ov{\cK}_\a({\psi},R_\th \s,\th) 
[\psi_i]_{{i \in \cN \setminus\{j\}}}\Bigr),
\end{align*}
where the second summand does not appear if $|\cN|=0$. Clearly we also have that 
\be \label{eq:est-psi-cN}
\| \Psi^{\cN} \| _{C^{1,\b}(S)} \leq c (1+ \|{\psi}\|_{C^{1,\b}(S)})^{c}    \prod_{i \in \cN^c} \| {\psi}_{i}\| _{C^{1,\b}(S)}.
\ee

Denoting
$$
\cM_1^\cN: S_e \to \R, \qquad \cM_1^{\cN}(\th)= \int_{S} M^{\cN}_1(\s,\th)\,dV(\s),
$$
by Lemma~\ref{lem:est-cand-deriv} it follows that $\cM_1^\cN \in C^{\beta-\alpha}(S_e)$ and that 
$$
\| \cM_1^\cN \|_{C^{\beta-\alpha}(S_e)} \le c(1+ \|{\psi}\|_{C^{1,\b}(S)} )^{c}  \|\Psi^\cN\|_{C^{\b}(S)}    \prod_{i \in \cN} \|{\psi}_i\|_{C^{1, \b}(S)}   
$$
with a constant $c>1$ depending only on $N,\a,\b,r$, and $|\cN|$ (in particular, independent of $e\in S$). 
Since 
$$
\cM_1(\th ) =  \sum_{\cN\in       \scrS_k} \cM^{\cN}_1(\th) \qquad \text{for $\th \in S_e$}
$$
by  \eqref{eq:decomposition-M-1-der}, we thus obtain the estimates (\ref{eq:M-1-e-est}) and (\ref{eq:M-i-e-est}) 
for $i=1$, as desired. 

The estimate for $\cM_2$ follows the same arguments as above but using  \eqref{eq:Dk-QT2} in the place of 
\eqref{eq:Dk-LT2} while the one  for    $\cM_3$  is similar but easier. In theses two cases, we get a $C^\b(S)$
estimate for  $\cM_2$ and $\cM_3$, and in particular a $C^{\b-\a}(S)$ estimate.
\end{proof}

We are now in position to prove the regularity of  the NMC operator of  perturbed spheres, thereby completing the proof of Theorem~\ref{sec:nmc-lattice-prop-h-0}. 

\begin{theorem}
\label{theorem-smoothness-h-1}
With $\cO$ defined by \eqref{eq:calOset}, the map ${h}: \cO \to  C^{\b-\a} (S)$ defined by \eqref{eq:first-funct-eq} 
is smooth.
\end{theorem}

\begin{proof}
We fix $r>0$. By \eqref{eq:formh}, we have
\begin{equation}
  \label{eq:decomposition-h-tilde-h}
h(\vp)= -(1+\vp ) \,\ti{h}_1(1+\vp ) +\ti{h}_2(1+\vp)+\frac{1+\vp}{2} \, \ti{h}_3(1+\vp)
\end{equation}
for $\phi \in \cO$ with $\phi>r-1$ on $S$, where the maps    $\ti{h}_j :  \cB(r) \to  C^{\b-\a} (S)$ are given by 
$$
\ti{h}_j(\psi)(\th)=  \int_{{S}}M_j(\psi,\s,\th)  \,     dV(\s)
$$
for $j=1,2,3$ and the function  $ {M_j}$ is defined  in Proposition~\ref{prop:smooth-ovH-0} -- which guarantees that
$\ti{h}_j$ takes values in $C^{\b-\a} (S)$.  
Thus, it suffices to establish that $\ti{h}_j $, for $j=1,2,3$, are smooth on  $ \cB(r)$ for every $r>0$.

For this, we only need to prove that, for $k \in \N  $,  
\begin{equation}
  \label{eq:statement-H-1}
D^k \ti{h}_j({\psi} )= \int_{S}    D^k_\psi M_j(\psi,\s,\cdot)\,  dV(\s) \qquad \text{in Fr\'echet sense}   
\end{equation}
for $j=1,2,3$. Then the continuity of $D^k \ti{h}_j$ is a well known consequence of the existence of $D^{k+1}\ti{h}_j$ in Fr{\'e}chet sense. To prove (\ref{eq:statement-H-1}), we proceed by induction. For $k = 0$, the statement is true by definition. 
 Let us now assume that the statement holds true for some $k \ge 0$. Then $D^k \ti{h}_j({\psi} )$ is given by 
 \be
\label{eq:induc-hypo-ell}
 D^{k} \ti{h}_j({\psi})[\psi_1,\dots, \psi_{k}] (\th)=\int_{S }  D^{k}_{\psi} M_j({\psi},\s,\th)[\psi_1,\dots,\psi_{k}]\,   dV(\s).
\ee 
We fix ${\psi}_1,\dots,{\psi}_{k} \in C^{1,\b}(S)$. Moreover, for  ${\psi}\in \cB(r) $ and $v \in C^{1,\beta}(S)$,  we put
$$
\G({\psi},v,\th )=\int_{S }   D^{k+1}_{\psi} M_j({\psi},\s,\th)[\psi_1,\dots,\psi_k,v]\,  dV(\s).
$$
Let $\psi\in\cB(r)$ and $v\in  C^{1,\beta}(S)  $  with  $\|v\|_{C^{1,\beta}(S)}<r/2$. We have  
\begin{align*}
&D^{k}  \ti{h}_j({\psi}+v)[\psi_1,\dots,\psi_{k}](\th )-  D^{k}  \ti{h}_j({\psi} )[\psi_1,\dots,\psi_k] (\th)- \G({\psi},v , \th)
 \nonumber \\
 &=\int_{S }  \int_0^1 \!\! \left\{ D^{k+1}_{\psi}  M_j({\psi}+\rho v ,\s,\th)-D^{k+1}_{\psi} M_j({\psi},\s,\th)\right\} \!\![\psi_1,\dots,\psi_k,v] d\rho dV(\s)  \nonumber \\
 &= \int_0^1 \rho \int_0^1  {\cH}^{\rho,\tau}(\th)d\t d\rho,
\end{align*} 
where
$$
{\cH}^{\rho,\tau}(\th ):= \int_{S }   D^{k+2}_{\psi} M_j({\psi}+\t \rho  v,\s,\th)[\psi_1,\dots,\psi_k,v,v] \, dV(\s).
$$

Note that ${\psi}+\t \rho  v\in  \cB(r/2)$ for every $\t,\rho\in[0,1]$.    By Proposition~\ref{prop:smooth-ovH-0}, we have 
\begin{align*}
\|{\cH}^{\rho,\tau}\|_{C^{\b-\a}(S)}  &\leq  c(1+ \|\psi+\t \rho  v \|_{C^{1,\b}(S)} )^c \|v\|_{C^{1, \b}(\R)}^2 \prod_{i=1}^k \|u_i\|_{C^{1,\beta}(S)}\\
&\leq  c(1+ \|\psi\|_{C^{1,\b}(S)} +\| v \|_{C^{1,\b}(S)})^c  \|v\|_{C^{1, \b}(S)}^2 \prod_{i=1}^k \|\psi_i\|_{C^{1,\beta}(S)}
\end{align*}
with a constant $c>1$ independent of $\rho,\tau,\psi,\psi_1,\dots,\psi_{k}$ and $v$. Consequently, 
\begin{align*}
\| D^{k}   & \ti{h}_j({\psi}+v)[\psi_1,\dots,\psi_{k}]-  D^{k}  \ti{h}_j({\psi} )[\psi_1,\dots,\psi_k]- \G({\psi},v , \cdot)\|_{C^{\b-\a}(S)}\\
&\leq  c(1+ \|\psi\|_{C^{1,\b}(S)} +\| v \|_{C^{1,\b}(S)})^c  \|v\|_{C^{1, \b}(S)}^2 \prod_{i=1}^k \|\psi_i\|_{C^{1,\beta}(S)}.
\end{align*}
This shows that $D^{k+1} \ti{h}_j(\psi)$ exists in Fr\'echet sense, and that 
$$
 D^{k+1}  \ti{h}_j(\psi)[\psi_1,\dots,\psi_{k},v] = \G(\psi,v,\cdot) \quad \in \; C^{\beta-\alpha}(S).
$$
We conclude that (\ref{eq:statement-H-1}) holds for $k+1$ in place of $k$, and thus the proof is finished. 
\end{proof}

We close this section with an outline of the proof of Proposition~\ref{sec:nmc-lattice-prop-Gk}, 
which is concerned with the maps $G$ and $G_p$. Since the definition of these functions in (\ref{eq:1-def-Gk}) and \eqref{eq:defG} does not involve singular integrals, the proof is much easier than the proof of Theorem~\ref{sec:nmc-lattice-prop-h-0} but still somewhat lengthy if all details are carried out. In the following, we point out the main steps. 

\begin{proof}[Sketch of the proof of Proposition~\ref{sec:nmc-lattice-prop-Gk}]
We fix $c_1 \in (0,c_0)$ arbitrarily, and we note that it clearly suffices to prove the statement with $c_0$ replaced by $c_1$. For $p \in \scrL_*$, we then use polar coordinates to write $G_p$ as follows: 
\begin{align}
-\frac{1}{\alpha} G_p(\tau,\phi)(\th )&= \int_{B_\vp }\frac{1}{ |\tau(y- F_\phi(\th ))+{p}|^{N+\alpha}} dy\nonumber\\
&= \int_{S} \int_0^{1+\phi(\s)}
\frac{r^{N-1}}{|\tau(r \s - F_\phi(\th ))+{p}|^{N+\alpha}} dr\, dV(\s) \nonumber\\ 
&= \int_{S}  \int_0^{1} \frac{(1+\phi(\s))^N \rho^{N-1}}{|\tau\left\{\rho(1+\phi(\s)) \s - (1+\phi(\th)) \th\right\}+{p}|^{N+\alpha}} d\rho\, dV(\s) \nonumber \\
&= \int_{S}  \int_0^{1} \frac{(1+\phi(\s))^N \rho^{N-1}}{|\cD_p(\tau,\phi)(\rho,\s,\th)|^{N+\alpha}} d\rho\, dV(\s),
\label{polar-coord-representation}
 \end{align}
with 
$$
\cD_p(\tau,\phi)(\rho,\s,\th):= \tau\left\{\rho(1+\phi(\s)) \s - (1+\phi(\th)) \th\right\}+{p}.
$$
We point out that, for $\phi \in \cO$, $\tau \in (-\frac{c_1}{4},\frac{c_1}{4})$, $p \in \scrL_*$, $\rho \in [0,1]$ and $\s,\th \in S$ we have
\begin{align}
|\cD_p(\tau,\phi)(\rho,\s,\th)| \ge |p|- |\tau|\, \Bigl|\rho(1+\phi(\s)) \s - (1+\phi(\th)) \th \Bigr|  
&\ge |p|- \frac{c_1}{4} \, (2 + 2 \|\phi\|_{L^\infty(S)}) \nonumber\\
&\ge |p|- c_1 \label{denominator-estimate}\\
& \ge c_0 -c_1>0 \nonumber 
\end{align}
by the definition of $c_0$ in \eqref{eq:inf-cJ}. 

We now claim that $G_p$ is of class $C^k$ for all $k \in \N \cup \{0\}$, and that 
every partial derivative  $\partial^\gamma G_p$ of order  $|\gamma|=k$ with respect to $\tau$ and $\phi$ can be written as  
\begin{align*}
\partial^\gamma G_p(\tau,\phi)&[\psi_1,\dots,\psi_\ell](\th)\\
&=  \int_{S}   \int_0^{1} \sum_{\cN \subset \scrS_\ell}\, \prod_{i \in \cN}  \psi_i(\th) \prod_{j \in \cN^c}\psi_j(\s) 
\frac{P^{\gamma,\cN}(\tau,\rho,\s,\th,\phi(\s),\phi(\th), p)}{|\cD_p(\tau,\phi)(\rho,\s,\th)|^{N+\alpha+2k}} d\rho\, dV(\s) 
\end{align*}
for $\psi_1,\dots,\psi_\ell \in C^{1,\beta}(S)$, $\th \in S$. Here, $\ell \le k$ is the number of derivatives with respect to $\phi$, $\scrS_\ell$ is is the set of subsets of $\{1,\dots, \ell\} $ and  $ \cN^c= \{1,\dots, \ell \}\setminus \cN $ 
for $\cN \in  \scrS_\ell$. Moreover, the functions $P^{\gamma,\cN}$ are polynomials in all variables which are of 
degree at most $2k$ in the variable $p=(p_1,\dots,p_N)$. This representation follows easily from 
\eqref{polar-coord-representation}. We use it, together with a similar induction argument as in the proof of 
Theorem \ref{theorem-smoothness-h-1}, to show that $G_p$ is a smooth map. 
In this step we also use the embeddings $C^{1,\beta}(S) \hookrightarrow C^{1}(S) \hookrightarrow C^{\beta-\alpha}(S)$ and the estimate
\begin{align}
&\Bigl\| \sum_{\cN \subset \scrS_\ell}\, \prod_{i \in \cN}  \psi_i(\cdot) \prod_{j \in \cN^c}\psi_j(\s) 
\frac{P^{\gamma,\ell}(\tau,\rho,\s,\cdot,\phi(\s),\phi(\cdot), p)}{|\cD_p(\tau,\phi)((\rho,\s,\cdot+)|^{N+\alpha+2k}} \Bigr\|_{C^{1}(S)}\nonumber\\ 
&\le c_\gamma \bigl(1+ \|\phi\|_{C^1(S)}\bigr)^{c_\gamma} |p|^{-N-\alpha} \prod_{i=1}^\ell \|\psi_\ell\|_{C^1(S)} ,
\label{C1-norm-est}
\end{align}
which can be deduced from \eqref{denominator-estimate} since $|p|-c_1= |p|- (c_1/c_0) c_0\geq |p|(1-c_1/c_0)>0$. 
Here $c_\gamma>1$ is a constant which depends on $\gamma$ but not on $\tau,\rho,\s$ and $p$.

It thus follows that ${G}_p: (-\frac{c_1}{4},\frac{c_1}{4}) \times \cO \to C^{\beta-\alpha}(S)$ is of class $C^\infty$, and that  
\begin{equation}
\label{eq:G_p-to-G}
\Bigl\|\partial^\gamma G_p(\tau,\phi)[\psi_1,\dots,\psi_\ell]\Bigr\|_{C^{\beta-\alpha}(S)} \le
d_\gamma\bigl(1+ \|\phi\|_{C^{1,\beta}(S)}\bigr)^{d_\gamma} |p|^{-N-\alpha} \prod_{i=1}^\ell \|\psi_i \|_{C^{1,\beta}(S)}  
\end{equation}
for every partial derivative  $\partial^\gamma G_p$ of the form above, and with a constant $d_\gamma>0$ independent of $\tau$ and $p$. 
Consequently, the series $\sum_{p \in \scrL_*} \partial^\gamma {G}_p(\tau,\phi) $ is convergent in the space $C^{\beta-\alpha}(S)$, and the convergence is uniform in $\tau \in (-\frac{c_1}{4},\frac{c_1}{4})$ and $\phi \in \cO$. From this we deduce that the map 
$$
G= \sum \limits_{p \in \scrL_*} {G}_p : (-\frac{c_1}{4},\frac{c_1}{4}) \times \cO \to C^{\beta-\alpha}(S)
$$
is of class $C^\infty$, as claimed.
\end{proof}

\section{The linearized NMC operator}
\label{sec:line-nmc-oper}
 In this section, we compute a simple expression for the linearization at zero of the nonlocal mean curvature 
 operator $h$ defined in
(\ref{eq:def-h}), and we study its invertibility properties between suitably chosen function spaces. As we mentioned in the introduction, once the Fr{\'e}chet
differentiability of $h$ is proved -- as we have done --, the expression can also be derived from the results of \cite[Appendix, Proposition B.2]{Davila2014B} 
and \cite[Section 6]{FFMMM} applied to the special case of the sphere $S$. For completeness, we give a direct proof in our setting based on formula 
\eqref{eq:decomposition-h-tilde-h}.
\begin{lemma}
\label{lem:lin-NMC-op}
Let $\alpha \in (0,1)$, $\beta \in (\alpha,1)$, and let $h: \cO \subset  C^{1,\b}(S) \to C^{\b-\a}(S)$ be defined by
\eqref{eq:def-h}. Then, we have  
\begin{equation}
  \label{eq:linearization-formula}
\frac{1}{\a}Dh(0)\vp=L_\a \vp -  \l_1\vp  \qquad \text{in $C^{\b-\a}(S)\quad $ for $\phi \in C^{1,\beta}(S),$}
\end{equation}
with 
\be \label{eq:def-L-a-sec-lin}
L_\a\vp(\th )=PV \int_S\frac{\vp(\th)-\vp(\s)}{|\th -\s|^{N+\a}}\, dV(\s) \qquad \text{for $\th \in S$}
\ee
and $\l_1$ given in \eqref{eq:l_k-0} for $k =1$. In addition, $L_\a$ defines a continuous linear operator  
$C^{1,\beta}(S) \to C^{\beta-\alpha}(S)$.
\end{lemma}

Before proving the lemma, we first discuss the spectral representation of the operator $L_\alpha$ 
and the special role of $\lambda_1$. For $k \in \N$, we let $\cE_k$ be the space of 
spherical harmonics of degree $k$, and we denote by  $n_k$ its dimension. We recall that $n_0=1$ and that $\cE_0$ 
consists of constant functions, whereas $n_1 = N$ and $\cE_1$ is spanned by the coordinate functions 
$\theta \mapsto  \theta_i$ for $i=1,\dots,N$. As already mentioned in the introduction,  we have
\be \label{eq:L-a-sph-harm}
L_\a  \psi=\l_k \psi \qquad \textrm{for every $k\in\N$ and  $\psi\in \cE_k$},
\ee
where    
\be \label{eq:l_k-0-sec-lin}
\l_k= \frac{\pi^{(N-1)/2}    \Gamma((1-\a)/2)    }{  (1+\a)2^\a       \Gamma((N+\a)/2) } 
\left(  \frac{ \Gamma\left(  \frac{2k+N+\a}{2} \right)  }{ \Gamma\left(  \frac{2k+N-\a-2}{2} \right) } -   
\frac{\Gamma\left(  \frac{N+\a}{2} \right)  }{ \Gamma\left(  \frac{N-\a-2}{2} \right) }   \right),
\ee 
see e.g.  \cite[Lemma 6.26]{Samko2002}. Here  $\Gamma$ denotes  the usual Gamma function, see e.g. 
\cite[Section 8.3]{GR} for its generalization to negative non-integer real numbers.
For the proof of Lemma \ref{lem:lin-NMC-op},  it will be useful to represent the  eigenvalue $\lambda_1$ in 
a different form. For this we note that, if $\th \in S$ is fixed, then the function 
$\sigma \mapsto Y_\th(\sigma):= \sigma \cdot \th$ is a spherical harmonic of degree one, so that 
\begin{equation}
  \label{eq:def-rem-l1-0}
\lambda_1 = \l_1 Y_\th(\th)=  L_\a Y_\th(\th)= \int_{{S}}\frac{1- \s\cdot\th}{|\th-\s|^{N+\a}}\, dV(\s) =
\int_{{S}}\frac{(\s- \theta) \cdot\s }{|\th-\s|^{N+\a}}\, dV(\s).
\end{equation}
Comparing this with (\ref{eq:def-frac-curvature}), we see that 
$\frac{2}{\a} d_{N,\a} \lambda_1$ equals the NMC of the sphere $S$, as stated in \eqref{eq:NMC-sphere}. 
Moreover, since $|\theta-\sigma|^2= 2(1-\sigma\cdot\theta)$,
we may rewrite the first integral in (\ref{eq:def-rem-l1-0}) to obtain the equality  
\begin{equation}
  \label{eq:def-rem-l1}
\l_1=\frac{1}{2} \int_{{S}}\frac{1}{|\th-\s|^{N+\a-2}}\, dV(\s),
\end{equation}
which will be used in the proof of Lemma~\ref{lem:lin-NMC-op}.


\begin{proof}[Proof of Lemma~\ref{lem:lin-NMC-op}]
By \eqref{eq:decomposition-h-tilde-h}, we have 
\begin{equation}
 \label{eq:decomposition-h-tilde-h-reminder}
h(\vp)= -(1+\vp ) \,\ti{h}_1(1+\vp ) +\ti{h}_2(1+\vp)+\frac{1+\vp}{2} \, \ti{h}_3(1+\vp)
\end{equation}
with 
$$
\ti{h}_j(\psi)(\th)=  \int_{{S}}M_j(\psi,\s,\th)  \,     dV(\s),
$$
for $j=1,2,3$, where the functions  $ {M_j}$ are defined  in Proposition~\ref{prop:smooth-ovH-0}.  
Let $\phi \in C^{1,\beta}(S)$.  By (\ref{eq:statement-H-1}),   the functions 
$D \ti{h}_j(1)\phi \in C^{0,\beta-\alpha}(S)$ are given by 
\begin{equation}
  \label{eq:tilde-hj-derv}
\left(D \ti{h}_j(1)\phi\right)(\th)= \int_{{S}}\partial_{\psi} M_j(1,\s,\theta )\phi  \, dV(\s).
\end{equation}

For $\sigma, \theta \in S$, $\sigma \not = \theta$ we have 
$$
M_1(1,\s,\th )=M_2(1,\s,\th)=0, \qquad M_3(1,\s,\th )= \frac{1}{|\th-\s|^{N+\a-2}}
$$
and 
\begin{align*}
\partial_{\psi} M_1(1,\s,\th )\phi &= \frac{\Lambda_1(\phi,\sigma,\theta)}{|\th-\s|^{N+\alpha}}=
\frac{\phi(\th)-\phi(\s)- (\th-\s) \cdot \nabla \phi(\s)}{|\th-\s|^{N+\alpha}},\\
\partial_{\psi} M_2(1,\s,\th )\phi &= 0,\\
\partial_{\psi} M_3(1,\s,\th )\phi &=  \frac{\partial_{\psi} \ov{\cK}_\a(1, \s,\th)\phi + 
(N-1)\phi(\s) }{|\th-\s|^{N+\a-2}}= \frac{-\frac{N+\alpha}{2}(\phi(\th)+\phi(\s)) + (N-1)\phi(\s) }
{|\th-\s|^{N+\a-2}}\\
&=-\frac{(N+\alpha) \phi(\th)+ (2+\alpha-N)\phi(\s)}{2|\th-\s|^{N+\a-2}}, 
\end{align*}
since $\Lambda_1(1,\cdot,\cdot)\equiv 0$ and $\Lambda_2(1,\cdot,\cdot)\equiv \partial_{\psi}
\Lambda_2(1,\cdot,\cdot)\phi \equiv 0$ on $S \times S$. Combining this with 
(\ref{eq:decomposition-h-tilde-h-reminder}) and (\ref{eq:tilde-hj-derv}), 
and also using \eqref{eq:def-rem-l1-0} or \eqref{eq:def-rem-l1}, we find that 
\begin{align}
&\hspace{-5mm}  \left(D h  (0)\phi \right)(\th)\nonumber\\
&=\int_{{S}}\Bigl\{-\partial_{\psi} M_1(1,\s,\th )\phi  + \partial_{\psi} M_2(1,\s,\th )\phi + 
\frac{\phi(\th)  M_3(1,\s,\th )+  \partial_{\psi} M_3(1,\s,\th )\phi}{2}\Bigr\}dV(\s)\nonumber\\
&= - \int_{{S}}\frac{\phi(\th)-\phi(\s)- (\th-\s) \cdot \nabla \phi(\s)}{|\th-\s|^{N+\alpha}}\, dV(\s)\nonumber\\
&\hspace{1cm}-\int_{{S}}\frac{(N+\alpha-2) \phi(\th)+ (2+\alpha-N)\phi(\s)}{4|\th-\s|^{N+\a-2}}\,dV(\s)\nonumber\\
&= - \int_{{S}} \frac{\phi(\th)-\phi(\s)- (\th-\s) \cdot \nabla \phi(\s)}{|\th-\s|^{N+\alpha}}\, dV(\s)\nonumber 
- \alpha \lambda_1 \phi(\th) \\
&\hspace{1cm}-\frac{N-\alpha-2}{2}\int_{{S}}\frac{(1-\s \cdot \th)
(\phi(\th)-\phi(\s))}{|\th-\s|^{N+\a}}\,dV(\s).
\label{h-prime-phi}
 \end{align}

 Next, for $\th\in S$, we let $B_\e(\th)$ be a ball on 
$S$ centered at $\th\in S$ with radius $\e\in (0,1)$. We have 
\begin{align}
 \int_{{S}}&\frac{\phi(\th)-\phi(\s)- (\th-\s) \cdot \nabla \phi(\s)}{|\th-\s|^{N+\alpha}}\, dV(\s) \nonumber\\
&=\lim_{\eps \to 0}  \int_{{S \setminus B_\eps(\th)}} \frac{\phi(\th)-\phi(\s)+ (\s-\th) \cdot \nabla \phi(\s)}
{|\th-\s|^{N+\alpha}}\, dV(\s),\label{lebesgue-pv}
\end{align}
and, for $\eps>0$ small, integrating by parts,  
\begin{align}
\int_{{S \setminus B_\eps(\th)}} \frac{(\s-\th) \cdot \nabla \phi(\s)}{|\th-\s|^{N+\alpha}}\, dV(\s)&=
\int_{\partial B_\eps(\th)} \frac{(\s-\th) \cdot \tilde{\nu}(\sigma)(\phi(\s)-\phi(\theta))}
{|\th-\s|^{N+\alpha}}\, d\ti{V}(\s)\nonumber\\
&\hspace{1cm}+ \int_{{S \setminus B_\eps(\th)}}(\phi(\th)-\phi(\s)) \div_{\sigma} \frac{P_{\s} (\s-\th)}
{|\th-\s|^{N+\alpha}}\,dV(\sigma). \label{integration-by-parts-5-1}
\end{align}
Here and in the following, $\partial B_\eps(\th)$ denotes the relative boundary of  $B_\eps(\th)$ in $S$,   
$d\ti{V}$ denotes the $(N-2)$-dimensional Hausdorff measure on $\partial B_\eps(\th)$ and $\tilde{\nu}$ the 
unit outer normal vector field of $ \de B_\e(\th)$ on $S$. Moreover, the differential operators 
$\nabla=\nabla_\sigma$, $\div_\sigma$, and $\Delta_\sigma$ 
on the sphere $S$ are all defined with respect to the standard metric on $S$, and 
\begin{equation}
  \label{eq:S-tangent-proj}
P_{\s}(\s-\th)=  \s- \th - \left((\s- \th) \cdot \s \right) \s = \s- \th -(1-\th \cdot \s)\s = (\th \cdot \s) \s - 
\th
\end{equation}
is the orthogonal projection of $\s-\th$ onto the tangent space $T_\s S$. Since 
$$
\phi(\s)-\phi(\theta)= \nabla \phi(\th) \cdot (\s-\th)+ O(|\s-\th|^{1+\beta}) \qquad \text{as $|\s-\th|\to 0$,}
$$
and, by antisymmetry    with respect to reflection at the axis $\R \th$, 
$$
\int_{\partial B_\eps(\th)}\!\! \frac{((\s-\th)\cdot \tilde{\nu}(\sigma))(\nabla \phi(\th) \cdot (\s-\th)) }
{|\th-\s|^{N+\alpha}}\, d\ti{V}(\s)=0, 
$$
 we find that 
 \begin{align}
&\lim_{\e\to0}  \int_{\partial B_\eps(\th)} \!\!\frac{(\s-\th)\cdot \tilde{\nu}(\sigma)(\phi(\s)-\phi(\theta))}
{|\th-\s|^{N+\alpha}}\, d\ti{V}(\s)=0. \label{boundary-integral-limit}
 \end{align}
 
Now we note that, by (\ref{eq:S-tangent-proj}), we have
$$
 -\nabla_\sigma(\sigma\cdot\theta) = -\{ \theta-(\theta\cdot\sigma)\sigma\} = P_\sigma(\sigma -\theta) \qquad \text{ on } S, 
$$
and therefore 
$$
\div_{\s} P_{\s}(\s-\th) = -\Delta_{\s}(\s \cdot \th) = (N-1) (\s \cdot \th).
$$
Consequently, 
\begin{align}
 \div_{\s}&\frac{P_{\s}(\s-\th)}{|\th-\s|^{N+\alpha}} = (N-1)\frac{\s \cdot \th}{|\th-\s|^{N+\alpha}} +
 P_{\s}(\s-\th) \cdot \nabla |\th-\s|^{-N-\alpha} \nonumber\\
&= (N-1)\frac{\s \cdot \th}{|\th-\s|^{N+\alpha}} - (N+\alpha) ((\th \cdot \s)\s - \th) \cdot \frac{\s- \th}
{|\th-\s|^{N+\alpha+2}} \nonumber\\
&= (N-1)\frac{\s \cdot \th}{|\th-\s|^{N+\alpha}} - \frac{N+\alpha}{2} \frac{1+\s \cdot \th}{|\th-\s|^{N+\alpha}}=
- \frac{(2+\alpha-N)  \s \cdot \th+(N+\alpha) }{2 |\th-\s|^{N+\alpha}}.\label{div-formula}
\end{align}
Combining (\ref{lebesgue-pv}), (\ref{integration-by-parts-5-1}), (\ref{boundary-integral-limit}) and 
(\ref{div-formula}), we conclude that 
\begin{align}
 \int_{{S}}& \frac{\phi(\th)-\phi(\s)- (\th-\s) \cdot \nabla \phi(\s)}{|\th-\s|^{N+\alpha}}\, dV(\s)\\
&= \frac{1}{2}\lim_{\eps \to 0}  \int_{{S \setminus B_\eps(\th)}} \frac{\left\{(2-N-\alpha) +(N-2-\alpha) 
\s \cdot \th \right\} (\phi(\th)-\phi(\s))}{|\th-\s|^{N+\alpha}}\, dV(\s)\
\end{align}
and thus, by \eqref{h-prime-phi}, 
\begin{align*}
 (Dh(0)\phi)(\th)=&  \frac{1}{2}\lim_{\eps \to 0} \int_{{S \setminus B_\eps(\th)}} \frac{\left\{(\alpha + N- 2) 
 + (2+\alpha-N) \s \cdot \th\right\} (\phi(\th)-\phi(\s))}{|\th-\s|^{N+\alpha}}\, dV(\s)  \\
&-\frac{N-\alpha-2}{2} \lim_{\eps \to 0} \int_{{S \setminus B_\eps(\th)}}\frac{(1-\s \cdot \th)(\phi(\th)-
\phi(\s))}{|\th-\s|^{N+\a}}\,dV(\s)- \alpha \lambda_1 \phi(\th) \\
=& \alpha \lim_{\eps \to 0}  \int_{{S \setminus B_\eps(\th)}} \frac{\phi(\th)-\phi(\s)}{|\th-\s|^{N+\alpha}}\, 
dV(\s) - \alpha \lambda_1 \phi(\th), 
\end{align*}
as claimed.

The last statement -- that $L_\alpha$ is continuous between $C^{1,\beta}(S) \to C^{\beta-\alpha}(S)$ --
is a direct consequence of our nonlinear result of Theorem \ref{sec:nmc-lattice-prop-Gk}.
\end{proof}

Next, we wish to study invertibility properties of the linearized operator $Dh(0)$ between suitably chosen 
function spaces. 
The following theorem is the main result of this section. 
\begin{theorem}\label{prop:Dh0-invert}
Let $\a\in(0,1)$, $\b\in (\a,1)$, and let the subspaces $X \subset C^{1,\beta} (S)$,  $Y \subset 
C^{\beta-\a} (S)$ be defined by \eqref{eq:def-space-X} and \eqref{eq:def-space-Y}.
Then, the restriction to $X$ of the linearized NMC operator \mbox{$Dh(0) : X\to Y$} is an isomorphism onto $Y$.
\end{theorem}

The remainder of this section is devoted to the proof of this theorem.

In the following, for $k \in \N \cup \{0\}$, we let $P_k: L^2(S) \to L^2(S)$ denote the 
$\langle \cdot,\cdot \rangle_{L^2}$-orthogonal projections on $\cE_k$ -- the space of 
spherical harmonics of degree $k$. For $\rho \ge 0$, we then define the Sobolev space 
\begin{equation}
  \label{eq:def-hpe}
  H^{\rho}(S) := \{u \in L^2(S)\::\: \sum_{k= 0}^\infty  (1+k^2)^{\rho} \| P_k u\|^2_{L^2(S)} < \infty\},
\end{equation}
which is a Hilbert space with the scalar product 
\begin{equation}
  \label{eq:scp-hj}
(u,v) \mapsto   \sum_{k=0}^\infty (1+k^2)^{\rho} 
\langle P_k u, P_k v \rangle_{L^2}  \qquad \text{for $u,v \in H^\rho(S)$.}
\end{equation}

We need the following result on the mapping properties of the operator $L_\alpha$ with regard to the scale 
of Sobolev spaces $H^\rho(S)$.
\begin{lemma}
\label{sec:spectr-prop-line-sobolev} Let $\alpha \in (0,1)$ and $\beta \in (\alpha,1)$.
\begin{enumerate}
\item For given $\rho \ge 0$, the map 
\begin{equation}
  \label{eq:def-tilde-l-alpha}
v \mapsto \widetilde L_\alpha v := \sum_{k =0}^\infty \lambda_k P_k v= \sum_{k =0}^\infty L_\alpha P_k v
\end{equation}
defines a continuous linear operator $\widetilde L_\alpha: H^{\rho+1+\alpha}(S) \to H^{\rho}(S)$.\\
Moreover, $\widetilde L_\alpha+ \id: H^{\rho+1+\alpha}(S) \to H^{\rho}(S)$ is an isomorphism. 
\item We have $C^{1,\beta}(S) \subset H^{1+\alpha}(S)$ and
\begin{equation}
  \label{eq:l-alpha-tilde-l-alpha}
L_\alpha \psi = \widetilde L_\alpha \psi \qquad \text{in $L^2(S)$}\qquad \text{for $\psi \in C^{1,\beta}(S)$} 
\end{equation}
with $\tilde L_\alpha: H^{1+\alpha}(S) \to L^2(S)$ given in (\ref{eq:def-tilde-l-alpha}).
\item The operator $L_\alpha + \id$ restricts to a bijective map $C^\infty(S) \to C^\infty(S)$.
\end{enumerate}
\end{lemma}

\begin{proof}
(i) Since 
$$
\lim_{\tau \to +\infty}\frac{\Gamma(\tau+\varrho)}{\Gamma(\tau) \tau^\varrho}=1 \qquad \textrm{  for all } 
\varrho \in\R
$$ 
 (see e.g. \cite[Page 15, Problem 7]{Lebedev}), we deduce from \eqref{eq:l_k-0-sec-lin} that  
\be \label{eq:asymp-l_k}
\lim_{k\to +\infty } \frac{\l_k}{ k^{1+\a}} = \frac{\pi^{(N-1)/2}    \Gamma((1-\a)/2)    }{  (1+\a)2^\a      
\Gamma((N+\a)/2) } \quad \in \quad (0,\infty). 
\ee  
Using this and the fact that $\lambda_k \ge 0$  for all $k \in \N \cup \{0\}$, we infer that $\widetilde L_\alpha$, 
as defined in (\ref{eq:def-tilde-l-alpha}), is a well defined continuous linear operator 
$H^{\rho+1+\alpha}(S) \to H^{\rho}(S)$, and that
$\widetilde L_\alpha+ \id$ is an isomorphism.

(ii) In the following, we let $C_1,C_2, \dots$ denote positive constants depending only on $N,\alpha$ and $\beta$. 
For $\psi \in C^{1,\beta}(S)$, by Lemma~\ref{lem:lin-NMC-op},  we have 
$$
  \|L_\alpha \psi\|_{L^2(S)}\leq C_1 \|L_\alpha \psi\|_{C^{\b-\a}(S)}\leq  C_2 \|\psi \|_{ C^{1,\beta}(S) }
$$
and thus  
\begin{equation}
\label{eq:C-infty}
\|L_\alpha \psi\|_{L^2(S)}+ \|\psi\|_{L^2(S)} \le C_3\|\psi \|_{ C^{1,\beta}(S) }. 
\end{equation}
Next we remark that, as a consequence of the spectral representation of the Laplace-Beltrami 
operator on $S$, we have   
\begin{equation}
  \label{C-infty-H-infty}
C^\infty(S)= \bigcap \limits_{\rho \in \N}H^\rho(S)= \bigcap \limits_{\rho \ge 0}H^\rho(S). 
\end{equation}
Moreover, for $\psi \in C^\infty(S)$ the series $\sum_{k =0}^\infty \lambda_k P_k \psi$ converges in 
$L^2(S)$ and the series $\sum_{k=0}^\infty P_k \psi$ converges in $C^m(S)$ for every $m \in \N$. 
From this and (\ref{eq:C-infty}) we deduce that 
$$
L_\alpha \psi = \lim_{\ell \to \infty} L_\alpha \sum_{k=0}^\ell  P_k \psi = \lim_{\ell \to \infty} 
\sum_{k=0}^\ell L_\alpha P_k \psi = \widetilde L_\alpha \psi \quad \text{for $\psi \in C^\infty(S)$.}
$$
Combining this with (i) and (\ref{eq:C-infty}), we find that 
\begin{equation}
  \label{eq:C-infty-1}
\|\psi\|_{H^{1+\alpha}(S)}\le  C_4 \Bigl(\|L_\alpha \psi\|_{L^2(S)}+ \|\psi\|_{L^2(S)}\Bigr)\le C_5 
\|\psi \|_{ C^{1,\beta}(S) }\qquad \text{for $\psi \in C^\infty(S)$.}
\end{equation}
Next, let $\psi \in C^{1,\beta}(S)$, and let $\psi_n \in C^\infty(S)$, $n \in \N$ satisfy $\psi_n \to \psi$ 
in $C^{1,\beta}(S)$.  Then (\ref{eq:C-infty-1}) implies that $(\psi_n)_{n\in\N}$ is a Cauchy sequence in 
$H^{1+\alpha}(S)$, and by completeness this forces $\psi \in H^{1+\alpha}(S)$. Moreover, by passing to the 
limit, we deduce that $\psi_n \to \psi$ in $H^{1+\alpha}(S)$, which implies that 
$$
\widetilde L_\alpha  \psi= \lim_{n \to \infty} \widetilde L_\alpha \psi_n  = \lim_{n \to \infty} L_\alpha 
\psi_n \qquad \text{in $L^2(S)$.}
$$
Since moreover $L_\alpha \psi = \lim \limits_{n \to \infty} L_\alpha \psi_n$ in $C^{\beta-\alpha}(S)$ 
by Lemma~\ref{lem:lin-NMC-op}, we obtain  \eqref{eq:l-alpha-tilde-l-alpha}.

(iii) This follows immediately from (i), (ii) and (\ref{C-infty-H-infty}).
\end{proof}

The following lemma provides an analogue for $L_\a+id$ of the classical interior H\"older regularity estimate 
for the classical fractional Laplacian. In the proof, we will apply a series of 
changes of variables to reduce our problem to one where regularity for the classical fractional
Laplacian can be applied.

\begin{lemma}
\label{sec:line-nmc-oper-1}
Let $\alpha \in (0,1)$, $\beta \in (\alpha,1)$. Then there exists a constant $C=C(N,\alpha,\beta)>0$ such that 
\begin{equation}
  \label{eq:c-infty-est}
\|\psi\|_{C^{1,\beta}(S)} \le C \|L_\alpha \psi + \psi\|_{C^{\beta-\alpha}(S)} \qquad \text{for all 
$\psi \in C^\infty(S).$}  
\end{equation}
\end{lemma}

To establish this lemma we need a standard interpolation estimate. 
We include a simple proof for the convenience of the reader.

\begin{lemma}
\label{standard-interpolation}
 Let $\beta \in (0,1)$. Then for every $\eps>0$ there exists $K=K(\eps,N,\beta)>0$ such that 
 \begin{equation}
  \label{eq:interpol-0}
\|\psi \|_{ C^{\b}(S) } \le \eps \|\psi \|_{ C^{1}(S) }+ K \|\psi\|_{L^2(S)} \qquad \text{for every $\psi \in C^1(S)$.}
 \end{equation}
 \end{lemma}
 
 \begin{proof}
In the following, we let $C_1,C_2,\dots,$ denote positive constants which only depend on~$N$. 
As a consequence of \eqref{eq:gamma-dot-bound}, we have 
\begin{equation}
  \label{eq:interpol-1}
|\psi(\th)-\psi(\s)| \le C_1 \|\psi \|_{ C^{1}(S) } |\th-\s| \qquad \text{for $\th,\s \in S$,}  
\end{equation}
and thus 
$$
\frac{|\psi(\th)-\psi(\s)|}{|\th-\s|^\beta}  \le C_1|\th- \s|^{1-\b} \|\psi \|_{ C^{1}(S) } \le C_1 \delta^{1-\b}
\|\psi \|_{ C^{1}(S) } \qquad \text{for $\th,\s \in S$ with $|\th-\s| \le \delta$.}  
$$
Moreover, 
$$
\frac{|\psi(\th)-\psi(\s)|}{|\th-\s|^\beta}  \le \frac{2}{\delta^\beta}\|\psi\|_{L^\infty(S)} \qquad \text{for 
$\th,\s \in S$ with $|\th-\s| \ge \delta$.}  
$$
Combining these inequalities, we find that    
\begin{equation}
  \label{eq:interpol-4}
\frac{|\psi(\th)-\psi(\s)|}{|\th-\s|^\beta}  \le C_1 \delta^{1-\b}\|\psi \|_{ C^{1}(S) } + \frac{2}{\delta^\beta}
\|\psi\|_{L^\infty(S)} \qquad \text{for $\th,\s \in S$.}  
\end{equation}

Next, for $0<r<1$ and $\th \in S$, we let $d_r$ denote the $(N-1)$-dimensional volume of the ball  $B_r(\th) $ 
on $S$, which clearly does not depend on $\th$. By (\ref{eq:interpol-1}) we then have 
\begin{align*}
\Bigl|\psi(\th)-\frac{1}{d_r}\int_{B_r(\th)}\psi(\s) & \,d V(\s)\Bigr| \le \frac{1}{d_r}\int_{B_r(\th)}| 
\psi(\th)-\psi(\s)|\,d V(\s)\\
& \le  \frac{C_1}{d_r}\|\psi \|_{ C^{1}(S) } \int_{B_r(\th)}|\th-\s|\,d V(\s) \le  r C_2 \|\psi\|_{C^1(S)}  
\qquad \text{for all $\th \in S$},  
\end{align*}
and thus 
\begin{align}
\|\psi\|_{L^\infty(S)} &\le r C_2 \|\psi\|_{C^1(S)} + \max_{\th \in S} \Bigl| \frac{1}{d_r}\int_{B_r(\th)}
\psi(\s)\,d V(\s)\Bigr| \nonumber\\
&\le r C_2 \|\psi\|_{C^1(S)} + \frac{\|\psi\|_{L^1(S)}}{d_r} \le r C_2 \|\psi\|_{C^1(S)} + \frac{|S|^{1/2}}{d_r} 
\|\psi\|_{L^2(S)}.  \label{eq:interpol-5}
\end{align}
Combining (\ref{eq:interpol-4}) and (\ref{eq:interpol-5}), we find that 
\begin{align*}
\|\psi \|_{ C^{\b}(S) }&= \sup_{\stackrel{\th, \s \in S}{\th \not = \s}}\frac{|\psi(\th)-\psi(\s)|}{|\th-\s|^\beta} 
+    \|\psi\|_{L^\infty(S)} \le C_1 \delta^{1-\b}\|\psi \|_{ C^{1}(S) } + \Bigl(\frac{2}{\delta^\beta}+1\Bigr) 
\|\psi\|_{L^\infty(S)}\\
&\le \Bigl\{C_1 \delta^{1-\b} + r C_2\Bigl(\frac{2}{\delta^\beta}+1\Bigr)\Bigr\}  \|\psi\|_{C^1(S)}
+ \frac{|S|^{1/2}}{d_r}\Bigl(\frac{2}{\delta^\beta}+1\Bigr) \|\psi\|_{L^2(S)}.
\end{align*}

For a given $\eps>0$, we may now choose $\delta>0$ such that $C_1 \delta^{1-\b} \le \frac{\eps}{2}$ and then 
$r>0$ such that $r C_2\Bigl(\frac{2}{\delta^\beta}+1\Bigr)\le \frac{\eps}{2}$. Then (\ref{eq:interpol-0}) 
follows with $K = \frac{|S|^{1/2}}{d_r}\Bigl(\frac{2}{\delta^\beta}+1\Bigr)$. 
 \end{proof}
 
We can now give the 
 
\begin{proof}[Proof of Lemma \ref{sec:line-nmc-oper-1}]
In the following, the latter $C$ stands for positive constants which may change from line to line but only 
depend on $N, \alpha$ and $\beta$. Let $\psi \in C^\infty(S)$ and $g:= L_\alpha \psi + \psi \in C^\infty(S)$. 
We define $u \in C^\infty(\R^N \setminus 0) \cap L^\infty(\R^N)$ by 
$u(x)=\psi(x/|x|)$ for $x \not = 0$.   For $x\in \R^N\setminus\{0\}$,  a change of variable in polar coordinates 
gives, with $r=|x|$ and $\th = \frac{x}{|x|}$, 
\begin{align*}
(-\D)^{(1+\a)/2} u(x)&=C \int_{\R^N}\frac{u(x)-u(y)}{|x-y|^{N+1+\a}}\,dy\\
&= C \int_0^\infty\int_{{S}}\frac{\psi(\th)-\psi(\s)}{( (\rho-r)^2 + r\rho|\th-\s|^2   )^{(N+1+\a)/2}}\rho^{N-1}\, 
dV(\s) d\rho.
\end{align*}
We make the change of variable $t=\frac{ \rho-r}{|\th-\s |}$ to get   
\begin{equation}
\label{eq:first-exprs-Ds-polar}
(-\D)^{(1+\a)/2} u(x)= C \int_{{S}}\frac{\psi(\th)-\psi(\s)}{      |\th-\s|^{N+\a} } \int_0^\infty
\frac{(t|\th-\s|+r )^{N-1} }{( t^2 + r(t|\th-\s|+r )  )^{(N+1+\a)/2}}\, dt dV(\s).
\end{equation}
To further simplify this expression, we define the function 
$$
Q: [0,\infty) \times (0,\infty) \to \R, \qquad Q(a,b):=  C \int_0^\infty\frac{(ta+b)^{N-1}}
{( t^2 + b(ta+b )  )^{(N+1+\a)/2}}\, dt.
$$
Using the change of variable $s=\frac{t}{b}$, we see that 
$$
Q(a,b)= b^{-1-\a}C   \int_0^\infty\frac{(s a+1)^{N-1}}{(s^2 + sa+1   )^{(N+1+\a)/2}}\, ds.
$$
From this we see that $Q \in C^\infty([0,\infty) \times (0,\infty))$. Moreover,  from \eqref{eq:first-exprs-Ds-polar} 
we get that 
\begin{align*}
(-\D)^{(1+\a)/2} u(x) &=  \int_{{S}}\frac{\psi(\th)-\psi(\s)}{      |\th-\s|^{N+\a} }  Q(|\th-\s|),r)        
dV(\s) \\
&= Q(0,r)\, L_\a \psi  (\th)   +  \int_{{S}}\frac{\psi(\th)-\psi(\s)}{      |\th-\s|^{N+\a} }\Bigl(Q(|\th-\s|),r)-  
Q(0,r)  \Bigr)       dV(\s)\\
&= Q(0,r)\, (g(\th)-\psi(\th))  +  \int_{{S}}\frac{\psi(\th)-\psi(\s)}{      |\th-\s|^{N+\a-1} } \int_0^1 \de_aQ(\t |\th-\s|,r)
\,d\t       dV(\s) .
\end{align*}

Next, we define 
$$
Q_{g,\psi}(x):= Q(0,|x|) ( g(x/|x|)-\psi(x/|x|))
$$
and
$$
G_\psi(x)=   \int_{{S}}\frac{\psi(\th)-\psi(\s)}{      |\th-\s|^{N+\a-1} } \int_0^1 \de_a Q(\t |\th-\s|,r)\,d\t      
dV(\s) \quad \text{for $x\in\R^N \setminus \{0\}$, $r=|x|$, $\th=\frac{x}{|x|}$,}
$$ 
so that 
\be \label{eq:Dsu=g--G}
 (-\D)^{(1+\a)/2} u(x) =  Q_{g,\psi}(x)+ G_\psi(x) \qquad \text{for $x \in \R^N \setminus \{0\}$.}
\ee
We also put $A:= \{x \in \R^N\::\: \frac{1}{2}\le |x| \le 2\}$. 
We have $Q_{g,\psi} \in C^{{\beta-\alpha}}(A)$ and
\begin{equation}
  \label{eq:r-1-g}
\|Q_{g,\psi}\|_{ C^{ \b-\a}(A)} \le C  \left( \|g\|_{ C^{\b-\a}(S) } + \Vert\psi\Vert_{C^{\b-\a}(S)} \right) . 
\end{equation}

Next we show that 
\begin{equation}
  \label{eq:G_A}
\text{${G_\psi}\in  C^{ \b-\a}(A)\quad $ with $\quad  \|G_\psi \|_{ C^{ \b-\a}(A)} \leq C   \|\psi \|_{ C^{\b}(S) }$.}
\end{equation}
To this end, we write 
\begin{equation*}
G_\psi(x)=\int_{{S}}\frac{\psi(\th)-\psi(\s)}{      |\th-\s|^{N+\a-1} } \int_0^1 \de_aQ(\t |\th-\s|,r)\,d\t      
dV(\s)=   \de_aQ(0,r)L_{\alpha-1} \psi(\th)+ \widetilde G_\psi(x)
\end{equation*}
with 
$$
L_{\alpha-1} \psi(\th):=\int_{{S}}\frac{\psi(\th)-\psi(\s)}{      |\th-\s|^{N+\a-1} } \,      dV(\s) 
$$
and 
$$
\widetilde G_\psi(x):= \int_{{S}}\frac{\psi(\th)-\psi(\s)}{      |\th-\s|^{N+\a-2} } \int_0^1 \int_0^1\t \de_a^2Q 
(\l\t |\th-\s|,r)\,d\t  d\l      dV(\s) 
$$
for $\psi \in S$. We will show that 
\begin{equation}
  \label{eq:L-alpha-1-1}
\|L_{\alpha-1} \psi \|_{ C^{ \b-\a}(S)} \leq C   \|\psi \|_{ C^{\b}(S) } 
\end{equation}
and that 
\begin{equation}
  \label{eq:G1}
\|\widetilde G_\psi \|_{ C^{ \b-\a}(A)} \leq C   \|\psi \|_{ C^{\b}(S) } .
\end{equation}
From this \eqref{eq:G_A} follows, since $\partial_qQ(0,r)$ is equal to a constant times $r^{-1-\a}$.

To show \eqref{eq:L-alpha-1-1} and \eqref{eq:G1}, it suffices to fix $e \in S$ arbitrarily, and prove that 
\begin{equation}
  \label{eq:L-alpha-1-1-e}
\|L_{\alpha-1} \psi \|_{ C^{ \b-\a}(S_e)} \leq C   \|\psi \|_{ C^{\b}(S) } \qquad \text{with 
$S_e:= \{\th \in S \::\: \th \cdot e \ge 0\}$} 
\end{equation}
and that 
\begin{equation}
  \label{eq:G1-e}
\|\widetilde G_\psi \|_{ C^{ \b-\a}(A_e)} \leq C   \|\psi \|_{ C^{\b}(S) }\qquad \text{with 
$A_e:= \{x \in A\::\: x \cdot e \ge 0\}$.} 
\end{equation}
To show these estimates, we consider  again a Lipschitz continuous map of rotations $S \mapsto SO(N)$, 
$\th \mapsto R_\theta$ with the property that \eqref{eq:def-Se} \eqref{eq:Rota-isom} holds, so that by a 
change of variable we have 
\begin{equation}
  \label{eq:L-alpha-1-2-e}
L_{\alpha-1} \psi(\th)=\int_{{S}}\frac{\psi(\th)-\psi(R_\th \s)}{      |e-\s|^{N+\a-1} } \,      dV(\s) 
\qquad \text{for $\th \in S_e$.}
\end{equation}
Since 
\begin{align*}
\Bigl | \{\psi(\th_1)-\psi(R_{\th_1} \s)\}&- \bigl\{\psi(\th_2)-\psi(R_{\th_2} \s)\bigr\}\Bigr|\\
&\le C \|\psi\|_{C^{\beta}(S)}
\min \{|\th_1-\th_2|^\beta, |\th_1-R_{\th_1}\s|^\beta+|\th_2-R_{\th_2}\s|^\beta\}\\
&\le C \|\psi\|_{C^{\beta}(S)} 
\min \{|\th_1-\th_2|^\b, |e-\s|^\b\} \qquad \text{for $\th_1,\th_2 \in S_e$ and $\s \in S$,}
\end{align*}
we may deduce by a similar integration as in the proof of Lemma~\ref{lem:est-cand-deriv} that 
(\ref{eq:L-alpha-1-1-e}) holds. 

To prove (\ref{eq:G1-e}), we write, again by a change of variable,
\begin{equation}
  \label{eq:G1-2}
\widetilde G_\psi(x)= \int_{{S}}\frac{\psi(\th)-\psi(R_\th \s)}{      |e-\s|^{N+\a-2} } q(r,\s)   dV(\s) 
\qquad \text{for $x = r \th \in A_e$,}
\end{equation}
with 
$$
q \in C^\infty([{1}/{2},2]\times S), \qquad q(r,\s) :=  \int_0^1 \int_0^1\t \de_a^2Q(\l\t |e-\s|,r)\,d\t  d\l.
$$
Since the function $\s \mapsto \frac{1}{|e-\s|^{N+\a-2}}$ is integrable over $S$, it is then easy to deduce that 
$$
\|\widetilde G_\psi \|_{ C^{ \b-\a}(A_e)} \leq C \|\widetilde G_\psi \|_{ C^{ \b}(A_e)}\le C \|\psi \|_{ C^{\b}(S) }.
$$
Hence (\ref{eq:G1-e}) holds as well.

In view of \eqref{eq:Dsu=g--G}, (\ref{eq:r-1-g}) and (\ref{eq:G1}), we can thus apply 
local Hölder regularity estimates for the fractional Laplacian. A version suited for our situation is that
of Theorem~1.3 of \cite{DSV}, which we apply rescaled and with $k=0$ and $\gamma=\beta-\a$.
We conclude that  $u\in C_{\loc}^{1,\beta}(A)$, and that 
\begin{align*}
  \|\psi \|_{ C^{1,\beta}(S) } \leq C  \left( \|g\|_{ C^{\b-\a}(S) } + \Vert\psi\Vert_{C^{\b}(S)} 
  + \Vert u\Vert_{L^{\infty}(\R^N)}   \right) \leq C  \left( \|g\|_{ C^{\b-\a}(S) } + \Vert\psi\Vert_{C^{\b}(S)} 
  \right).
\end{align*}
We finally combine this with Lemma~\ref{standard-interpolation}, applied with $\eps= \frac{1}{2C}$. We also
apply the isomorphism statement in Lemma~\ref{sec:spectr-prop-line-sobolev}(i) to get that
$ \|\psi \|_{L^{2}(S) } \leq  \Vert\psi\Vert_{H^{1+\a}(S)}\leq C \|g\|_{ L^{2}(S) }$.
We conclude the estimate 
 $$
 \|\psi \|_{ C^{1,\beta}(S) } \leq C \left(  \|g \|_{ C^{\beta-\alpha}(S) } +\|\psi\|_{L^2(S)}    \right) 
 \le C \left(      \|g \|_{ C^{\beta-\alpha}(S) }+ \|g\|_{L^2(S)}\right) \le
C \|g\|_{C^{\beta-\alpha}(S)}.
 $$
Thus (\ref{eq:c-infty-est}) holds.
\end{proof}

By a density argument, we may now deduce the following proposition from Lemma~\ref{sec:spectr-prop-line-sobolev}(iii) 
and Lemma~\ref{sec:line-nmc-oper-1}.

\begin{proposition}
\label{sec:spectr-prop-line-hoelder} 
Let $\alpha \in (0,1)$, $\beta \in (\alpha,1)$, and let the operator $L_\alpha$ be given by \eqref{eq:def-L-a-sec-lin}. Then the operator 
$$
L_\alpha + \id: C^{1,\beta}(S) \to C^{\beta-\alpha}(S)
$$ 
is an isomorphism. 
\end{proposition}

\begin{proof}
We first show that $\textrm{Ker}\,(L_\alpha + \id)=  \{0\}$. Let $\psi \in C^{1,\beta}(S)$ with  $L_\alpha \psi 
+ \psi = 0$ in $C^{\beta-\alpha}(S)$. By Lemma~\ref{sec:spectr-prop-line-sobolev}(ii) we then have 
$\widetilde L_\alpha \psi+\psi  =  L_\alpha \psi + \psi = 0$ in $L^2(S)$, and thus $\psi = 0$ by 
Lemma~\ref{sec:spectr-prop-line-sobolev}(i).

Next we show that $L_\alpha + \id$ is onto. For this, we   let $g \in C^{\beta-\alpha}(S)$ and let $g_n \in 
C^\infty(S)$ be a sequence such that $g_n \to g$ in $C^{\beta-\alpha}(S)$. By 
Lemma~\ref{sec:spectr-prop-line-sobolev}(iii), there exists $\psi_n \in C^\infty(S)$, $n \in \N$, with 
$L_\alpha \psi_n + \psi_n = g_n$. Moreover, by Lemma~\ref{sec:line-nmc-oper-1} we have
$$
\|\psi_n-\psi_m\|_{C^{1,\beta}(S)} \le C \|g_n-g_m\|_{C^{\beta-\alpha}(S)} \qquad \text{for $n, m \in \N$.}  
$$
Consequently, the sequence $(\psi_n)_n$ is a Cauchy sequence in $C^{1,\beta}(S)$, so that $\psi_n \to \psi$ 
in $C^{1,\beta}(S)$. 
Moreover, by continuity we have 
$$
L_\alpha \psi + \psi = \lim_{n \to \infty}(L_\alpha \psi_n + \psi_n)= \lim_{n \to \infty}g_n = g \qquad 
\text{in $C^{\beta-\alpha}(S)$.}
$$
It follows that the continuous linear map $L_\alpha + \id: C^{1,\beta}(S) \to C^{\beta-\alpha}(S)$ is bijective, 
and thus it is an isomorphism by the open mapping theorem. 
\end{proof}

\begin{proof}[Proof of Theorem~\ref{prop:Dh0-invert} (completed)] 
By Proposition~\ref{sec:spectr-prop-line-hoelder}, we have that 
\be \label{eq:L-1}
  L_\a + \id: X\to  Y \qquad \text{is an isomorphism.}  
\ee
Since the inclusion $\id: X \to Y$ is a compact operator, it follows that 
$$
\cL:= L_\a -\l_1 \id: X \to  Y \qquad \text{is a Fredholm operator of index zero.}  
$$
Moreover, by Lemma~\ref{sec:spectr-prop-line-sobolev}(ii) and \eqref{eq:L-a-sph-harm}, we have that  
$\textrm{Ker}\,\cL=X\cap \cE_1=  \{0\}$ since $X$ is made of even functions and $\cE_1$ contains only odd ones. 
Consequently, $\cL$ is an isomorphism, and thus  
$  Dh(0) = \a \cL : X\to Y$ is an isomorphism as well. 
\end{proof}

\bigskip\bigskip
\centerline{{\sc Acknowledgments}} 
\smallskip

The first author would like to thank Joan Sol\`a-Morales for many interesting discussions in the subject 
of this paper.

\end{document}